\documentclass[12pt]{article}

\usepackage[T1]{fontenc}
\usepackage{mathptmx}

\usepackage{setspace}
\usepackage[left=1in,right=1in,top=1in,bottom=1in]{geometry}
\usepackage{amsfonts}
\usepackage{amsmath}
\usepackage{amssymb}
\usepackage{amsthm}
\usepackage{mathrsfs}
\usepackage{mathtools}
\usepackage{slashed}
\usepackage{enumitem}
\usepackage{hyperref}
\usepackage{cleveref}
\usepackage{tikz,tikz-cd}
\usepackage{extpfeil}
\usepackage{hhline}
\usepackage{float}
\usepackage{thmtools}
\usepackage{thm-restate}
\usepackage{xcolor}
\hypersetup{
    colorlinks=true,
    linkcolor=blue,
    filecolor=magenta,      
    urlcolor=cyan,
    citecolor=blue,
    }

\newcommand\KK{\mathbb{K}}
\newcommand\LL{\mathbb{L}}
\newcommand\QQ{\mathbb{Q}}
\newcommand\CC{\mathbb{C}}
\newcommand\HH{\mathbb{H}}
\newcommand\PP{\mathbb{P}}
\newcommand\RR{\mathbb{R}}
\newcommand\NN{\mathbb{N}}
\newcommand\ZZ{\mathbb{Z}}
\numberwithin{equation}{section}

\theoremstyle{plain}
\newtheorem{theorem}{Theorem}[section]
\newtheorem{thmx}{Theorem}

\newtheorem{proposition}[theorem]{Proposition}
\newtheorem{corollary}[theorem]{Corollary}
\newtheorem{lemma}[theorem]{Lemma}

\theoremstyle{definition}
\newtheorem{definition}[theorem]{Definition}
\newtheorem{example}[theorem]{Example}

\theoremstyle{definition}
\newtheorem{remark}[theorem]{Remark}

\newcommand{\frM}{{\mathfrak{M}}}
\newcommand{\frN}{{\mathfrak{N}}}

\newcommand{\SSS}{{\slashed{S}}}
\newcommand{\DDD}{{\slashed{D}}}
\newcommand{\dd}{{\boldsymbol{\Delta}}}

\DeclareMathOperator{\ch}{{\mathrm{ch}}}
\DeclareMathOperator{\ph}{{\mathrm{ph}}}

\DeclareMathOperator{\SO}{{\mathrm{SO}}}
\DeclareMathOperator{\BSO}{{\mathrm{BSO}}}
\DeclareMathOperator{\BO}{{\mathrm{BO}}}
\DeclareMathOperator{\Sp}{{\mathrm{Sp}}}
\DeclareMathOperator{\Spin}{{\mathrm{Spin}}}
\DeclareMathOperator{\BSpin}{{\mathrm{BSpin}}}
\DeclareMathOperator{\MSpin}{{\mathrm{MSpin}}}
\DeclareMathOperator{\Cl}{{\mathrm{Cl}}}
\DeclareMathOperator{\CCl}{{\mathbb{C}\mathrm{l}}}
\DeclareMathOperator{\KO}{{\mathrm{KO}}}
\DeclareMathOperator{\KU}{{\mathrm{KU}}}
\DeclareMathOperator{\KSp}{{\mathrm{KSp}}}
\DeclareMathOperator{\KR}{{\mathrm{KR}}}
\DeclareMathOperator{\KQ}{{\mathrm{KQ}}}
\DeclareMathOperator{\KM}{{\mathrm{KM}}}
\DeclareMathOperator{\pt}{{\mathrm{pt}}}
\DeclareMathOperator{\spin}{{\mathrm{spin}}}
\DeclareMathOperator{\ind}{{\mathsf{ind}}}
\DeclareMathOperator{\Ind}{{\mathrm{Ind}}}
\DeclareMathOperator{\Res}{{\mathrm{Res}}}
\DeclareMathOperator{\tind}{{\mathsf{t-ind}}}
\DeclareMathOperator{\aind}{{\mathsf{a-ind}}}
\DeclareMathOperator{\Aut}{{\mathrm{Aut}}}
\DeclareMathOperator{\Ad}{{\mathrm{Ad}}}
\DeclareMathOperator{\Hom}{{\mathrm{Hom}}}
\DeclareMathOperator{\Ext}{{\mathrm{Ext}}}
\DeclareMathOperator{\End}{{\mathrm{End}}}
\DeclareMathOperator{\inv}{{\mathrm{inv}}}

\DeclareMathOperator{\coker}{{\mathrm{coker}}}
\title{\textbf{Invariants of Real Vector Bundles}}

\author{Jiahao Hu}
\date{}
\begin{document}
\maketitle
\begin{spacing}{1.1}
\begin{abstract}
\noindent
For a compact smooth manifold with corners (or finite CW-complex) $X$, we can prescribe a finite set of spin or spin$^h$ manifolds (possibly with boundary) mapping into it so that every real vector bundle over $X$ is determined, up to stable equivalence, by the Dirac indices of the real vector bundle when pulled-back onto those prescribed spin or spin$^h$ manifolds.
\end{abstract}
%\end{spacing}
\setcounter{tocdepth}{2}
\tableofcontents
%\begin{spacing}{1.1}
\section{Introduction}
\subsection{Main results}
The main purpose of this paper is to present a complete set of invariants for deciding whether a real vector bundle is \textit{stably} trivial. Unlike obstruction theory where higher order invariants are defined only when the previous ones vanish, our invariants will be a priori given. Our result is analogous to De Rham's theorem which asserts that a closed differential form is exact if and only if its periods (i.e. integrals) are zero over a set of a priori chosen cycles. Likewise, our invariants for real vector bundles arise from pairing real vector bundles against a set of a priori chosen cycles of the base. The major difference is, in our case the cycles will come equipped with extra geometric structures adapted to the question, and our pairing invokes geometry intensively. Moreover the cycles in our case should be broadly interpreted to include $\ZZ_k$ cycles\footnote{a chain is a $\ZZ_k$ cycle if its boundary is zero modulo $k$.} to deal with stably non-trivial bundles whose certain multiple is stably trivial. The appropriate geometric structures to put on the cycles and how the geometrically structured cycles pair against real vector bundles are the central topics of this paper.

Two types of correlated geometric structures will be considered. One is the well-known spin structure, and the other is its quaternionic sibling--the spin$^h$ structure (\Cref{spinhstructure}), which is less-known but appears to be more natural for the subject of this paper due to a certain duality between the reals and quaternions. The way a real vector bundle pairs against a (spin or spin$^h$) structured cycle is through geometry by means of Dirac operator. The invariants we get for real vector bundles are indices of twisted Dirac operators.

To elaborate, let us now describe our invariants more concretely in terms of spin structured cycles. First consider closed spin manifolds mapping into the base $X$. Such a mapping $f: M\to X$ is called a \textbf{spin cycle} in $X$. For a vector bundle $E$ over $X$, we define a pairing
\[
\left\langle M\xrightarrow{f} X| X\gets E\right\rangle:=\text{index of Dirac operator on $M$ twisted by $f^* E$}
\]
taking values either in $\ZZ$ or $\ZZ_2$ depending on the dimension of $M$ (see e.g. \cite{LM89}). We call these \textbf{integer invariants} and \textbf{parity invariants}. 

Next we consider spin manifolds $M$ with boundary, whose boundary $\partial M$ has several, say $k$, identical parts denoted by $\beta M$ (Bockstein of $M$). Denote by $\overline{M}$ the quotient space of $M$ obtained by gluing $\partial M$ onto $\beta M$.
\begin{figure}[H]
\centering
\tikzset{every picture/.style={line width=0.75pt}} %set default line width to 0.75pt        

\begin{tikzpicture}[x=0.75pt,y=0.75pt,yscale=-0.8,xscale=0.8]
%uncomment if require: \path (0,442); %set diagram left start at 0, and has height of 442

%Curve Lines [id:da6813320657087881] 
\draw    (198,108.5) .. controls (24.4,108.7) and (14.4,259.7) .. (197,259.5) ;
%Shape: Ellipse [id:dp15576601958392544] 
\draw   (183,122.5) .. controls (183,114.77) and (189.72,108.5) .. (198,108.5) .. controls (206.28,108.5) and (213,114.77) .. (213,122.5) .. controls (213,130.23) and (206.28,136.5) .. (198,136.5) .. controls (189.72,136.5) and (183,130.23) .. (183,122.5) -- cycle ;
%Shape: Ellipse [id:dp9270922173675821] 
\draw   (182.33,181.83) .. controls (182.33,174.1) and (189.05,167.83) .. (197.33,167.83) .. controls (205.62,167.83) and (212.33,174.1) .. (212.33,181.83) .. controls (212.33,189.57) and (205.62,195.83) .. (197.33,195.83) .. controls (189.05,195.83) and (182.33,189.57) .. (182.33,181.83) -- cycle ;
%Shape: Ellipse [id:dp6280849744557612] 
\draw   (182,245.5) .. controls (182,237.77) and (188.72,231.5) .. (197,231.5) .. controls (205.28,231.5) and (212,237.77) .. (212,245.5) .. controls (212,253.23) and (205.28,259.5) .. (197,259.5) .. controls (188.72,259.5) and (182,253.23) .. (182,245.5) -- cycle ;
%Curve Lines [id:da8184848136863638] 
\draw    (93,178.34) .. controls (98.11,198.91) and (136.1,198.91) .. (143.4,184.86) ;
%Curve Lines [id:da8821755154159043] 
\draw    (99.57,182.85) .. controls (103.96,168.8) and (136.83,171.81) .. (137.56,185.86) ;
%Shape: Ellipse [id:dp4742433253849416] 
\draw   (477.33,182.17) .. controls (477.33,174.43) and (484.05,168.17) .. (492.33,168.17) .. controls (500.62,168.17) and (507.33,174.43) .. (507.33,182.17) .. controls (507.33,189.9) and (500.62,196.17) .. (492.33,196.17) .. controls (484.05,196.17) and (477.33,189.9) .. (477.33,182.17) -- cycle ;
%Curve Lines [id:da426692926631986] 
\draw    (417.67,108.83) .. controls (501.27,109.8) and (461.27,168.47) .. (492.33,168.17) ;
%Curve Lines [id:da9788990096354018] 
\draw    (417.67,136.83) .. controls (472.53,137.2) and (431.47,195.4) .. (492.33,196.17) ;
%Straight Lines [id:da8952468626336303] 
\draw    (417,168.17) -- (492.33,168.17) ;
%Straight Lines [id:da31786177232204027] 
\draw    (417,196.17) -- (492.33,196.17) ;
%Curve Lines [id:da34747964412156174] 
\draw    (416.67,231.83) .. controls (467.93,231.8) and (432.27,168.8) .. (492.33,168.17) ;
%Curve Lines [id:da7622739227186491] 
\draw    (416.67,259.83) .. controls (492.67,259.83) and (469.6,196.13) .. (492.33,196.17) ;
%Curve Lines [id:da9124218259045427] 
\draw    (197.33,167.83) .. controls (156.4,167.7) and (154.4,136.7) .. (198,136.5) ;
%Curve Lines [id:da6656080921551834] 
\draw    (197.33,195.83) .. controls (154.4,195.7) and (156.4,231.7) .. (197,231.5) ;
%Curve Lines [id:da7922548840746431] 
\draw    (417.67,108.83) .. controls (244.07,109.03) and (234.07,260.03) .. (416.67,259.83) ;
%Curve Lines [id:da5974235619388528] 
\draw    (312.67,178.67) .. controls (317.78,199.25) and (355.76,199.25) .. (363.07,185.19) ;
%Curve Lines [id:da5024424994804679] 
\draw    (319.24,183.19) .. controls (323.62,169.13) and (356.49,172.15) .. (357.22,186.2) ;
%Curve Lines [id:da9790079532252692] 
\draw    (417,168.17) .. controls (376.07,168.03) and (374.07,137.03) .. (417.67,136.83) ;
%Curve Lines [id:da18753810310125096] 
\draw    (417,196.17) .. controls (374.07,196.03) and (376.07,232.03) .. (416.67,231.83) ;
%Straight Lines [id:da3329759864888947] 
\draw    (219.33,181.83) -- (273.2,181.9) ;
\draw [shift={(275.2,181.9)}, rotate = 180.07] [color={rgb, 255:red, 0; green, 0; blue, 0 }  ][line width=0.75]    (10.93,-3.29) .. controls (6.95,-1.4) and (3.31,-0.3) .. (0,0) .. controls (3.31,0.3) and (6.95,1.4) .. (10.93,3.29)   ;

% Text Node
\draw (227,159) node [anchor=north west][inner sep=0.75pt]   [align=left] {glue};
% Text Node
\draw (122,139.4) node [anchor=north west][inner sep=0.75pt]    {$M$};
% Text Node
\draw (344,136.4) node [anchor=north west][inner sep=0.75pt]    {$\overline{M}$};
% Text Node
\draw (512,174.4) node [anchor=north west][inner sep=0.75pt]    {$\beta M$};

\end{tikzpicture}

\caption{Gluing $\partial M$ onto $\beta M$}
\end{figure}
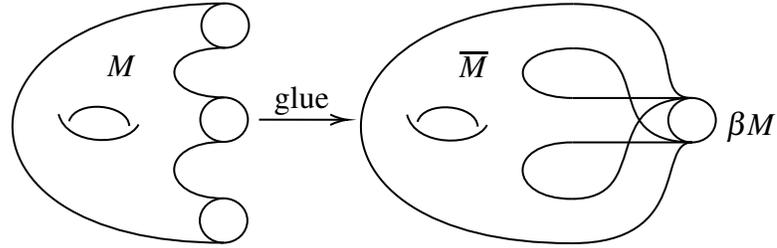
A continuous map $f: \overline{M}\to X$ is called a \textbf{$\ZZ_k$-spin-cycle} or \textbf{torsion spin cycle} in $X$. Then we define a pairing
\begin{align*}
\left\langle \overline{M}\xrightarrow{f} X| X\gets E \right\rangle :=\frac{1}{\epsilon k}\cdot&\text{index of Dirac operator\footnotemark on $M$ twisted by $f^* E$}(\bmod\ZZ)
\end{align*}
\footnotetext{with appropriate Atiyah-Patodi-Singer boundary condition.}taking values in $\QQ/\ZZ$, where $\epsilon=2$ if $\dim M\equiv 4\bmod 8$ and $\epsilon=1$ otherwise. We call these \textbf{angle invariants}.

It follows from index theory that stably equivalent bundles have the same integer, parity and angle invariants over \textit{all} spin cycles and torsion spin cycles. In other words, a \textit{necessary} condition for two vector bundles to be stably equivalent is that they have the same those invariants. One of our main theorems is that this necessary condition is also \textit{sufficient}. Thus these invariants form a complete set of invariants. In particular, a real vector bundle is stably trivial if and only if it has the same invariants as a trivial bundle.
\begin{thmx}[\Cref{thm1'}]\label{thm1}
	For each compact manifold with corners (or finite CW-complex) $X$, there exists a \textit{finite} set of spin cycles and torsion spin cycles in $X$ such that every real vector bundle on $X$ can be determined, up to stable equivalence, by the corresponding integer, parity and angle invariants.
\end{thmx}

The same holds true with spin replaced by spin$^h$, this is our \Cref{thm3}. In fact we will prove \Cref{thm3} first and derive \Cref{thm1} from it by a geometric consequence of Bott periodicity. It is implicitly implied that Dirac operators can be defined on spin$^h$ manifolds and they share similar properties with Dirac operators on spin manifolds. Indeed much effort of this paper is devoted to studying Dirac operators and their indices on spin$^h$ manifolds.

Our \Cref{thm1} is analogous to \cite{Freed88}, in which it is proved that complex vector bundles up to stable equivalence are determined by the indices of twisted Dirac operators over spin$^c$ cycles and torsion spin$^c$ cycles. As is common in algebraic topology, the major difficulty in the real case we consider here lies in handling the $2$-torsion information. To overcome this difficulty, we come up with a general duality theorem between cycles and cocycles. This general duality in our context is a duality between the reals and quaternions that we earlier mentioned, and is the reason why spin$^h$ manifolds should be more natural than spin manifolds in our question. 

To explain this duality, we note the functor
\[
\KO^0(X)=\text{real vector bundles over $X$ modulo stable equivalence}
\]
that we are interested in extends to a generalized cohomology theory $\KO^*(X)$--the real K-theory. We approach the question of finding complete invariants for real vector bundles up to stable equivalence by putting it into a general framework of finding complete invariants for cohomology classes in a generalized cohomology theory. Our idea is that cohomology classes should be determined by their periods over \textit{suitable} cycles. To be precise, let us consider a generalized cohomology theory $h^*$. Since functors $\Hom(-,\QQ)$ and $\Hom(-,\QQ/\ZZ)$ are exact, the functors $\Hom(h^*(-),\QQ)$ and $\Hom(h^*(-),\QQ/\ZZ)$ define generalized \textit{homology} theories. Now there is a third generalized homology theory $Dh_*$, known as the \textbf{Anderson dual} of $h^*$, fitting into the following long exact sequence
\[
\cdots\to Dh_*\to \Hom(h^*,\QQ)\to \Hom(h^*,\QQ/\ZZ)\to Dh_{*-1}\to \cdots
\]
It turns out cohomology classes in a generalized cohomology theory are determined by their periods over cycles and torsion cycles for its Anderson dual theory.
\begin{thmx}[\Cref{thm2'}]\label{thm2}
Let $h^*$ be a generalized cohomology theory of finite type\footnote{that is $h^i(\pt)$ is finitely generated for all $i$.} and let $Dh_*$ be its Anderson dual homology theory. Suppose $X$ is a finite CW-complex, then $h^i(X)$ is naturally isomorphic to the group of homomorphisms $Dh_i(X;\QQ/\ZZ)\xrightarrow{\varphi} \QQ/\ZZ$ that can be lifted to $Dh_i(X;\QQ)\xrightarrow{\tilde{\varphi}}\QQ$ through the canonical "covering homomorphisms" $Dh_i(X;\QQ)\to Dh_i(X;\QQ/\ZZ)$ and $\QQ\to \QQ/\ZZ$. That is,
\[
h^i(X)\cong
\Biggl\{
\begin{tikzcd}[ampersand replacement=\&]
	Dh_i(X;\QQ)\ar[r,"\tilde{\varphi}"]\ar[d] \& \QQ\ar[d]\\
	Dh_i(X;\QQ/\ZZ)\ar[r,"\varphi"] \& \QQ/\ZZ
\end{tikzcd}
\Biggl\}.
\]
The isomorphism holds true with all the groups localized at a set of primes.
\end{thmx}

The maps $\varphi$ and $\widetilde{\varphi}$ record the periods of a cohomology class over cycles and torsion cycles for $Dh$, this will be explained in detail in \Cref{geometricinterpretation}.

Now back to our interested case $h^*=\KO^*$. It is known the Anderson dual of real K-theory is symplectic K-theory (see \cite{Anderson}), therefore \Cref{thm2} implies that real vector bundles up to stable equivalence should be determined by their periods over cycles and torsion cycles for symplectic K-theory. This thus leads us to finding cycles for symplectic K-theory and describing periods of a real vector bundle over those cycles. As discussed above, the most natural cycles we find for symplectic K-theory are spin$^h$ manifolds, and the periods we desire are indices of Dirac operators on spin$^h$ manifolds twisted by the bundle.

%\begin{table}[H]
%\begin{center}
%	\begin{tabular}{c c c}
%		$\RR$ & $\CC$ & $\HH$\\
%		\hline
%		Spin & Spin$^c$ & Spin$^h$\\
%		\hline
%		$\KO$ & $\KU$ & $\KSp$
%	\end{tabular}
%\end{center}
%\end{table}
\subsection{Organization}
This paper is organized as follows. In \Cref{sec2}, we prove \Cref{thm2} and use it to interpret cohomology classes as periods over cycles. \Cref{sec3}, \labelcref{sec4} and \labelcref{sec5} constitute a thorough study of Dirac operators on spin$^h$ manifolds. These three chapters respectively discuss the algebraic, topological and analytical aspects related to spin$^h$ manifolds. In \Cref{sec3} we study quaternionic Clifford algebras and their modules in order to establish a quaternionic version of Atiyah-Bott-Shapiro isomorphism relating quaternionic Clifford modules to symplectic K-theory. In \Cref{sec4} we build "Thom classes" for spin$^h$ vector bundles and prove a Riemann-Roch theorem for spin$^h$ maps. Along the way, we pick out a special characteristic class for spin$^h$ vector bundles analogous to the $\hat{A}$-class for spin vector bundles. We also discuss all possible characteristic classes for (stable) spin$^h$ vector bundles. In \Cref{sec5} we define Dirac operators on spin$^h$ manifolds and study their indices. An index theorem for families of quaternionic operators is needed, whose proof is postponed to \Cref{sec7}. All these are parallel to the study of Dirac operators on spin manifolds. In \Cref{sec6} we prove spin$^h$ manifolds provide enough cycles for symplectic K-theory and show the periods of a real vector bundle, abstractly obtained from \Cref{thm2}, coincides with indices of Dirac operators. This in turn implies \Cref{thm3} and we deduce \Cref{thm1} from \Cref{thm3}.
%\subsection{Conventions and notations}

\section{Periods of generalized cohomology classes}\label{sec2}
In this chapter, we prove \Cref{thm2} and use it to interpret $h^*$-cohomology classes as cocycles over a suitable class of manifolds-with-singularities determined by the homology theory $Dh_*$.

Throughout, generalized (co)homology theories are functors defined on the category of CW pairs satisfying all the Eilenberg-Steenrod axioms except for the dimension axiom. A fundamental result by Brown, Whitehead and Adams says generalized (co)homology theories can be represented by spectra. Spectra will be \underline{underlined} in our notation.

\subsection{Anderson duality}
Let $h^*$ be a generalized cohomology theory with $h^i(\pt)$ finitely generated for all $i$, i.e. $h^*$ is of \textbf{finite type}. Now since functors $\Hom(-,\QQ)$ and $\Hom(-,\QQ/\ZZ)$ are exact, $\Hom(h^*(-),\QQ)$ and $\Hom(h^*(-),\QQ/\ZZ)$ define generalized homology theories which come with a natural transformation $$\Hom(h^*(-),\QQ)\to \Hom(h^*(-),\QQ/\ZZ)$$ induced by the natural quotient map $\QQ\to \QQ/\ZZ$. Since $\Hom(h^*(\pt),\QQ)$ and $\Hom(h^*(\pt),\QQ/\ZZ)$ are countable, by Brown representability theorem these homology theories are represented by spectra denoted by $\underline{D}_\QQ \underline{h}$ and $\underline{D}_{\QQ/\ZZ}\underline{h}$ respectively. Moreover, the natural transformation induced by $\QQ\to \QQ/\ZZ$ is represented by a map $\underline{D}_\QQ \underline{h}\to \underline{D}_{\QQ/\ZZ}\underline{h}$ whose fiber is denoted by $\underline{D}\underline{h}$. We define the \textbf{Anderson dual} homology theory $Dh_*$ of $h^*$ to be the homology theory represented by $\underline{D}\underline{h}$. We also call the cohomology theory $Dh^*$ represented by $\underline{D}\underline{h}$ the Anderson dual cohomology theory of $h^*$.

From definition, for $X$ a finite CW-complex there is a long exact sequence
\begin{equation}\label{les1}
	\cdots \to Dh_i(X)\to \Hom(h^i(X),\QQ)\to \Hom(h^i(X),\QQ/\ZZ)\to Dh_{i-1}(X)\to \cdots
\end{equation}
from which one has for all $i$ the following splittable short exact sequence 
\[
0\to \Ext(h^{i-1}(X),\ZZ)\to Dh_i(X)\to \Hom(h^i(X),\ZZ)\to 0.
\]
In particular $Dh_i(\pt)$ is (non-canonically) isomorphic to the direct sum of the free part of $h^i(\pt)$ and the torsion part of $h^{i-1}(\pt)$.
\begin{example}[\cite{Anderson}]Consider the cases where $h^*$ is singular cohomology $H^*$ (with $\ZZ$-coefficients), complex K-theory $\KU^*$ or real K-theory $\KO^*$.
	\begin{enumerate}[label=(\roman*)] 		\item $DH_*(\pt)$ is concentrated in degree zero and $DH_0(\pt)=\ZZ$. Therefore $DH_*$ is the singular homology theory.
		\item $D\KU_{2i}(\pt)=\ZZ$ and $D\KU_{2i-1}(\pt)=0$. In fact, Anderson showed $D\KU_*=\KU_*$.
		\item $D\KO_*(\pt)=\ZZ, 0,0,0,\ZZ, \ZZ_2,\ZZ_2,0,\ZZ$ for $0\le *\le 8$. In fact, Anderson showed $D\KO_*=\KSp_*$.
	\end{enumerate}
\end{example}
\begin{example}[\cite{Sto12}]
The Anderson dual of topological modular form\footnote{the non-connective, non-periodic version corresponding to Deligne-Mumford compactified moduli stack of elliptic curves.} is the 21-fold suspension of itself.
\end{example}
\begin{remark}
A version of \Cref{thm1} should hold for topological modular form with spin manifolds replaced by string manifolds, provided we have a good index theory developed for string manifolds and topological modular form.
\end{remark}

Anderson duality is indeed a duality in the sense that $D^2$ is the identity. For a proof of this and for more about Anderson duality, we refer the reader to the original paper of Anderson \cite{Anderson}. For our purpose here, we need one extra fact about $Dh_*$ essentially due to Anderson.

\begin{proposition}[{\cite[pp. 42-43]{Anderson}}]\label{prop:uct}
	Let $h^*$ be a generalized cohomology theory of finite type and $Dh_*$ its Anderson dual homology theory. Then for all finite CW-complex $X$ and all $i\in\NN$ there are isomorphisms
	\begin{enumerate}[label=(\roman*)]
		\item $Dh_i(X;\QQ)\cong\Hom(h^i(X),\QQ)$, and
		\item $Dh_i(X;\QQ/\ZZ)\cong\Hom(h^i(X),\QQ/\ZZ)$.
	\end{enumerate}
	Moreover, under these isomorphisms, the long exact sequence \labelcref{les1} is identified with the coefficient long exact sequence
	\begin{equation}\label{les2}
	\cdots \to Dh_i(X)\to Dh_i(X;\QQ)\to Dh_i(X;\QQ/\ZZ)\to Dh_{i-1}(X)\to \cdots
\end{equation}
associated to the short exact sequence $0\to \ZZ\to \QQ\to \QQ/\ZZ\to 0$.
\end{proposition}
\begin{proof}
	Let $\underline{S}$ denote the sphere spectrum and let $\underline{S}_\QQ$, $\underline{S}_{\QQ/\ZZ}$ be the Moore spectra for $\QQ$ and $\QQ/\ZZ$ respectively. Then the fiber sequence $\underline{S}\to \underline{S}_\QQ\to \underline{S}_{\QQ/\ZZ}$ induces a fiber sequence
	\[
	\underline{Dh}\xrightarrow{f} \underline{D} \underline{h}\wedge \underline{S}_\QQ\to \underline{D} \underline{h}\wedge \underline{S}_{\QQ/\ZZ}
	\]
	which corresponds to the coefficient long exact sequence \labelcref{les2}.
	On the other hand, by definition we have the fiber sequence
	\[
	\underline{D}\underline{h}\xrightarrow{g}\underline{D}_\QQ \underline{h}\to \underline{D}_{\QQ/\ZZ}\underline{h}
	\]
	which corresponds to the long exact sequence \labelcref{les1}.
	We will show there is a homotopy equivalence $\varphi: \underline{D} \underline{h}\wedge \underline{S}_\QQ\to \underline{D}_\QQ \underline{h}$ so that $g=\varphi\circ f$. Then it follows there is an induced homotopy equivalence $\varphi': \underline{D} \underline{h}\wedge \underline{S}_{\QQ/\ZZ}\to\underline{D}_{\QQ/\ZZ} \underline{h}$ such that the following diagram commutes:
	\[
	\begin{tikzcd}
	\underline{Dh}\ar[r,"f"]\ar[d,equal]& \underline{D} \underline{h}\wedge \underline{S}_\QQ\ar[r]\ar[d,"\varphi"] & \underline{D} \underline{h}\wedge \underline{S}_{\QQ/\ZZ}\ar[d,"\varphi'"]\\
		\underline{D} \underline{h}\ar[r,"g"] & \underline{D}_\QQ \underline{h}\ar[r]& \underline{D}_{\QQ/\ZZ} \underline{h}
	\end{tikzcd}
	\]
	Now the desired isomorphisms are induced by $\varphi$ and $\varphi'$, and the claim about long exact sequences follows as well.

	To construct $\varphi$, let us consider the commutative diagram
	\[
	\begin{tikzcd}
		\underline{D} \underline{h}\ar[r,"g"]\ar[d,"f"] & \underline{D}_\QQ \underline{h}\ar[d,"j"]\\
		\underline{D} \underline{h}\wedge \underline{S}_\QQ\ar[r,"g\wedge \underline{S}_\QQ"] & \underline{D}_\QQ \underline{h}\wedge \underline{S}_\QQ
	\end{tikzcd}
	\]
	where $j$ is induced by $\underline{S}\to \underline{S}_\QQ$. Applying $\pi_*(-)$ we have
	\[
	\begin{tikzcd}
		Dh_*(\pt)\ar[r,"g_*"]\ar[d,"f_*"] & \Hom(h^*(\pt),\QQ)\ar[d,"j_*"]\\
		Dh_*(\pt)\otimes\QQ\ar[r,"g_*\otimes\QQ"] & \Hom(h^*(\pt),\QQ)\otimes\QQ
	\end{tikzcd}
	\]
	It is clear that $j_*$ is an isomorphism, and it follows by applying the exact functor $\otimes\QQ$ to \labelcref{les1} that $g_*\otimes\QQ$ is an isomorphism since $\Hom(h^*(\pt),\QQ/\ZZ)\otimes\QQ=0$. Therefore both $g\wedge \underline{S}_\QQ$ and $j$ are homotopy equivalences, and thus the desired $\varphi$ can be obtained by composing $g\wedge \underline{S}_\QQ$ with a homotopy inverse to $j$.
	\end{proof}

%\begin{example}[Anderson]$\KU$ is Anderson self-dual and $\KO$ is Anderson dual to $\KSp$.
%\end{example}
%\begin{proof}[Sketch of proof] The cap product $\KU^i(X)\otimes \KU_i(X;\Lambda)\xrightarrow{\cap} \KU_{0}(\pt;\Lambda)=\Lambda$ yields homomorphisms
%\[
%\KU^i(X)\to \Hom(\Hom(\KU_i(X;\Lambda), \Lambda)
%\]
%In particular it yields a map $\KU^i\to D\KU^i$ which is an isomorphism on point.
%
%For $\KO$ it is similar.
%	
%\end{proof}
\subsection{Proof of \Cref{thm2}}
To prove \Cref{thm2}, we need an algebraic version of Pontryagin duality in which the circle group $\RR/\ZZ$ is replaced by $\QQ/\ZZ$ and continuity is imposed by a lifting property using a covering-like construction. To elaborate, let $A$ be an abelian group. Consider the abelian group of commutative diagrams of the form
\begin{equation}\label{doubledual}
\begin{tikzcd}[row sep=scriptsize, column sep=scriptsize]
	\Hom(A,\QQ)\ar[r]\ar[d,"\bmod\ZZ"'] & \QQ\ar[d,"\bmod\ZZ"]\\
	\Hom(A,\QQ/\ZZ)\ar[r] & \QQ/\ZZ
\end{tikzcd}
\end{equation}
In other words, consider the group of pairs $(\varphi,\tilde{\varphi})$ of homomorphisms 
\[
\varphi:\Hom(A,\QQ/\ZZ)\to \QQ/\ZZ,\text{ and } \tilde{\varphi}:\Hom(A,\QQ)\to \QQ
\]
in which $\tilde{\varphi}$ is a "lifting" of $\varphi$ via the "covering homomorphisms"
\[
\Hom(A,\QQ)\xrightarrow{\bmod \ZZ} \Hom(A,\QQ/\ZZ),\text{ and }\QQ\xrightarrow{\bmod\ZZ}\QQ/\ZZ.
\]

Notice that $\tilde{\varphi}$ is determined by $\varphi$ because any two "liftings" of $\varphi$ are differed by a homomorphism from $\Hom(A,\QQ)$ into $\ZZ$ which must be zero. Therefore this group can be viewed as the subgroup of the group of homomorphisms from $\Hom(A,\QQ/\ZZ)$ into $\QQ/\ZZ$ consisting of those "liftable" homomorphisms. In analogy with covering theory, we may think of the "liftable" homomorphisms as continuous homomorphisms and the group of such commutative diagrams as the continuous dual to the group $\Hom(A,\QQ/\ZZ)$, which itself can be considered as a dual to $A$. With these understood, the following can be viewed as an algebraic version of Pontryagin duality.
\begin{proposition}
Let $A$ be a finitely generated abelian group. Then the map 
\[
A\to \Biggl\{\begin{tikzcd}[row sep=scriptsize, column sep=scriptsize]
	\Hom(A,\QQ)\ar[r]\ar[d] & \QQ\ar[d]\\
	\Hom(A,\QQ/\ZZ)\ar[r] & \QQ/\ZZ
\end{tikzcd}\Biggl\}.
\]
induced by evaluation is an isomorphism, where the target is the group of diagrams of the form \labelcref{doubledual}. The isomorphism clearly remains true with all the groups localized at a set of primes.
\end{proposition}
\begin{proof}It suffices to prove for $\ZZ_n$ and $\ZZ$. The former case is a straightforward verification using that $\Hom(\ZZ_n,\QQ)=0$ and $\Hom(\ZZ_n,\QQ/\ZZ)=\ZZ_n$. The latter case is equivalent to the claim that every homomorphism $\varphi:\QQ/\ZZ\to \QQ/\ZZ$ that can be lifted to $\tilde{\varphi}:\QQ\to\QQ$ is a multiplication by  some integer. Indeed $\tilde{\varphi}$ must be a multiplication by some rational number $q$ and in order for $\tilde{\varphi}$ to descend to a map $\varphi$ it is necessary that $q\cdot\ZZ\subset \ZZ$ and consequently $q$ is an integer which in turn implies $\varphi$ is a multiplication by an integer.
\end{proof}
\begin{remark}
	This is an instance of coherent duality of Serre-Grothendieck-Verdier: $\ZZ$, as a coherent sheaf over $\mathrm{Spec}(\ZZ)$, is the dualizing sheaf for coherent duality over $\mathrm{Spec}(\ZZ)$ and $\QQ\to\QQ/\ZZ$ is an injective replacement of $\ZZ$.
\end{remark}
	
%\begin{corollary}
%Let $\mathcal{P}$ be a set of prime numbers, and $A$ a finitely generated $\ZZ_{(\mathcal{P})}$-module, where $\ZZ_{(\mathcal{P})}$ is $\ZZ$ localized at the set of primes $\mathcal{P}$. Then the map
%\[
%A\to \Biggl\{\begin{tikzcd}
%	\Hom(A,\QQ)\ar[r]\ar[d] & \QQ\ar[d]\\
%	\Hom(A,\QQ/\ZZ_{(\mathcal{P})})\ar[r] & \QQ/\ZZ_{(\mathcal{P})}
%\end{tikzcd}\Biggl\}.
%\]
%induced by evaluation is an isomorphism.
%\end{corollary}
%\begin{proof}
%	Apply the exact functor $\otimes\ZZ_{(\mathcal{P})}$ to the previous lemma and use that every finitely generated $\ZZ_{(\mathcal{P})}$-module is localized from a finitely generated $\ZZ$-module.
%\end{proof}

\begin{theorem}[\Cref{thm2}]\label{thm2'} Let $h^*$ be a generalized cohomology theory of finite type, and let $Dh_*$ be its Anderson dual. Suppose $X$ is a finite CW-complex, then for all $i$ we have natural isomorphisms
\begin{equation}\label{Dhcocyclediagram}
	h^i(X)\cong \Biggl\{
\begin{tikzcd}[row sep=scriptsize, column sep=scriptsize]
	Dh_i(X;\QQ)\ar[r]\ar[d] & \QQ\ar[d]\\
	Dh_i(X;\QQ/\ZZ)\ar[r] & \QQ/\ZZ
\end{tikzcd}\Biggl\}.
\end{equation}
The isomorphism remains true with all the groups localized at a set of primes.
\end{theorem}
\begin{proof}
	Combine the above proposition with \Cref{prop:uct}.
\end{proof}
\begin{remark}
	Compare the above to \cite[Theorem 2.1, Theorem 2.2]{MS74} in which the special case $h^*=H^*$ is obtained.
\end{remark}
\begin{remark}\label{inverselimit}
	For $X=\bigcup_n X_n$ an infinite complex filtered by an increasing sequence of finite sub-complexes $X_n$, the the same isomorphism holds with $h^i(X)$ replaced by $\varprojlim_n h^i(X_n)$. This can be proved by applying the above theorem each to $X_n$ and using the facts that homology commutes with direct limit and that $\Hom(\varinjlim-,-)=\varprojlim\Hom(-,-)$. In general $h^i(X)\neq \varprojlim_n h^i(X_n)$, see \cite{Milnor62}.
\end{remark}
%\begin{definition}\label{Dhcocycledefinition}
%	A commutative diagram in the right hand side of \labelcref{Dhcocyclediagram} is called a \textbf{$Dh_i$-cocycle} over $X$.
%\end{definition}
Since the Anderson dual of $\KO$ is $\KSp$, we get
\begin{corollary}\label{periodKO}
	Let $X$ be a finite CW-complex. Then for all $i$ we have natural isomorphisms
	\[
\KO^i(X)\cong\Biggl\{
\begin{tikzcd}[row sep=scriptsize, column sep=scriptsize]
	\KSp_i(X;\QQ)\ar[r]\ar[d] & \QQ\ar[d]\\
	\KSp_i(X;\QQ/\ZZ)\ar[r] & \QQ/\ZZ
\end{tikzcd}\Biggl\}.
\]
\end{corollary}
%\begin{corollary}\label{periodKO}
%	For $X$ a finite CW-complex, we have for all $i$ natural isomorphisms
%	\[
%\KO^i(X)\cong\big\{\text{$\KSp_i$-cocycles}\big\}.
%\]
%\end{corollary}
%\begin{remark}
%	Anderson's proof in fact shows the isomorphism is induced by slant products 
%\[
%	-\backslash-:\KSp_i(X;\Lambda)\otimes \KO^i(X)\to \KSp_0(\pt;\Lambda)=\Lambda
%\]
%\end{remark}
\begin{remark}
	This corollary localized at odd primes, i.e. with $2$ inverted, is obtained by Sullivan in the 70's in studying geometric topology, see \cite[Theorem 6.3]{MITnotes} and note $\KSp[\frac{1}{2}]=\KO[\frac{1}{2}]$. Sullivan's proof relies on that $\KO_*(-;\ZZ_n)$ is a $\ZZ_n$-module for $n$ odd. But since $\KO_2(\pt;\ZZ_2)=\ZZ_4$ is not a $\ZZ_2$-module, that proof cannot be directly applied to deal with $\KO^*$ at prime $2$. It is this difficulty at prime $2$ that motivated the author to formulate and prove \Cref{thm2}.
\end{remark}
%From the last example, we see that real K-theory classes, namely vector bundles up to stable equivalences, are determined by their periods over symplectic K-theory.

%Our next goal would be to find appropriate cycles for symplectic K-theory.

\subsection{Interpretation of \Cref{thm2}}\label{geometricinterpretation}
We will explain why \Cref{thm2} means cohomology classes for $h^*$ are determined by their periods over cycles and torsion cycles for $Dh_*$. Let us begin by pointing out generalized homology classes can be represented by geometric cycles.
\subsubsection{Representing homology classes by geometric cycles}
This subsection is a quick summary of the work of Buoncristiano, Rouke and Sanderson in \cite{BRS}. They showed every generalized \textit{homology} theory is a bordism theory of a suitable class of manifolds-with-singularities. To explain, let us introduce the wonderful notion of transverse CW-complex.
\begin{definition}
	Let $M$ be a compact smooth manifold (with boundary) and $X$ a CW-complex. A continuous map $f: M\to X$ is transverse to an open cell $e$ of $X$ if either $f^{-1}(e)=\varnothing$ or there is a commuting diagram
	\[
	\begin{tikzcd}[row sep=small, column sep=small]
		cl(T) \ar[rr,"t"]\ar[dr,"f|"']& & D^n\ar[dl,"h"]\\
		& X &
	\end{tikzcd}
	\]
	where $T=f^{-1}(e)$, $h$ from the $n$-dimensional closed unit disk $D^n$ into $X$ is characteristic map of the cell $e$ so that $h$ restricted to the open disk is a homeomorphism onto $e$, $t$ is a projection of a smooth bundle (necessarily trivial) and $cl(T)$ is the closure of $T$. Notice that this implies $\hat{T}=t^{-1}(0)$ is a submanifold of $M$ of codimension $n$ and $cl(T)$ is diffeomorphic to $\hat{T}\times D^n$.
	
	We say $f$ is a transverse map if $f$ is transverse to all cells of $X$. We say $X$ is a \textbf{transverse CW-complex} if its attaching maps are transverse to all previous cells.
\end{definition}
\begin{theorem}[Transversality theorem {\cite[pp. 134-135]{BRS}}]
	Every CW-complex is homotopy equivalent to a transverse CW-complex. Every continuous map from a manifold with boundary to a transverse CW-complex which is already transverse on the boundary can be deformed into a transverse map by a homotopy relative to boundary.
\end{theorem}

A transverse CW-complex behaves like a Thom space.
\begin{example}[Moore space]\label{zkexample}
	Let $X$ be the $2$-dimensional Moore space $S^1\cup_{3} D^2$ for $\ZZ_3$ obtained by attaching a $2$-cell onto $S^1$ by a transverse degree $3$ map $g$, then $X$ is a transverse CW-complex. Denote by $\hat{e}_i$ the center of the (unique) $i$-cell $e_i$ of $X$, $i=0,1,2$. Let $Y$ denote the union of three closed line segments in $D^2$ connecting $\hat{e}_2$ to $g^{-1}(\hat{e}_1)$. Denote $Y-g^{-1}(\hat{e}_1)$ by $\mathring{Y}$ and the closure of $\mathring{Y}$ in $X$ by $\overline{Y}$. Then $X$ can be viewed as the Thom space of the "normal bundle" of $\overline{Y}$ in $X$ illustrated by the figure below, in which anything outside of a neighborhood of $\overline{Y}$ is collapsed to a single point $\hat{e}_0$.
	
	\begin{figure}[H]
	\centering
	\tikzset{every picture/.style={line width=0.75pt}} %set default line width to 0.75pt       
	\begin{tikzpicture}[yscale=0.8,xscale=0.8]
		\coordinate (O) at (0, 0);
		\coordinate (A) at (-1.732,1) ;
		\coordinate (B) at (1.732,1);
		\coordinate (C) at (0, -2);
		
		\draw (O) circle (2);
		\draw[thick] (O)--(A);
		\draw[thick] (O)--(B);
		\draw[thick] (O)--(C);
		\filldraw (O) circle (0.05) node[above] {$\hat{e}_2$};
		\filldraw (A) circle (0.05) node[above left] {$\hat{e}_1$};
		\filldraw (B) circle (0.05) node[above right] {$\hat{e}_1$};
		\filldraw (C) circle (0.05) node[below] {$\hat{e}_1$};
		
		\draw (-1.414,1.414) arc(-135:-45:2);
		\draw (-0.5176,-1.931) arc(-15:75:2);
		\draw (0.5176,-1.931) arc(195:105:2);
		
		\coordinate (A1) at (-1.3856,0.8);
		\coordinate (A2) at (-1.0392,0.6);
		\coordinate (A3) at (-0.6928,0.4);
		\coordinate (A4) at (-0.3464,0.2);
		
		\coordinate (a1) at (-1.1756,1.21);
		\coordinate (a2) at (-0.908,1.0464);
		\coordinate (a3) at (-0.618,0.9263);
		\coordinate (a4) at (-0.313,0.853);
		
		\coordinate (a_1) at (-1.636,0.413);
		\coordinate (a_2) at (-1.36,0.263);
		\coordinate (a_3) at (-1.11,0.072);
		\coordinate (a_4) at (-0.895,-0.156);
		
		\draw (A1)--(a1);
		\draw (A2)--(a2);
		\draw (A3)--(a3);
		\draw (A4)--(a4);
		\draw (A1)--(a_1);
		\draw (A2)--(a_2);
		\draw (A3)--(a_3);
		\draw (A4)--(a_4);
		\draw (O)--(0,0.828);
		
		\coordinate (B1) at (1.3856,0.8);
		\coordinate (B2) at (1.0392,0.6);
		\coordinate (B3) at (0.6928,0.4);
		\coordinate (B4) at (0.3464,0.2);
		
		\coordinate (b4) at (0.313,0.8531);
		\coordinate (b3) at (0.618,0.926);
		\coordinate (b2) at (0.908,1.046);
		\coordinate (b1) at (1.18,1.21);
		
		\coordinate (b_4) at (0.895,-0.156);
		\coordinate (b_3) at (1.11,0.072);
		\coordinate (b_2) at (1.36,0.263);
		\coordinate (b_1) at (1.636,0.413);
		
		\draw (B1)--(b1);
		\draw (B2)--(b2);
		\draw (B3)--(b3);
		\draw (B4)--(b4);
		\draw (B1)--(b_1);
		\draw (B2)--(b_2);
		\draw (B3)--(b_3);
		\draw (B4)--(b_4);

		\coordinate (C1) at (0,-1.6);
		\coordinate (C2) at (0,-1.2);
		\coordinate (C3) at (0,-0.8);
		\coordinate (C4) at (0,-0.4);
		
		\coordinate (c1) at (-0.4604,-1.623);
		\coordinate (c2) at (-0.452,-1.31);
		\coordinate (c3) at (-0.493,-0.998);
		\coordinate (c4) at (-0.582,-0.697);
		
		\coordinate (c_1) at (0.4604,-1.623);
		\coordinate (c_2) at (0.452,-1.31);
		\coordinate (c_3) at (0.493,-0.998);
		\coordinate (c_4) at (0.582,-0.697);
		
		\draw (C1)--(c1);
		\draw (C2)--(c2);
		\draw (C3)--(c3);
		\draw (C4)--(c4);
		\draw (C1)--(c_1);
		\draw (C2)--(c_2);
		\draw (C3)--(c_3);
		\draw (C4)--(c_4);
		\draw (O)--(-0.717,-0.414);
		\draw (O)--(0.717,-0.414);
		
	\end{tikzpicture}
	\caption{Moore space viewed as a Thom space}
	\end{figure}
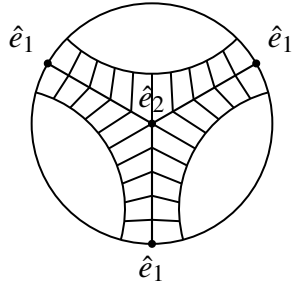
	
	Now let $f:S^n\to X$ be a transverse map, then $\beta M:=f^{-1}(\hat{e}_2)$ is a codimension $2$ submanifold of $S^n$ with an open neighborhood $f^{-1}(e_2)$ diffeomorphic to $\beta M\times e_2$. Denote $S_0:=S^n-f^{-1}(e_2)$, and $f_0:=f|_{S_0}$, then $M_0:=f_0^{-1}(\hat{e}_1)$ is a codimension $1$ submanifold of $S_0$ with an open neighborhood $f_0^{-1}(e_1)$ diffeomorphic to $M_0\times e_1$. Define $\overline{M}:=M_0\cup cl(f^{-1}(\mathring{Y}))$. Then $\overline{M}$ is a manifold with singularity whose singular locus is $\beta M$. Near $\beta M$, $\overline{M}$ is diffeomorphic to $\beta M\times \text{cone(3 points)}$. Therefore $\overline{M}$ is a $\ZZ_3$-manifold (see \Cref{defn:zkmanifold} below). Moreover $\overline{M}$ admits a normal framing inherited from the normal framing of $\hat{e}_1$ in $e_1$ and the normal framing of $\mathring{Y}$ in $e_2$.
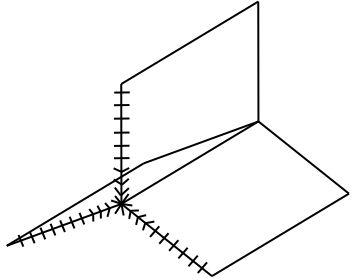
\begin{figure}[H]
\centering
\tikzset{every picture/.style={line width=0.75pt}} %set default line width to 0.75pt       
\begin{tikzpicture}[x=0.75pt,y=0.75pt,yscale=-0.85,xscale=0.85]
%uncomment if require: \path (0,451); %set diagram left start at 0, and has height of 451

%Straight Lines [id:da9656639672878554] 
\draw    (184.92,228.45) -- (185.04,299.63) ;
%Straight Lines [id:da9450250982558156] 
\draw    (185.04,299.63) -- (266.33,250.67) ;
%Straight Lines [id:da6038515153066174] 
\draw    (266.22,179.49) -- (266.33,250.67) ;
%Straight Lines [id:da06140338702221315] 
\draw    (238.75,342.5) -- (185.04,299.63) ;
%Straight Lines [id:da28499899064559275] 
\draw    (320.05,293.54) -- (266.33,250.67) ;
%Straight Lines [id:da903038588169383] 
\draw    (185.04,299.63) -- (117,324.5) ;
%Straight Lines [id:da2810733273606787] 
\draw    (266.33,250.67) -- (198.3,275.54) ;
%Straight Lines [id:da9290924313750312] 
\draw    (178.9,297.58) -- (185.04,299.63) ;
%Straight Lines [id:da5718006496420481] 
\draw    (185.04,299.63) -- (189.3,293.75) ;
%Straight Lines [id:da8490868914334062] 
\draw    (181.25,294.9) -- (185.04,299.63) ;
%Straight Lines [id:da4729528225846835] 
\draw    (185.04,299.63) -- (186.37,306.5) ;
%Straight Lines [id:da632675767965731] 
\draw    (185.04,299.63) -- (181.25,305.5) ;
%Straight Lines [id:da5404893919687563] 
\draw    (191.52,300.58) -- (185.04,299.63) ;
%Straight Lines [id:da070852433939112] 
\draw    (181.11,289.92) -- (184.69,294.63) ;
%Straight Lines [id:da8695939493626186] 
\draw    (189.47,290.25) -- (184.69,294.63) ;
%Straight Lines [id:da7393377125173701] 
\draw    (180.43,285.1) -- (185.34,288.23) ;
%Straight Lines [id:da4010777736883906] 
\draw    (180.43,278.5) -- (185.1,280.63) ;
%Straight Lines [id:da3098956611070536] 
\draw    (185.1,280.63) -- (189.64,278.3) ;
%Straight Lines [id:da5420483203207658] 
\draw    (185.34,288.23) -- (189.43,284.7) ;
%Straight Lines [id:da5879054197081399] 
\draw    (190.53,303.72) -- (195.37,303.5) ;
%Straight Lines [id:da8464299513844552] 
\draw    (190.53,303.72) -- (189.64,308.5) ;
%Straight Lines [id:da8947953762829045] 
\draw    (194.82,307.52) -- (193.73,312.1) ;
%Straight Lines [id:da24050846590244646] 
\draw    (199.67,306.3) -- (194.82,307.52) ;
%Straight Lines [id:da38091804692278297] 
\draw    (199.12,311.12) -- (197.82,315.3) ;
%Straight Lines [id:da32483936575981187] 
\draw    (199.12,311.12) -- (204.58,310.1) ;
%Straight Lines [id:da028406574902301363] 
\draw    (178.08,302.23) -- (176.75,307.3) ;
%Straight Lines [id:da6516507142329728] 
\draw    (175.32,299.1) -- (178.08,302.23) ;
%Straight Lines [id:da44967114824661425] 
\draw    (172.76,304.43) -- (172.86,308.3) ;
%Straight Lines [id:da5008405344413327] 
\draw    (170.81,300.7) -- (172.76,304.43) ;
%Straight Lines [id:da9170128949257327] 
\draw    (165.9,302.7) -- (168.26,305.63) ;
%Straight Lines [id:da2873348398143304] 
\draw    (168.26,305.63) -- (168.36,310.1) ;
%Straight Lines [id:da0157309849522711] 
\draw    (117,324.5) -- (198.3,275.54) ;
%Straight Lines [id:da8042125190378463] 
\draw   (238.75,342.5) -- (320.05,293.54) ;
%Straight Lines [id:da13247291981216303] 
\draw   (184.92,228.45) -- (266.22,179.49) ;
%Straight Lines [id:da8428505626753663] 
\draw    (180.6,272.53) -- (189.83,272.42) ;
%Straight Lines [id:da9632613993267056] 
\draw    (180.6,265.2) -- (189.83,265.08) ;
%Straight Lines [id:da9529036228300408] 
\draw    (180.6,257.37) -- (189.83,257.25) ;
%Straight Lines [id:da6208505369330477] 
\draw    (180.6,249.03) -- (189.83,248.92) ;
%Straight Lines [id:da41813932542091514] 
\draw    (180.6,241.03) -- (189.83,240.92) ;
%Straight Lines [id:da5629960070788914] 
\draw    (180.77,233.53) -- (190,233.42) ;
%Straight Lines [id:da7598558532147273] 
\draw    (202,319.25) -- (207.5,312.92) ;
%Straight Lines [id:da3833912806262655] 
\draw    (213.33,327.5) -- (218.83,321.17) ;
%Straight Lines [id:da03004077893099133] 
\draw    (207.5,323.33) -- (213,317) ;
%Straight Lines [id:da11065906764603739] 
\draw    (218.83,332) -- (224.33,325.67) ;
%Straight Lines [id:da6948796296899079] 
\draw    (224.83,336.67) -- (230.33,330.33) ;
%Straight Lines [id:da8379281812539807] 
\draw    (230.33,340.68) -- (235.83,334.35) ;
%Straight Lines [id:da9985257699398685] 
\draw    (163,311.83) -- (160.17,304.92) ;
%Straight Lines [id:da37710791159178814] 
\draw    (157.17,314.17) -- (154.33,307.25) ;
%Straight Lines [id:da6257413883569733] 
\draw    (151.5,316.33) -- (148.67,309.42) ;
%Straight Lines [id:da23900207101535975] 
\draw    (145.83,318.33) -- (143,311.42) ;
%Straight Lines [id:da8452660946316419] 
\draw    (139.83,320.67) -- (137,313.75) ;
%Straight Lines [id:da02938745411735466] 
\draw    (133.5,323) -- (130.67,316.08) ;
%Straight Lines [id:da5858731270032596] 
\draw    (126.83,325) -- (124,318.08) ;
\end{tikzpicture}
\caption{Local structure near $\beta M$}
\end{figure}
	
Reversing the above construction, we see conversely that every framed $\ZZ_3$-manifold embedded in $S^n$ yields a transverse map from $S^n$ to $X$. So $X$ can be viewed as the Thom space for (codimension 1) framed $\ZZ_3$-manifold.
\end{example}

\begin{definition}\label{defn:zkmanifold}
	A \textbf{$\ZZ_k$-manifold} is a pair $(M,\beta M)$ consisting of an oriented smooth manifold with boundary $M$ and a closed oriented manifold $\beta M$ so that $\partial M$ is a disjoint union of $k$-copies (assume labelled) of $\beta M$. Let $\overline{M}$ be the quotient space of $M$ obtained by gluing $\partial M$ onto $\beta M$, we also say $\overline{M}$ is a $\ZZ_k$-manifold. The manifold $\beta M$ (also denoted by $\beta\overline{M}$) is called the \textbf{Bockstein} of $\overline{M}$.
	
	If $(N,\beta N)$ is another $\ZZ_k$-manifold, then $(M,\beta M)$ is \textbf{$\ZZ_k$-cobordant} to $(N,\beta N)$ if $\beta M$ is cobordant to $\beta N$ via a cobordism $W$, i.e. $\partial W=\beta W\sqcup (-\beta N)$, and $ M\cup_{\partial M} kW\cup_{(-\partial N)} (-N)$ is a boundary. A $\ZZ_k$-manifold is a \textbf{$\ZZ_k$-boundary} if it is $\ZZ_k$-cobordant to the empty $\ZZ_k$-manifold.
	
\begin{figure}[H]
\centering

\tikzset{every picture/.style={line width=0.75pt}} %set default line width to 0.75pt        

\begin{tikzpicture}[x=0.75pt,y=0.75pt,yscale=-0.6,xscale=0.6]
%uncomment if require: \path (0,490); %set diagram left start at 0, and has height of 490

%Shape: Ellipse [id:dp03594902648240739] 
\draw   (150,99) .. controls (150,90.53) and (162.61,83.67) .. (178.17,83.67) .. controls (193.72,83.67) and (206.33,90.53) .. (206.33,99) .. controls (206.33,107.47) and (193.72,114.33) .. (178.17,114.33) .. controls (162.61,114.33) and (150,107.47) .. (150,99) -- cycle ;
%Curve Lines [id:da42964725923736835] 
\draw    (150,99) .. controls (155,141) and (148.33,184.33) .. (75.67,189.67) ;
%Shape: Ellipse [id:dp42978836325054603] 
\draw   (74.64,245.99) .. controls (66.18,245.84) and (59.54,233.1) .. (59.82,217.55) .. controls (60.11,202) and (67.2,189.51) .. (75.67,189.67) .. controls (84.13,189.82) and (90.77,202.55) .. (90.49,218.11) .. controls (90.2,233.66) and (83.11,246.14) .. (74.64,245.99) -- cycle ;
%Shape: Ellipse [id:dp7041755708040223] 
\draw   (272.26,195.95) .. controls (280.71,196.36) and (286.97,209.28) .. (286.23,224.82) .. controls (285.48,240.36) and (278.02,252.63) .. (269.56,252.22) .. controls (261.1,251.82) and (254.85,238.89) .. (255.59,223.36) .. controls (256.34,207.82) and (263.8,195.55) .. (272.26,195.95) -- cycle ;
%Shape: Boxed Bezier Curve [id:dp6049449429840542] 
\draw    (272.26,195.95) .. controls (230.56,188.86) and (190.88,170.21) .. (206.33,99) ;
%Curve Lines [id:da301109844938639] 
\draw    (74.64,245.99) .. controls (149.67,225) and (180.33,229) .. (269.56,252.22) ;
%Curve Lines [id:da4647683165037243] 
\draw    (145.33,200.33) .. controls (156.33,175) and (183,167) .. (207.67,193) ;
%Curve Lines [id:da33154443930688826] 
\draw    (151,195) .. controls (163.33,207.67) and (192.33,204.33) .. (199.67,190.33) ;
%Shape: Ellipse [id:dp32517317869608764] 
\draw   (150,85.67) .. controls (150,77.2) and (162.61,70.33) .. (178.17,70.33) .. controls (193.72,70.33) and (206.33,77.2) .. (206.33,85.67) .. controls (206.33,94.14) and (193.72,101) .. (178.17,101) .. controls (162.61,101) and (150,94.14) .. (150,85.67) -- cycle ;
%Curve Lines [id:da09057878736205816] 
\draw    (150,85.67) .. controls (147.67,43) and (201,26.33) .. (206.33,85.67) ;
%Shape: Ellipse [id:dp9426696961918606] 
\draw   (283.94,196.52) .. controls (292.39,197.13) and (298.32,210.21) .. (297.2,225.72) .. controls (296.08,241.24) and (288.32,253.32) .. (279.87,252.71) .. controls (271.43,252.1) and (265.49,239.02) .. (266.61,223.51) .. controls (267.74,207.99) and (275.49,195.91) .. (283.94,196.52) -- cycle ;
%Curve Lines [id:da4663217798319713] 
\draw    (283.94,196.52) .. controls (326.66,197.28) and (339.44,251.67) .. (279.87,252.71) ;
%Shape: Ellipse [id:dp6565302521043124] 
\draw   (63.32,246.47) .. controls (54.85,246.76) and (47.57,234.38) .. (47.05,218.84) .. controls (46.53,203.29) and (52.97,190.46) .. (61.43,190.17) .. controls (69.89,189.89) and (77.18,202.26) .. (77.7,217.81) .. controls (78.22,233.36) and (71.78,246.19) .. (63.32,246.47) -- cycle ;
%Curve Lines [id:da18132174592601236] 
\draw    (63.32,246.47) .. controls (20.75,250.24) and (2.31,197.49) .. (61.43,190.17) ;

% Text Node
\draw (123.33,201.4) node [anchor=north west][inner sep=0.75pt]    {$M$};
% Text Node
\draw (213,84.07) node [anchor=north west][inner sep=0.75pt]    {$\beta M$};
% Text Node
\draw (58.67,249.4) node [anchor=north west][inner sep=0.75pt]    {$\beta M$};
% Text Node
\draw (264.67,254.4) node [anchor=north west][inner sep=0.75pt]    {$\beta M$};
% Text Node
\draw (26.67,211.07) node [anchor=north west][inner sep=0.75pt]    {$W$};
% Text Node
\draw (301.33,219.4) node [anchor=north west][inner sep=0.75pt]    {$W$};
% Text Node
\draw (169.33,54.07) node [anchor=north west][inner sep=0.75pt]    {$W$};

\end{tikzpicture}
\caption{$\ZZ_k$-boundary}
\end{figure}

%In general a \textbf{$\ZZ_k$-manifold with boundary} is a pair $(M,\beta M)$ of oriented manifolds-with-boundary with $k$ disjoint codimension zero compact embeddings of $\beta M$ into $\partial M$. The boundary of $(M,\beta M)$ is the $\ZZ_k$-manifold $\left(\partial M-int (k\beta M),\partial (\beta M)\right)$. We also call $\overline{M}$, the quotient space of $M$ obtained by gluing $\beta M$'s together, a $\ZZ_k$-manifold with boundary, and denote $\partial \overline{M}=\left(\partial M-int (k\beta M),\partial (\beta M)\right)$.
\end{definition}
\begin{example}[$\ZZ_k$-sphere]\label{zksphereexample}
	Let $X$ be $S^n$ with $k$ disjoint open disks removed, then $X$ is a $\ZZ_k$-manifold with $\beta X=S^{n-1}$, we call $(X,\beta X)$ the $\ZZ_k$-$n$-sphere and denote it by $\overline{S^n}$. Observe that $\overline{S^n}$ is $\ZZ_k$-cobordant to the $\ZZ_k$-manifold $(S^n,\varnothing)$ since $\beta X=S^{n-1}$ is cobordant to $\varnothing$ via $D^n$ and $X\cup_{\partial X} kD^n$ is exactly $S^n$. Moreover, both $\overline{S^n}$ and $(S^n,\varnothing)$ are framed, and the $\ZZ_k$-cobordism is a framed one.
\end{example}
We leave it to the reader to define smooth maps between $\ZZ_k$-manifolds, $\ZZ_k$-submanifolds, $\ZZ_k$-manifolds with extra tangential structures and their corresponding notions of cobordisms.

Combining the transversality theorem and the previous example one can show there is a 1-1 correspondence between the set of homotopy class of maps $[S^n, S^j\cup_k D^{j+1}]$ and the set of cobordism classes of codimension $j$ framed $\ZZ_k$-manifolds embedded in $S^n$. In general every CW-complex can be viewed as a generalized Thom space for framed manifolds-with-singularities whose singularity type is determined by the CW-complex. Consequently every CW-spectrum can be viewed as a generalized Thom spectrum whose corresponding homology theory is a bordism theory of framed manifolds with certain type of singularities.

For a CW-spectrum $\underline{X}$, we call its corresponding homology theory the bordism theory of manifolds with $\underline{X}$-singularity. A mapping from a manifold with $\underline{X}$-singularity into a space is called an \textbf{$\underline{X}$-cycle} in that space, an $\underline{X}$-cycle is an \textbf{$\underline{X}$-boundary} if the map is cobordant (with $\underline{X}$-singularity) to the empty $\underline{X}$-cycle.
\begin{example}\label{mooreasthom}
	\begin{enumerate}[label=(\roman*)]
		\item For a usual Thom spectrum, the singularity is virtual in the sense that the corresponding bordism theory are formed by smooth manifolds, but with normal structures degenerated from normal framings to weaker structures.
		\item The Moore spectrum for $\ZZ_k$, $\underline{S}_{\ZZ_k}$, corresponds to the bordism theory of framed $\ZZ_k$-manifolds.
			\begin{enumerate}
			\item The reducing modulo $k$ transformation $\underline{S}\to \underline{S}_{\ZZ_k}$ induced by $\ZZ\xrightarrow{\bmod k}\ZZ_k$ corresponds to the geometric operation $M\mapsto (M,\varnothing)$.
			\item The transformation $\underline{S}_{\ZZ_k}\to \underline{S}_{\ZZ_{kl}}$ induced by $\ZZ_k\xrightarrow{\times l}\ZZ_{kl}$ corresponds to the geometric multiplication operation $(M,\beta M)\mapsto (\underbrace{M\sqcup M\sqcup\cdots\sqcup M}_\text{$l$ times} ,\beta M)=:(lM,\beta M)$.
			\item  The Bockstein transformation $\underline{S}_{\ZZ_k}\to \Sigma\underline{S}$ corresponds exactly to the geometric operation $(M,\beta M)\mapsto \beta M$, where $\Sigma$ means suspension.
			\end{enumerate}
		\item If $\underline{X},\underline{Y}$ are CW-spectra, then $\underline{X}\wedge \underline{Y}$-singularity combines $\underline{X}$-singularity and $\underline{Y}$-singularity. For instance $\underline{\MSpin}\wedge \underline{S}_{\ZZ_k}$ corresponds to the bordism theory of spin $\ZZ_k$-manifolds. Here $\underline{\MSpin}$ is the Thom spectrum for spin cobordism theory.
\end{enumerate}
\end{example} 
\subsubsection{Cohomology classes as periods over cycles}\label{periods}
Let $h^*$ be a generalized cohomology theory of finite type and $X$ a finite CW-complex. Then using $Dh_i(X;\QQ)=Dh_i(X)\otimes \QQ$ and $Dh_i(X;\QQ/\ZZ)=\varinjlim_k Dh_i(X,\ZZ_k)$ where the direct system is formed by maps induced by $\ZZ_k\xrightarrow{\times l}\ZZ_{kl}$ for all $k,l$, \Cref{thm2} says a cohomology class in $h^i(X)$ is equivalent to the following data:
\begin{enumerate}[label=(\roman*)]
	\item An assignment $\lambda_\QQ$, called $\QQ$-periods, that assigns to each $i$-dimensional $\underline{Dh}$-cycle (see remark below) in $X$ a rational number that is additive upon disjoint union of cycles and vanishes on $\underline{Dh}$-boundaries in $X$.
	\item An assignment $\lambda_{\QQ/\ZZ}$, called $\QQ/\ZZ$-periods, that assigns for each $k$ and to each $i$-dimensional $\underline{Dh}$-$\ZZ_k$-cycle in $X$ a rational angle in $\QQ/\ZZ$ that is additive upon disjoint union of $\ZZ_k$-cycles and vanishes on $\underline{Dh}$-$\ZZ_k$-boundaries. Moreover the assignment is unchanged if a $\ZZ_k$-cycle is treated as a $\ZZ_{kl}$-cycle by multiplication $(M,\beta M)\mapsto (lM,\beta M)$.
	\item The two assignments $\lambda_\QQ$ and $\lambda_{\QQ/\ZZ}$ are compatible, that is the following diagram commutes:
	\[
	\begin{tikzcd}
		\text{$\underline{Dh}$-cycles}\ar[r,"\frac{1}{k}\lambda_\QQ"]\ar[d,"\bmod k"] &\QQ\ar[d,"\bmod\ZZ"] \\
		\text{$\underline{Dh}$-$\ZZ_k$-cycles}\ar[r,"\lambda_{\QQ/\ZZ}"] &\QQ/\ZZ
	\end{tikzcd}
	\]
\end{enumerate}
In short, a $h^*$-cohomology class is equivalent to a compatible system of $\QQ$- and $\QQ/\ZZ$-periods over $\underline{Dh}$-cycles and $\underline{Dh}$-$\ZZ_k$-cycles.

\begin{remark}
	We caution the reader that by an $i$-dimensional $\underline{Dh}$-cycle, we mean a geometric cycle representing an element in $Dh_i(X)$, which does not necessarily have geometric dimension $i$, but might be of mix dimension (possibly unbounded). However if $\underline{Dh}$ happens to be connective, then $Dh_*(X)=0$ for $*<0$ and an element in $Dh_i(X)$ can indeed be represented by a geometric cycle of homogeneous dimension $i$. See also \cite[p. 144, Remark 5.2]{BRS}.
\end{remark}

In practice the $\underline{Dh}$-singularity is hard to describe concretely, and if $h^*$-cohomology classes themselves have geometric interpretations, one wonders how the periods are related to geometry. These are the problems we will address for $h^*=\KO^*$ and $Dh_*=\KSp_*$ in the rest of this paper.

\section{Quaternionic Clifford modules}\label{sec3}
In this chapter, we develop an algebraic theory of quaternionic Clifford algebras and quaternionic Clifford modules for the geometric study of spin$^h$ manifolds and their Dirac operators. In particular, we prove a quaternionic version of Atiyah-Bott-Shapiro isomorphism that relates quaternionic Clifford modules to symplectic K-theory.

Throughout $\RR,\CC,\HH$ stand for the real, complex, quaternion number-fields respectively. $\mathbf{i},\mathbf{j},\mathbf{k}$ will be the standard basis for the imaginary quaternions, and $\mathbf{i}\in\CC$ is the standard square root of $-1$.

\subsection{Quaternionic Clifford algebras}
\subsubsection{Review of Atiyah-Bott-Shapiro isomorphism}
Let $\Cl_n$ be the (real) Clifford algebra on $\RR^n$ with respect to the quadratic form $\|\cdot\|^2$ where $\|\cdot\|$ the Euclidean norm. That is, $\Cl_n$ is the unital associative $\RR$-algebra generated by $\RR^n$ subject to relations $e^2=-\|e\|^2$ for all $e\in\RR^n$. If $e_1,\dots,e_n$ is the standard orthonormal basis of $\RR^n$ then $\Cl_n$ is the associative $\RR$-algebra generated by a unit and the symbols $e_i$ subject to relations $e_i^2=-1$ for $1\le i\le n$ and $e_i e_j+e_j e_i=0$ for $i\neq j$. A basis of $\Cl_n$ is given by $\{e_{i_1}e_{i_2}\cdots e_{i_k}\}$ for $1\le i_1<\dots<i_k\le n$. In particular $\Cl_n$ has dimension $2^n$.

The antipodal map $\RR^n\to \RR^n, e\mapsto -e$ extends to an automorphism of $\Cl_n$ of order $2$, whose eigenspace decomposition yields a $\ZZ_2$-grading $\Cl_n=\Cl_n^0\oplus\Cl_n^1$ where $\Cl_n^\alpha$ is the eigenspace of eigenvalue $(-1)^\alpha$. It is clear that the multiplication in $\Cl_n$ respects the $\ZZ_2$-grading in the sense that $\Cl_n^\alpha\cdot \Cl_n^\beta\subset \Cl_n^{\alpha+\beta}$ for $\alpha,\beta\in \ZZ_2$. In other words, $\Cl_n$ is a $\ZZ_2$-graded algebra.

For $\KK=\RR,\CC$ or $\HH$, by a $\ZZ_2$-graded $\KK$-module over $\Cl_n$ we mean a $\ZZ_2$-graded $\KK$-module $V=V^0\oplus V^1$ equipped with a $\RR$-linear (left) $\Cl_n$-action that commutes with scalar multiplication by $\KK$ and satisfies $\Cl_n^\alpha\cdot V^\beta\subset V^{\alpha+\beta}$ where $\alpha,\beta\in \ZZ_2$.

Let $\hat{\frM}_n(\KK)$ denote the Grothendieck group of finite dimensional $\ZZ_2$-graded $\KK$-modules over $\Cl_n$ with respect to direct sum. Let $\hat{\frN}_n(\KK)$ denote the cokernel of the map $i^*:\hat{\frM}_{n+1}(\KK)\to \hat{\frM}_{n}(\KK)$ induced by the embedding of $\ZZ_2$-graded algebras $i_*: \Cl_n\to \Cl_{n+1}$ extended from the \textit{isometric} embedding $\RR^n\to \RR^{n+1}, e\mapsto (e,0)$.

Using that $\Cl_n\hat{\otimes}\Cl_m=\Cl_{m+n}$ where $\hat{\otimes}$ means $\ZZ_2$-graded tensor product (\Cref{Z2gradedtensorproduct}), Atiyah, Bott and Shapiro \cite{ABS64} showed for $\KK=\RR$ or $\CC$ the graded group $\hat{\frM}_\bullet(\KK)=\bigoplus_{n\ge 0}\hat{\frM}_n(\KK)$ forms a commutative graded ring (with unit) with respect to direct sum and $\ZZ_2$-graded tensor product. Further, this ring structure descends to make $\hat{\frN}_\bullet(\KK)=\bigoplus_{n\ge 0}\hat{\frN}_n(\KK)$ into a commutative graded ring. Moreover they constructed the following celebrated graded ring isomorphisms relating Clifford modules to K-theories
	\[
	\begin{aligned}
		&\varphi: \hat{\frN}_\bullet(\RR)\xrightarrow{\simeq} \KO^{-\bullet}(\pt)=\bigoplus_{n\ge 0}\KO^{-n}(\pt),\\
		&\varphi^c: \hat{\frN}_\bullet(\CC)\xrightarrow{\simeq} \KU^{-\bullet}(\pt)=\bigoplus_{n\ge 0}\KU^{-n}(\pt).
	\end{aligned}
	\]
As we will see, the construction of the above isomorphisms easily extends to the quaternionic case to give a map of graded abelian groups
 $$\varphi^h:\hat{\frN}_\bullet(\HH)\to \KSp^{-\bullet}(\pt)=\bigoplus_{n\ge 0}\KSp^{-n}(\pt).$$
 We would like to show $\varphi^h$ is an isomorphism. In order so, we must examine $\HH$-modules over $\Cl_n$. But notice that $\HH$-modules over $\Cl_n$ are the same as $\RR$-modules over $\Cl_n\otimes_\RR\HH$. So we will begin by studying this algebra $\Cl_n\otimes_\RR\HH$. Our following discussion will be guided by the Bott's $4$-fold periodicity $\KSp^{-n}=\KO^{-n-4}$, $\KO^{-n}=\KSp^{-n-4}$.
\subsubsection{Quaternionic Clifford algebra and complexification}
\begin{definition}
	We define the ($n$-th) \textbf{quaternionic Clifford algebra} $\Cl_n^h$ to be the associative $\RR$-algebra $\Cl_n\otimes_\RR\HH$. Adopting the notation of \cite{LM89} we denote the complex Clifford algebra $\Cl_n\otimes_\RR\CC$ by $\CCl_n$ and denote the complexification of the quaternionic Clifford algebra $\Cl_n^h\otimes_\RR\CC$ by $\CCl_n^h$. Both $\CCl_n$ and $\CCl_n^h$ are associative $\CC$-algebras.
\end{definition}

Since the real Clifford algebras are classified, we can classify the quaternionic Clifford algebras and their complexifications using the well-known identities:
\begin{equation}\label{fundamentalid}
	\begin{aligned}
	&\CC\otimes_\RR\CC\cong \CC\oplus\CC\\
    &\HH\otimes_\RR\CC\cong \CC(2)\\
    &\HH\otimes_\RR\HH\cong\RR(4)\\
    &\RR(n)\otimes_\RR\RR(m)\cong\RR(nm)\text{ for all }n,m\\
    &\RR(n)\otimes_\RR \KK\cong \KK(n)\text{ for all }n, \KK=\RR,\CC,\HH
\end{aligned}
\end{equation}
where $\KK(n)$ means the full $n\times n$ matrix algebra over $\KK$ for $\KK=\RR,\CC$ or $\HH$.
The results are listed in the following table.
\begin{table}[H]
	\begin{center}
    \begin{tabular}{|c||c|c|c|c|}
    \hline
    $n$ & $\Cl_n$ & $\CCl_n$ & $\Cl_{n}^h$ & $\CCl_{n}^h$\\
%    \hline 
    \hhline{|=#=|=|=|=|}
    $0$ & $\RR$ & $\CC$ & $\HH$ & $\CC(2)$\\
    \hline
    $1$ & $\CC$ & $\CC\oplus\CC$ & $\CC(2)$ & $\CC(2)\oplus\CC(2)$\\
    \hline
    $2$ & $\HH$ & $\CC(2)$ & $\RR(4)$ & $\CC(4)$\\
    \hline
    $3$ & $\HH\oplus\HH$ & $\CC(2)\oplus\CC(2)$ & $\RR(4)\oplus\RR(4)$ & $\CC(4)\oplus\CC(4)$\\
    \hline
    $4$ & $\HH(2)$ & $\CC(4)$ & $\RR(8)$ & $\CC(8)$\\
    \hline
    $5$ & $\CC(4)$ & $\CC(4)\oplus\CC(4)$ & $\CC(8)$ & $\CC(8)\oplus\CC(8)$\\
    \hline
    $6$ & $\RR(8)$ & $\CC(8)$ & $\HH(8)$ & $\CC(16)$\\
    \hline
    $7$ & $\RR(8)\oplus\RR(8)$ & $\CC(8)\oplus\CC(8)$ & $\HH(8)\oplus\HH(8)$ & $\CC(16)\oplus\CC(16)$\\
    \hline
    $8$ & $\RR(16)$ & $\CC(16)$ & $\HH(16)$ & $\CC(32)$\\
    \hline
    \end{tabular}
    \caption{Clifford algebras}\label{cliffalg}
\end{center}
\end{table}
The rest can be deduced from this table because the real and complex Clifford algebras are periodic in the following sense:
\begin{equation}\label{periodicity}
	\begin{aligned}
	\Cl_{n+8}&\cong\Cl_n\otimes_\RR\Cl_8\cong\Cl_n\otimes_\RR\RR(16),\\
	\CCl_{n+2}&\cong\CCl_n\otimes_\CC \CCl_2\cong \CCl_n\otimes_\CC \CC(2).
	\end{aligned}
\end{equation}
Consequently the quaternionic Clifford algebras and their complexifications are periodic as well.

Now we observe from the above table there is a symmetry between the real and quaternionic Clifford algebras:
\begin{proposition}\label{halfbott}
For $n\ge 0$, there are isomorphisms of associative $\RR$- and $\CC$-algebras
\begin{enumerate}[label=(\roman*)]
	\item $\Cl_{n+4}\cong \Cl_{n}^h\otimes_\RR\RR(2)$ and $\Cl_{n+4}^h\cong \Cl_n\otimes_\RR \RR(8)$;
	\item $\CCl_{n}^h\cong \CCl_n\otimes_\CC\CC(2)$.
\end{enumerate}
\end{proposition}
\begin{proof}
	The complex case (ii) follows directly from definition and \labelcref{fundamentalid}:
	\[
	\CCl_n^h=\Cl_n^h\otimes_\RR\CC=\Cl_n\otimes_\RR\HH\otimes_\RR\CC\cong\Cl_n\otimes_\RR\CC(2)=\CCl_n\otimes_\CC\CC(2).
	\]
	For (i), we claim that $\Cl_{n+4}\cong \Cl_n\otimes\Cl_4$. Then using $\Cl_4=\HH(2)$ we have
	\begin{align*}
	&\Cl_{n+4}\cong\Cl_n\otimes_\RR\HH(2)\cong\Cl_n\otimes_\RR\HH\otimes_\RR\RR(2)=\Cl_n^h\otimes_\RR \RR(2),\text{ and}\\
	&\Cl_{n+4}^h=\Cl_{n+4}\otimes_\RR\HH\cong\Cl_n\otimes_\RR\HH(2)\otimes_\RR\HH\cong \Cl_n\otimes_\RR \RR(8).
	\end{align*}
	Now to prove the claim, consider the linear map $f:\RR^{n+4}\to \Cl_n\otimes_\RR\Cl_4$ defined on the standard basis $e_1,\dots,e_{n+4}$ by
	\begin{equation}
	f(e_i)=
		\begin{cases}
			1\otimes e_i'' & \mbox{for } 1\le i\le 4\\
			e_{i-4}'\otimes e_1'' e_2'' e_3'' e_4''& \mbox{for } 5\le i\le n+4
		\end{cases}
	\end{equation}
	where $e_1',\dots,e_n'$ (resp. $e_1'',e_2'',e_3'',e_4''$) are the standard generators of $\Cl_n$ (resp. $\Cl_4$). It is straightforward to check that $f(e_i)^2=-1$ and $f(e_i)f(e_j)+f(e_j)f(e_i)=0$ for $i\neq j$. Therefore $f$ extends to an algebra morphism $\Cl_{n+4}\to \Cl_n\otimes_\RR\Cl_4$. Now notice $f$ maps onto a set of generators and the two algebras in question have the same dimension (both algebras are of dimension $2^{n+4}$), we conclude $\Cl_{n+4}\cong\Cl_n\otimes_\RR\Cl_4$.
\end{proof}
This proposition in particular shows $\Cl_{n+4}$ is \textbf{Morita equivalent} to $\Cl_n^h$ as associative $\RR$-algebras. More precisely, the functor
\begin{align*}
	\text{$\RR$-modules over $\Cl_n^h$}&\to \text{$\RR$-modules over $\Cl_{n+4}\cong\Cl_n^h\otimes\RR(2)$}\\
	V&\mapsto V\otimes\RR^2
\end{align*}
is an equivalence of categories. Similarly $\Cl_{n+4}^h$ is Morita equivalent to $\Cl_n$ as associative $\RR$-algebras, and $\CCl_n^h$ is Morita equivalent to $\CCl_n$ as associative $\CC$-algebras. We remark on several extra properties of the isomorphisms in \Cref{halfbott} that will be used later.
 \begin{remark}\label{extraremarks}
 \begin{enumerate}[label=(\roman*), leftmargin=*]
	\item These isomorphisms are compatible with the isometric embedding $i:\RR^n\to \RR^{n+1}$, $e\mapsto (e,0)$ in the sense that the isomorphisms commute with the induced embeddings of algebras $\Cl_n\xrightarrow{i_*} \Cl_{n+1}$, $\Cl_n^h\xrightarrow{i_*}\Cl_{n+1}^h$ and their complexifications. The complex case is clear, while the real case follows from that $f\circ i=i_*\circ f$.
	\item The quaternionic Clifford algebras and their complexifications inherit $\ZZ_2$-gradings from those on the real Clifford algebras by setting
	\[
	(\Cl_n^h)^\alpha=\Cl_n^\alpha\otimes_\RR \HH, (\CCl_n^h)^\alpha=(\Cl_n^h)^\alpha\otimes_\RR \CC
	\]
	where $\alpha=0,1$. Now the isomorphism in \Cref{halfbott}(ii) is a $\ZZ_2$-graded one where $\CCl_n$ is graded by $\CCl_n^\alpha=\Cl_n^\alpha\otimes_\RR\CC$ for $\alpha=0,1$. But the isomorphisms in \Cref{halfbott}(i) do \textit{not} preserve the $\ZZ_2$-gradings.
	\item In particular, the Morita equivalence between $\CCl_n^h$ and $\CCl_n$ is a $\ZZ_2$-graded one, in the sense that the equivalence $V_\CC\mapsto V_\CC\otimes_\CC\CC^2$ between the categories of $\CC$-modules restricts to an equivalence of the corresponding subcategories of $\ZZ_2$-graded modules.
\end{enumerate}
 \end{remark}

\subsubsection{$\ZZ_2$-grading} We can relate higher dimensional Clifford algebras to lower dimensional ones in two ways through their $\ZZ_2$-gradings. One is the following:
\begin{lemma}\label{inductive}
	Let $\{e_i\}$ be the standard orthonormal basis of $\RR^{n+1}$. Then the linear map $\RR^n\to \Cl_{n+1}$ given by
	$e_i\mapsto e_i e_{n+1}$ extends to an algebra isomorphism $\Cl_n\cong \Cl_{n+1}^0$ compatible with the isometric embeddings, i.e. the following diagram commutes:
	\[
	\begin{tikzcd}[row sep=small, column sep=small]
		\Cl_n\ar[d,"\cong"]\ar[r,"i_*"] & \Cl_{n+1}\ar[d,"\cong"]\\
		\Cl_{n+1}^0\ar[r,"i_*"] & \Cl_{n+2}^0
	\end{tikzcd}
	\]
	
\end{lemma}
\begin{proof}
	It is straightforward to check the map extends to an algebra map $\Cl_n\to\Cl_{n+1}$ whose image lies in $\Cl_{n+1}^0$. Now observe that the images of $e_i$ for $i\le n$ and $e_i e_j$ for $1\le i< j\le n$ generate $\Cl_{n+1}^0$. Therefore the algebra map $\Cl_n\to \Cl_{n+1}^0$ is surjective, but the two algebras in question have the same dimension. So the map is an isomorphism. The commutativity of the diagram follows directly from construction.
\end{proof}

The other involves $\ZZ_2$-graded tensor product.
\begin{definition}\label{Z2gradedtensorproduct}
Let $V=V^0\oplus V^1$ and $W=W^0\oplus W^1$ be $\ZZ_2$-graded $\KK$-vector spaces for $\KK=\RR$ or $\CC$. Then their \textbf{$\ZZ_2$-graded tensor product} $V\hat{\otimes}W$ is defined to be the vector space $V\otimes_\KK W$ equipped the $\ZZ_2$-grading $(V\hat{\otimes}W)^k =\bigoplus_{i+j\equiv k\bmod 2}(V^i\otimes W^j)$ where direct sum and tensor product are taken over the underlying field $\KK$. If further $V$ and $W$ are $\ZZ_2$-graded algebras, then $V\hat{\otimes}W$ is made into a $\ZZ_2$-graded algebra with the usual Koszul rule: $(v\hat{\otimes}w)\cdot(v'\hat{\otimes}w')=(-1)^{\deg w\deg v'}vv'\hat{\otimes}ww'$. If $V$ and $V'$ are $\ZZ_2$-graded $\KK$-modules over $\ZZ_2$-graded $\KK$-algebras $A$ and $A'$ respectively, then $V\hat{\otimes}_\KK V'$ is made into a $\ZZ_2$-graded module over $A\hat{\otimes}_\KK A'$ also by the usual Koszul rule.
\end{definition}

\begin{lemma}[{see \cite[Proposition 1.5]{LM89}}]\label{z2gradediso1}
Let $\{e_i\}$, $\{e_i'\}$ and $\{e_i''\}$ be standard orthonormal base of $\RR^{m+n}, \RR^m$ and $\RR^n$ respectively. Then the linear map $\RR^{m+n}\to \Cl_m\hat{\otimes}\Cl_n$ given by
\[
e_i\mapsto
\begin{cases}
	e_i'\hat{\otimes}1& \text{for $i\le m$} \\
	1\hat{\otimes}e_{i-m}''& \text{for $i>m$}
\end{cases}
\]
extends to an isomorphism of $\ZZ_2$-graded $\RR$-algebras $\Cl_{m+n}\cong\Cl_m\hat{\otimes}\Cl_n$ compatible with the isometric embeddings, i.e. the following diagram commutes:
\[
\begin{tikzcd}
	\Cl_m\hat{\otimes}\Cl_n\ar[r,"1\hat{\otimes}i_*"]\ar[d,"\cong"] & \Cl_m\hat{\otimes}\Cl_{n+1}\ar[d,"\cong"]\\
	\Cl_{m+n}\ar[r,"i_*"] & \Cl_{m+n+1}
\end{tikzcd}
\]
Upon tensoring with $\CC$, we get $\CCl_{m+n}\cong\CCl_m\hat{\otimes}\CCl_n$ as $\ZZ_2$-graded $\CC$-algebras compatible with isometric embeddings.
\end{lemma}
\begin{proof}
	It is straightforward to check the map extends to an algebra map $\Cl_{m+n}\to \Cl_m\hat{\otimes}\Cl_n$. Then notice the image of this map contains the generators of $\Cl_m\hat{\otimes}\Cl_n$ and the two algebras in question have the same dimension.
\end{proof}

	See also \Cref{generalz2gradediso} for a more general statement.

\begin{corollary}\label{z2gradediso}
	For all $m,n\ge 0$ there are isomorphisms of $\ZZ_2$-graded $\RR$-algebras
	\begin{enumerate}[label=(\roman*)]
		\item $\Cl_m\hat{\otimes}\Cl_n^h\cong \Cl_{m+n}^h$, and
		\item $\Cl_m^h\hat{\otimes}\Cl_n^h\cong \Cl_{m+n}\otimes \RR(4)$.
	\end{enumerate}
	In particular $\Cl_{m+n}$ is Morita equivalent to $\Cl_m^h\hat{\otimes}\Cl_n^h$ as $\ZZ_2$-graded algebras. Here the $\ZZ_2$-grading on $\Cl_{m+n}\otimes\RR(4)$ is given by $(\Cl_{m+n}\otimes \RR(4))^\alpha=\Cl_{m+n}^\alpha\otimes\RR(4)$ for $\alpha=0,1$.
\end{corollary}
\begin{proof}
	The first isomorphism is derived from the above lemma by tensoring with $\HH$. For (ii) recall the $\ZZ_2$-grading on $\Cl_n^h$ is given by $(\Cl_n^h)^\alpha=\Cl_n^\alpha\otimes\HH$. Therefore
	\[
	\Cl_m^h\hat{\otimes}\Cl_n^h=(\Cl_m\hat{\otimes}\Cl_n)\otimes \HH\otimes\HH\cong \Cl_{m+n}\otimes \RR(4).
	\]
	Since $\RR(4)$ does not contribute to the $\ZZ_2$-grading, the Morita equivalence is a $\ZZ_2$-graded one.
\end{proof}

\subsection{Quaternionic Clifford modules}
\subsubsection{Ungraded v.s. $\ZZ_2$-graded}
We are interested in $\ZZ_2$-graded modules, but $\ZZ_2$-graded modules and ungraded modules are closely related. Given a $\ZZ_2$-graded $\KK$-module $V$ over $\Cl_{n+1}$, $V^0$ is an ungraded $\KK$-module over $\Cl_{n+1}^0\cong \Cl_n$. Conversely given an ungraded $\KK$-module $V$ over $\Cl_n$, then 
\[
\Cl_{n+1}\otimes_{\Cl_{n+1}^0}V=(\Cl_{n+1}^0\otimes_{\Cl_{n+1}^0}V)\oplus (\Cl_{n+1}^1\otimes_{\Cl_{n+1}^0}V)
\] is a $\ZZ_2$-graded $\KK$-module over $\Cl_{n+1}$. It is clear these two constructions are inverses to each other and therefore the category of $\ZZ_2$-graded $\KK$-modules over $\Cl_{n+1}$ is isomorphic to the category of ungraded $\KK$-modules over $\Cl_n$. So even though we are interested in $\ZZ_2$-graded modules, we might as well study ungraded ones, which are easier to construct and classify using \Cref{cliffalg}.
\subsubsection{Grothendieck groups of modules}
For $\KK=\RR,\CC$ or $\HH$, let $\frM_n(\KK)$ denote the Grothendieck group of finite dimensional (ungraded) $\KK$-modules over $\Cl_n$ with respect to direct sum. Then for $n\ge 0$ we have
\begin{equation}\label{gradedungraded}
	\hat{\frM}_{n+1}(\KK)\cong\frM_n(\KK).
\end{equation}
Once all the groups $\frM_n(\KK)$ are known, we can recover all $\hat{\frM}_{n}(\KK)$ for $n\ge 1$. The special case $\hat{\frM}_0(\KK)$ is easy to work out: since $\Cl_0=\RR$ is concentrated in degree $0$, it has two inequivalent irreducible $\ZZ_2$-graded $\KK$-modules, both of which are one dimensional over $\KK$ but concentrated in degree $0$ and degree $1$ respectively. This means $\hat{\frM}_0(\KK)=\ZZ+\ZZ$.

Let us also denote by $\frM_n^h(\KK)$ (resp. $\hat{\frM}_n^h(\KK)$) the Grothendieck group of finite dimensional ungraded (resp. $\ZZ_2$-graded) $\KK$-modules over $\Cl_n^h$. Then we have
\[
\frM_n^h(\RR)=\frM_n(\HH),\text{ and } \hat{\frM}_n^h(\RR)=\hat{\frM}_n(\HH).
\]
There are more relations among these Grothendieck groups of modules derived from Morita equivalences of algebras.
\begin{proposition}\label{modulegroups}
	The isomorphisms of algebras in \Cref{halfbott} induce isomorphisms
	\begin{enumerate}[label=(\roman*)]
		\item $\frM_{n+4}(\RR)\cong\frM_n(\HH)$ and $\frM_{n+4}(\RR)\cong \frM_n(\HH)$ for $n\ge 0$;
		\item $\hat{\frM}_{n+4}(\RR)\cong\hat{\frM}_n(\HH)$ and $\hat{\frM}_{n+4}(\HH)\cong\hat{\frM}_n(\RR)$ for $n\ge 1$;
		\item $\hat{\frM}_n^h(\CC)\cong\hat{\frM}_n(\CC)$ for $n\ge 0$.
	\end{enumerate}
\end{proposition}
\begin{proof}
	It is clear (i) is a direct consequence of \Cref{halfbott} (i). Then (ii) follows from (i) by \labelcref{gradedungraded}. Finally (iii) follows from \Cref{halfbott} (ii) since the isomorphism therein is a $\ZZ_2$-graded one.
\end{proof}

We leave it to the reader to verify that the above isomorphisms are compatible with isometric embeddings. Then we get
\begin{corollary}
	The isomorphisms of algebras in \Cref{halfbott} induce further isomorphisms
	\begin{enumerate}[label=(\roman*)]
		\item $\hat{\frN}_{n+4}(\RR)\cong\hat{\frN}_n(\HH)$ and $\hat{\frN}_{n+4}(\HH)\cong\hat{\frN}_n(\RR)$ for $n\ge 1$;
		\item $\hat{\frN}_n^h(\CC)\cong\hat{\frN}_n(\CC)$ for $n\ge 0$.
	\end{enumerate}
\end{corollary}
Now from the knowledge of real and complex Clifford modules (see \cite[p.12]{ABS64}), we have
\begin{table}[H]
	\begin{center}
    \begin{tabular}{|c||c|c|c|c|c|c|c|c|}
    \hline
    $n\bmod 8$ & $1$ & $2$ & $3$ & $4$ & $5$ & $6$ & $7$ & $8$\\ \hhline{|=#=|=|=|=|=|=|=|=|}
    $\hat{\frM}_n(\RR)$ & $\ZZ$ & $\ZZ$ & $\ZZ$ & $\ZZ^2$ & $\ZZ$ & $\ZZ$ & $\ZZ$ & $\ZZ^2$\\
    \hline
    $\hat{\frM}_n(\HH)$ & $\ZZ$ & $\ZZ$ & $\ZZ$ & $\ZZ^2$ & $\ZZ$ & $\ZZ$ & $\ZZ$ & $\ZZ^2$\\
    \hline
    $\hat{\frN}_n(\RR)$ & $\ZZ_2$ & $\ZZ_2$ & $0$ & $\ZZ$ & $0$ & $0$ & $0$ & $\ZZ$\\
   	\hline
    $\hat{\frN}_n(\HH)$ & $0$ & $0$ & $0$ & $\ZZ$ & $\ZZ_2$ & $\ZZ_2$ & $0$ & $\ZZ$\\
    \hline
    \end{tabular}
    \quad
	\begin{tabular}{|c||c|c|}
    \hline
    $n\bmod 2$ & $1$ & $2$ \\ \hhline{|=#=|=|}
    $\hat{\frM}_n(\CC)$ & $\ZZ$ & $\ZZ^2$ \\
    \hline
    $\hat{\frM}_n^h(\CC)$ & $\ZZ$ & $\ZZ^2$ \\
    \hline
    $\hat{\frN}_n(\CC)$ & $0$ & $\ZZ$ \\
   	\hline
    $\hat{\frN}_n^h(\CC)$ & $0$ & $\ZZ$ \\
    \hline
    \end{tabular} 
    \end{center}
    \caption{Clifford modules}\label{cliffmod}
 \end{table}
 From \Cref{cliffalg}, all the algebras in question are semisimple (matrix algebras  over $\RR,\CC,\HH$ are simple), so their Grothendieck groups of (ungraded and consequently $\ZZ_2$-graded) modules are free abelian groups generated by inequivalent irreducible modules. This is reflected in \Cref{cliffmod}.
\subsubsection{Irreducible modules}\label{sec:irredmod} We now construct explicit generators for $\hat{\frM}_n(\KK)$ for $\KK=\RR,\CC$ or $\HH$. Let us begin with real and quaternionic modules. From \Cref{cliffalg}, whenever $n\not\equiv 3\bmod 4$, $\Cl_n$ and $\Cl_n^h$ are of the form $\KK(N)$. It is well-known that the matrix algebra $\KK(N)$ has a unique (up to equivalence) irreducible module $\KK^N$ by matrix multiplication. In view of \labelcref{gradedungraded}, we have described generators for $\hat{\frM}_n(\RR)$ and $\hat{\frM}_n(\HH)$ for $n\not\equiv 0\bmod 4$.
In the case $n\equiv 0\bmod 4$, we need the assistance of the (oriented) \textbf{volume element} $$\omega_n=e_1 e_2\cdots e_n\in \Cl_n,$$which enjoys the properties $\omega_n^2=(-1)^{n(n+1)/2}$ and $e\omega_n=(-1)^{n-1}\omega_n e$ for all $e\in \RR^n$ (see e.g. \cite[Proposition 3.3]{LM89}).

Now let $\Cl_4=\HH(2)$ act on $\HH^2$ by left matrix multiplication. Then since $(\omega_4)^2=1$, $\HH^2$ splits into a direct sum of $\pm 1$ eigenspaces $(1\pm \omega_4)\HH^2$ of $\omega_4$, denoted by $\HH_{\pm}$. Since $e\omega_4=-\omega_4 e$, multiplication by any $e\in\RR^4-0$ yields an isomorphism of real vector spaces $\HH_{+}\cong\HH_{-}$. So each of $\HH_{\pm}$ is of real dimension $4$. Further, since $\omega_4$ commutes with $\Cl_4^0$, $\HH_{\pm}$ are invariant under the action of $\Cl_4^0\cong\Cl_3$. 
Thus we may treat $\HH_\pm$ as $\Cl_3$-modules. Notice now $\omega_3$ is in the center of $\Cl_3$ and the action of $\omega_3$ on $\HH_{\pm}$ is through $\omega_4(=\omega_3 e_4)$, we conclude $\HH_{\pm}$ are inequivalent as $\Cl_3$-modules. The two inequivalent $\ZZ_2$-graded $\RR$-modules of $\Cl_4$ corresponding to the two inequivalent $\Cl_3$-modules $\HH_{\pm}$, denoted by $\Delta_{4,\RR}^{\pm}$, are tautologous: the underlying real vector spaces of $\Delta_{4,\RR}^\pm$ are both simply $\HH^2$, with $\ZZ_2$-gradings given by
$$\Delta_{4,\RR}^{\pm,0}=\HH_\pm, \quad\Delta_{4,\RR}^{\pm,1}=\HH_\mp.$$
As $\HH^2$ is an irreducible ungraded $\Cl_4$-module, $\Delta_{4,\RR}^\pm$ are irreducible as $\ZZ_2$-graded $\Cl_4$-modules.

Next observe that $\HH^2$ carries a natural right $\HH$-multiplication that commutes with the left matrix multiplication from $\Cl_4$. Since $\HH^{op}\cong\HH$ by conjugation, left and right modules of $\HH$ are no different. We can thus view $\HH^2$ as a left $\HH$-module and therefore an $\HH$-module of $\Cl_4$. Equipped with this $\HH$-module structure, $\Delta_{4,\RR}^\pm$ are enhanced into two inequivalent irreducible $\ZZ_2$-graded $\HH$-modules of $\Cl_4$, denoted by $\Delta_{4,\HH}^\pm$:
\begin{equation}
	\Delta_{4,\HH}^\pm=\Delta_{4,\RR}^\pm \text{ equipped with an $\HH$-module structure}.
\end{equation}

Similarly by considering the eigenspace decomposition of the volume element $\omega_8$ through the matrix multiplication of $\Cl_8=\RR(16)$ on $\RR^{16}$, we obtain two inequivalent irreducible $\ZZ_2$-graded $\RR$-modules $\Delta_{8,\RR}^\pm$ on whose even parts $\omega_8$ acts by $\pm 1$. The two inequivalent irreducible $\ZZ_2$-graded $\HH$-modules of $\Cl_8$, denoted by $\Delta_{8,\HH}^\pm$, can be obtained by considering the eigenspace decomposition of the volume element $\omega_8\otimes 1\in\Cl_8^h$ through the matrix multiplication of $\Cl_{8}^h=\HH(16)$ on $\HH^{16}$ as before. It is not hard to see
$$\Delta_{8,\HH}^\pm=\Delta_{8,\RR}^\pm\otimes_\RR\HH.$$

Using periodicity, we now have a complete description of irreducible $\ZZ_2$-graded $\RR$- and $\HH$-modules for the Clifford algebras. The story for $\CC$-modules is similar. Note that $\CC$-modules over $\Cl_n$ are the same as $\CC$-modules over $\CCl_n$ since $\CC$ is commutative. Now consider the left matrix multiplication of $\CCl_{2n}=\CC(2n)$ on $\CC^{2n}$. This is the unique (up to equivalence) irreducible ungraded $\CC$-module over $\CCl_{2n}$ and therefore gives rise to the unique (up to equivalence) irreducible $\ZZ_2$-graded $\CC$-modules over $\CCl_{2n+1}$. On the other hand observe the \textbf{complex volume element} $$\omega_{2n}^\CC=\mathbf{i}^n\omega_{2n}\in\CCl_{2n}$$ satisfies $(\omega_{2n}^\CC)^2=1$, so we obtain two $\ZZ_2$-graded $\CC$-modules $\Delta_{2n,\CC}^\pm$ for $\CCl_{2n}$ by setting
\[
\Delta_{2n,\CC}^{\pm,0}=(1\pm\omega_{2n}^\CC)\cdot\CC^{2n},\quad \Delta_{2n,\CC}^{\pm,1}=(1\mp\omega_{2n}^\CC)\cdot\CC^{2n}.
\]
These two $\ZZ_2$-graded modules $\Delta_{2n,\CC}^\pm$ are inequivalent because their even parts $\Delta_{2n,\CC}^{\pm,0}$ as $\CCl_{2n}^0\cong \CCl_{2n-1}$-modules are inequivalent. To see this, we simply note $\omega_{2n-1}$ is central in $\Cl_{2n-1}$ and the action of $\omega_{2n-1}$ on $\Delta_{2n,\CC}^{\pm,0}$, which is through the action of $\omega_{2n}=(-\mathbf{i})^{n}\omega_{2n}^\CC$, are different.
%Consequently we have $\ZZ_2$-graded $\CC$-modules for $\CCl_{2n}^h$
%\begin{equation}
%	\Delta_{2n,\CC}^{h,\pm}=\Delta_{2n,\CC}^\pm\otimes_\CC\CC^2
%\end{equation}

Now let us introduce notations for these irreducible Clifford modules.
\begin{definition}
	For $\KK=\RR$ or $\HH$, let $\Delta_{n,\KK}$ denote the unique (up to equivalence) irreducible $\ZZ_2$-graded $\KK$-module over $\Cl_n$ for $n\not\equiv 0$ mod $4$. For $n\equiv 0$ mod $4$, let $\Delta_{n,\KK}^\pm$ denote the two inequivalent irreducible $\ZZ_2$-graded $\KK$-module of $\Cl_n$, so that $\omega_n$ acts on $\Delta_{n,\KK}^{\pm,0}$ by $\pm 1$. We call these modules the \textbf{fundamental $\ZZ_2$-graded $\RR$- and $\HH$-modules} over the Clifford algebras. Similarly let $\Delta_{n,\CC}$ denote the unique (up to equivalence) irreducible $\ZZ_2$-graded $\CC$-modules over $\Cl_{n}$ for $n\not\equiv 0\bmod 2$. For $n\equiv 0\bmod 2$ let $\Delta_{n,\CC}^{\pm}$ denote the two inequivalent irreducible $\ZZ_2$-graded $\CC$-modules over $\Cl_n$, so that $\omega_n^\CC$ acts on $\Delta_{n,\CC}^{0,\pm}$ by $\pm 1$. We call these modules the \textbf{fundamental $\ZZ_2$-graded $\CC$-modules} over the Clifford algebras.
\end{definition}
The dimensions of these fundamental modules can be read off from \Cref{cliffalg} in the following two steps. First note the dimension of a fundamental $\ZZ_2$-graded module over $\Cl_n$ is twice the dimension of an ungraded irreducible module over $\Cl_{n}^0\cong\Cl_{n-1}$. Then the dimensions of ungraded irreducible modules can be read off directly from \Cref{cliffalg} (and periodicities). Let $d_{n,\RR},d_{n,\CC}$, and $d_{n,\HH}$ denote the \textit{real} dimensions of fundamental $\ZZ_2$-graded $\RR$-,$\CC$- and $\HH$-modules over $\Cl_n$ respectively. The results are listed below.
\begin{table}[H]
	\begin{center}
    \begin{tabular}{|c||c|c|c|c|c|c|c|c|}
    \hline
    $n$  & $1$ & $2$ & $3$ & $4$ & $5$ & $6$ & $7$ & $8$ \\ \hhline{|=#=|=|=|=|=|=|=|=|}
    $d_{n,\RR}$ & $2$ & $4$ & $8$ & $8$ & $16$ & $16$ & $16$ & $16$\\
    \hline
    $d_{n,\CC}$ & $4$ & $4$ & $8$ & $8$ & $16$ & $16$ & $32$ & $32$\\
    \hline
    $d_{n,\HH}$ & $8$ & $8$ & $8$ & $8$ & $16$ & $32$ & $64$ & $64$\\
    \hline
    \end{tabular}
    \end{center}
    \caption{Dimensions of fundamental $\ZZ_2$-graded modules}\label{cliffdim}
 \end{table}
The rest can be deduced from the recursive relation $d_{n+8,\KK}=2^{4} d_{n,\KK}$ for $\KK=\RR,\CC$ or $\HH$. These simple dimension counts reveal some relations among the fundamental $\ZZ_2$-graded modules. The fundamental $\ZZ_2$-graded $\RR$- and $\HH$-modules are related by forming $\ZZ_2$-graded tensor products as follows.
\begin{lemma}\label{z2identities}
	\begin{enumerate}[label=(\roman*)]
		\item As equivalence classes of $\ZZ_2$-graded $\KK$-modules over $\Cl_8\hat{\otimes}_\RR\Cl_n\cong \Cl_{n+8}$, we have
		\[
		\begin{cases}
			\Delta_{8,\RR}^+\hat{\otimes}_\RR\Delta_{n,\KK}=\Delta_{n+8,\KK} & (n\not\equiv 4\bmod 8)\\
			\Delta_{8,\RR}^+\hat{\otimes}_\RR\Delta_{n,\KK}^\pm=\Delta_{n+8,\KK}^\pm & (n\equiv 4\bmod 8)
		\end{cases}
		\]
		for $\KK=\RR$ or $\HH$.
		\item As equivalence classes of $\ZZ_2$-graded $\HH$-modules over $\Cl_n\hat{\otimes}_\RR\Cl_4\cong \Cl_{n+4}$, we have
		\[
		\begin{cases}
			\Delta_{n,\RR}\hat{\otimes}_\RR\Delta_{4,\HH}^+=\Delta_{n+4,\HH} & (n\not\equiv 4\bmod 8)\\
			\Delta_{n,\RR}^\pm\hat{\otimes}_\RR\Delta_{4,\HH}^+=\Delta_{n+4,\HH}^\pm & (n\equiv 4\bmod 8)
		\end{cases}
		\]
		\item As equivalence classes of $\ZZ_2$-graded $\RR$-modules over $\Cl_n^h\hat{\otimes}_\RR\Cl_4^h\cong \Cl_{n+4}\otimes_\RR\RR(4)$, we have
		\[
		\begin{cases}
			\Delta_{n,\HH}\hat{\otimes}_\RR\Delta_{4,\HH}^+=\Delta_{n+4,\RR}\otimes_\RR \RR^4 & (n\not\equiv 4\bmod 8)\\
			\Delta_{n,\HH}^\pm\hat{\otimes}_\RR\Delta_{4,\HH}^+=\Delta_{n+4,\RR}^\pm\otimes_\RR \RR^4 & (n\equiv 4\bmod 8)
		\end{cases}
		\]
	\end{enumerate}
\end{lemma}
\begin{proof}
	For (i), if $n\not\equiv 4\bmod 8$, by simple dimension counts $\Delta_{8,\RR}^+\hat{\otimes}_\RR\Delta_{n,\KK}$ has the same dimension as the unique (up to equivalence) $\ZZ_2$-graded $\KK$-module $\Delta_{n+8,\KK}$ over $\Cl_{n+8}$. Therefore $\Delta_{8,\RR}^+\hat{\otimes}_\RR\Delta_{n,\KK}$ must be equivalent to $\Delta_{n+8,\KK}$. If $n\equiv 4\bmod 8$, again by dimension counts $\Delta_{8,\RR}^+\hat{\otimes}_\RR\Delta_{n,\KK}$ must be equivalent to one of the two irreducible $\ZZ_2$-graded modules $\Delta_{n+8,\KK}^+$ or $\Delta_{n+8,\KK}^-$. To know which one, we observe the volume element $\omega_{n+8}$ coincides with $\omega_8\hat{\otimes}\omega_n$ under the isomorphism $\Cl_8\hat{\otimes}_\RR\Cl_n\cong \Cl_{n+8}$ (from \Cref{z2gradediso}). So the actions of $\omega_{n+8}$ on $(\Delta_{8,\RR}^+\hat{\otimes}_\RR\Delta_{n,\KK}^\pm)^0$ are $\pm 1$, and therefore $\Delta_{8,\RR}^+\hat{\otimes}_\RR\Delta_{n,\KK}^\pm=\Delta_{n+8,\KK}^\pm$. The same argument proves the rest.
\end{proof}

The fundamental $\RR$- and $\HH$-modules are related to fundamental $\CC$-modules through scalar extension and restriction. If $\KK\subset \LL$ are two of the skew-fields $\RR,\CC$ or $\HH$. Then we have two natural functors
\[
\begin{tikzcd}
	\KK\text{-vector spaces}\ar[r, shift left=1, "\Ind_\KK^\LL"]& \LL\text{-vector spaces}\ar[l, shift left=1,"\Res_\KK^\LL"]
\end{tikzcd}
\]
where $\Ind_\KK^\LL$ is $\LL\otimes_\KK-$ and $\Res_\KK^\LL$ is taking the underlying $\KK$-vector space. These functors restrict to the categories of $\KK$- and $\LL$-modules over the Clifford algebras. It is clear $\Res_\RR^\CC \circ \Res_\CC^\HH=\Res_\RR^\HH$ and $\Ind_\CC^\HH \circ \Ind_\RR^\CC=\Ind_\RR^\HH$.
\begin{lemma}\label{scalarchange}
\begin{enumerate}[label=(\roman*)]
	\item If $n\equiv 4\bmod 8$, then up to equivalence
%	\[
%	\Delta_{n,\HH}^\pm\xrightarrow{\Res_\CC^\HH}\Delta_{n,\CC}^{\mp}\xrightarrow{\Res_\RR^\CC} \Delta_{n,\RR}^\pm.
%	\]
\[
\begin{tikzcd}
	\Delta_{n,\HH}^\pm \ar[r,maps to, "\Res_\CC^\HH"]& \Delta_{n,\CC}^{\mp} \ar[r,maps to, "\Res_\RR^\CC"]& \Delta_{n,\RR}^\pm
\end{tikzcd}
\]
	\item If $n\equiv 0\bmod 8$, then up to equivalence
%	\[
%	\Delta_{n,\RR}^\pm\xrightarrow{\Ind_\RR^\CC}\Delta_{n,\CC}^{\pm}\xrightarrow{\Ind_\CC^\HH} \Delta_{n,\HH}^\pm.
%	\]
\[
\begin{tikzcd}
	\Delta_{n,\RR}^\pm \ar[r,maps to, "\Ind_\RR^\CC"]& \Delta_{n,\CC}^{\pm} \ar[r,maps to, "\Ind_\CC^\HH"]& \Delta_{n,\HH}^\pm
\end{tikzcd}
\]
\end{enumerate}	
\end{lemma}
\begin{proof}
	The proof is similar to the previous lemma by straightforward dimension counts and by looking at the actions of the real and complex volume elements. We note that $\omega_n^\CC=-\omega_n$ for $n\equiv 4\bmod 8$ and $\omega_n^\CC=\omega_n$ for $n\equiv 0\bmod 8$.
\end{proof}
%\begin{remark}
%	That $\Res_\RR^\HH (\Delta_{8k+4}^{\pm,\HH})=\Delta_{8k+4}^{\pm,\RR}$ and that $\Ind_\RR^\HH(\Delta_{8k}^{\pm,\RR})=\Delta_{8k}^{\pm,\HH}$ follow directly from our construction.
%\end{remark}
\subsection{Atiyah-Bott-Shapiro isomorphism}
\subsubsection{Multiplicative structures}
Recall that from \Cref{z2gradediso1} $\Cl_m\hat{\otimes}_\RR\Cl_n\cong\Cl_{m+n}$. Therefore, $\ZZ_2$-graded tensor product (over $\RR$) yields natural pairings
	\[
	\hat{\frM}_m(\RR)\otimes_\ZZ\hat{\frM}_n(\KK)\to \hat{\frM}_{m+n}(\KK)
	\]
for $\KK=\RR,\CC$ or $\HH$ making $\hat{\frM}_\bullet(\RR)$ into a graded ring and $\hat{\frM}_\bullet(\KK)$ a graded module over $\hat{\frM}_\bullet(\RR)$. Furthermore, since the isomorphisms in \Cref{z2gradediso1} are compatible with isometric embeddings, this pairings descend to pairings
\[
\hat{\frN}_m(\RR)\otimes_\ZZ\hat{\frN}_n(\KK)\to \hat{\frN}_{m+n}(\KK)
\]
making $\hat{\frN}_\bullet(\RR)$ into a graded ring and $\hat{\frN}_\bullet(\KK)$ a graded module over $\hat{\frN}_\bullet(\RR)$. Similarly $\ZZ_2$-graded tensor product over $\CC$ makes $\hat{\frM}_\bullet(\CC)$ and $\hat{\frN}_\bullet(\CC)$ into graded rings.
\begin{proposition}\label{freemodule}
	\begin{enumerate}[label=(\roman*)]
		\item For $\KK=\RR$ or $\HH$, $\Delta_{8,\RR}^+\hat{\otimes}_\RR-$ induces 8-fold periodicity isomorphisms of graded $\hat{\frM}_\bullet(\RR)$- and $\hat{\frN}_\bullet(\RR)$-modules
		\[
		\hat{\frM}_\bullet(\KK)\xrightarrow{\cong}\hat{\frM}_{\bullet+8}(\KK),\quad  \hat{\frN}_\bullet(\KK)\xrightarrow{\cong}\hat{\frN}_{\bullet+8}(\KK).
		\]
		\item $-\hat{\otimes}_\RR\Delta_{4,\HH}^+$ induces $4$-fold "periodicity" isomorphisms of graded $\hat{\frM}_\bullet(\RR)$- and $\hat{\frN}_\bullet(\RR)$-modules
		\[
		\hat{\frM}_\bullet(\RR)\xrightarrow{\cong}\hat{\frM}_{\bullet+4}(\HH),\quad  \hat{\frN}_\bullet(\RR)\xrightarrow{\cong}\hat{\frN}_{\bullet+4}(\HH);
		\]
		and
		\[
		\hat{\frM}_\bullet(\HH)\xrightarrow{\cong}\hat{\frM}_{\bullet+4}(\RR),\quad  \hat{\frN}_\bullet(\HH)\xrightarrow{\cong}\hat{\frN}_{\bullet+4}(\RR).
		\]
	\end{enumerate}
\end{proposition}
\begin{proof}
	These follow from \Cref{z2identities}.
\end{proof}
\begin{remark}
	Since $\Delta_{4,\HH}^+\hat{\otimes}_\RR\Delta_{4,\HH}^+=\Delta_{8,\RR}^+\otimes_\RR \RR^4$, the composition of the two 4-fold periodicities recovers the 8-fold periodicity.
\end{remark}

\begin{corollary}
	For $\KK=\RR,\CC$ or $\HH$, the residue classes of the fundamental $\ZZ_2$-graded $\KK$-modules of the Clifford algebras additively generate $\hat{\frN}_\bullet(\KK)$.
\end{corollary}
\begin{proof}
	The real and complex cases are well-known, see \cite[pp. 12-13]{ABS64}. The quaternionic case follows from the real case by the above proposition and \Cref{z2identities}.
\end{proof}
\begin{remark}
	For $n\equiv 0\bmod 4$ and $\KK=\RR$ or $\HH$, $\Delta_{n,\KK}^+\oplus_\RR \Delta_{n,\KK}^-=i^*\Delta_{n+1,\KK}$ (see \cite[Corollary 5.7]{ABS64}). So either $\Delta_{n,\KK}^+$ or $\Delta_{n,\KK}^-$ generates $\hat{\frN}_n(\KK)=\ZZ$. Similarly for $n\equiv 0\bmod 2$, $\Delta_{n,\CC}^+\oplus_\CC \Delta_{n,\CC}^-=i^*\Delta_{n+1,\CC}$.
\end{remark}
\subsubsection{Atiyah-Bott-Shapiro construction} Now let us quickly review Atiyah, Bott and Shapiro's construction that relates Clifford modules to K-theories.
For any $\ZZ_2$-graded $\RR$-module $V=V^0\oplus V^1$ over $\Cl_n$, we associate to it an element $\varphi(V)\in\KO(D^n,\partial D^n)$\footnote{Here, and from now on, we follow the convection that $\KO$ stands for $\KO^0$. Similarly $\KU$ (resp. $\KSp$) will stand for $\KU^0$ (resp. $\KSp^0$).} by setting
\[
\varphi(V):=[\mathcal{V}^0,\mathcal{V}^1;\mu]
\]
where $D^n$ is the unit disk in $\RR^n$ (with respect to the standard Euclidean metric), $\mathcal{V}^\alpha=D^n\times V^\alpha$ for $\alpha=0,1$, and $\mu$ is the Clifford module multiplication, i.e. at $e\in \RR^n$
\[
\mu_e: V\to V,\quad v\mapsto e\cdot v
\]
Note that $\mu_e$ interchanges $V^0$ and $V^1$, and satisfies $\mu_e^2=-\|e\|^2\cdot 1$ for $e\in \RR^n$. In particular, when restricted to $\partial D^n$, $\mu$ is a skew-adjoint isomorphism between bundles $\mathcal{V}^0$ and $\mathcal{V}^1$.

It is clear $\varphi(V)$ depends only on the $\ZZ_2$-graded equivalence class of $V$, and $\varphi$ preserves direct sums. Thus we have a homomorphism
\[
\varphi: \hat{\frM}_n(\RR)\to \KO(D^n,\partial D^n)\cong \KO^{-n}(\pt).
\]
If the $\ZZ_2$-graded module $V$ is the restriction of some $\ZZ_2$-graded module over $\Cl_{n+1}$ through the embedding $\Cl_n\subset \Cl_{n+1}$, then the isomorphism $\mu$ extends to $D^n$ by identifying $D^n$ with the upper hemisphere of $S^n=\partial D^{n+1}\subset \RR^{n+1}$. Therefore $\varphi$ descends to a graded homomorphism
\[
\varphi: \hat{\frN}_\bullet(\RR)\to \KO^{-\bullet}(\pt).
\]

Now the same construction applies to $\CC$- and $\HH$-modules over the Clifford algebras as well, yielding graded homomorphisms
\begin{align*}
	&\varphi^c: \hat{\frN}_\bullet(\CC)\to \KU^{-\bullet}(\pt)\\
	&\varphi^h: \hat{\frN}_\bullet(\HH)\to \KSp^{-\bullet}(\pt).
\end{align*}
It is proved in \cite{ABS64} that $\varphi$ and $\varphi^c$ are isomorphisms (of graded rings). In analogy, we will prove $\varphi^h$ is an isomorphism. For this, we need
\begin{lemma} $\varphi^h$ is a homomorphism of graded modules over the graded ring homomorphism $\varphi$. That is, the following diagram commutes
	\[
	\begin{tikzcd}
		\hat{\frN}_\bullet(\RR)\otimes_\ZZ\hat{\frN}_\bullet(\HH)\ar[r,"\hat{\otimes}"]\ar[d,"\varphi\otimes\varphi^h"] & \hat{\frN}_\bullet(\HH)\ar[d,"\varphi^h"]\\
		\KO^{-\bullet}(\pt)\otimes_\ZZ \KSp^{-\bullet}(\pt)\ar[r,"\boxtimes"] & \KSp^{-\bullet}(\pt)
	\end{tikzcd}
	\]
	where $\boxtimes$ is the module multiplication of $\KO$ on $\KSp$ induced by tensor product over $\RR$.
\end{lemma}
\begin{proof}
	The proof is essentially the same as that of \cite[Proposition 11.1]{ABS64} which asserts $\varphi$ is a graded ring homomorphism. We refer the reader to their proof.
\end{proof}
\subsubsection{Quaternionic ABS isomorphism}
Now we are ready to prove our quaternionic version of Atiyah-Bott-Shapiro isomorphism.
\begin{theorem}\label{ABSH}
	$\varphi^h:\hat{\frN}_\bullet(\HH)\to \KSp^{-\bullet}(\pt)$ is an isomorphism of graded modules over the graded ring isomorphism $\varphi: \hat{\frN}_\bullet(\RR)\cong\KO^{-\bullet}(\pt)$.
\end{theorem}
\begin{proof}
	Let us identify $\hat{\frN}_\bullet(\RR)$ with $\KO^{-\bullet}(\pt)$ through $\varphi$. From the previous lemma, $\varphi^h$ is a homomorphism of $\KO^{-\bullet}(\pt)$-modules. Now from \Cref{freemodule} and Bott's 4-fold periodicity, both $\hat{\frN}_\bullet(\HH)$ and $\KSp^{-\bullet}(\pt)$ are free modules of rank one over $\KO^{-\bullet}(\pt)$ in degrees $\ge 4$. So it suffices to prove $\varphi^h$ is an isomorphism in degrees $0$ and $4$ (the groups in question are zero in degrees $1,2,3$). In degree $0$, $\hat{\frN}_0(\HH)=\ZZ$ is generated by the residue class of $\Delta_{0,\HH}^+$, and it follows from construction that $\varphi^h(\Delta_{0,\HH}^+)$ is the trivial bundle $\HH\to \pt$ which generates $\KSp^{0}(\pt)=\ZZ$. In degree $4$, consider the commutative diagram
	\[
	\begin{tikzcd}
		\hat{\frN}_4(\HH) \ar[r,"\varphi^h"]\ar[d,"\Res_\RR^\HH"]& \KSp^{-4}(\pt)\ar[d,"\Res_\RR^\HH"]\\
		\hat{\frN}_4(\RR) \ar[r,"\varphi"]& \KO^{-4}(\pt)
	\end{tikzcd}
	\]
	By \Cref{scalarchange} $\Res_\RR^\HH(\Delta_{4,\HH}^+)=\Delta_{4,\RR}^+$, hence $\Res_\RR^\HH:\hat{\frN}_4(\HH)\to \hat{\frN}_4(\RR)$ is an isomorphism. From \cite{ABS64} $\varphi:\hat{\frN}_4(\RR)\to \KO^{-4}(\pt)$ is an isomorphism. Finally thanks to Bott \cite[3.14]{Bott} $\Res_\RR^\HH: \KSp^{-4}(\pt)\to \KO^{-4}(\pt)$ is an isomorphism. We conclude $\varphi^h$ is an isomorphism in degree $4$. This completes the proof.
\end{proof}

\subsection{More on modules}Previously we have been using the quaternionic Clifford algebras as a tool for studying $\HH$-modules over the (real) Clifford algebras. Now we put the quaternionic Clifford algebras in center stage and study their modules in their own rights.
\subsubsection{Notations}\label{notation}
Let us begin with some notation changes. Henceforth we will denote the fundamental $\ZZ_2$-graded $\RR$-, $\CC$-, $\HH$-module(s) over $\Cl_n$ by $\Delta_n$ (or $\Delta_n^\pm$ if $n\equiv 0\bmod 4$), $\Delta_{n,\CC}$ (or $\Delta_{n,\CC}^\pm$ if $n$ is even) and $\Delta_{n}^h$ (or $\Delta_{n}^{h,\pm}$ if $n\equiv 0\bmod 4$) respectively. That is we suppress $\RR$ from our notation and regard $\HH$-modules over $\Cl_n$ as $\RR$-modules over $\Cl_n^h$.

Recall that $\CCl_n^h\cong \CCl_n\otimes_\CC\CC(2)$ and
\[
\hat{\frM}_n(\CC)\to \hat{\frM}_{n}^h(\CC),\quad V_\CC\mapsto V_\CC\otimes_\CC \CC^2
\]
is an isomorphism. We denote the $\ZZ_2$-graded $\CC$-module(s) $\Delta_{n,\CC}\otimes_\CC \CC^2$ (or $\Delta_{n,\CC}^\pm\otimes_\CC \CC^2$) over $\CCl_n^h$ by $\Delta_{n,\CC}^h$ (or $\Delta_{n,\CC}^{h,\pm}$). We call $\Delta_{n}^h$ (or $\Delta_{n}^{h,\pm}$) and $\Delta_{n,\CC}^h$ (or $\Delta_{n,\CC}^{h,\pm}$) the \textbf{fundamental $\ZZ_2$-graded $\RR$- and $\CC$-modules over $\Cl_n^h$}.

We also introduce the following notations:
\[
\dd_n:=
\begin{cases}
	\Delta_n &\mbox{if $n\not\equiv 0\bmod 4$}\\
	\Delta_n^+&\mbox{if $n\equiv 0\bmod 8$}\\
	\Delta_n^-&\mbox{if $n\equiv 4\bmod 8$}
\end{cases}
\quad
\dd_{n,\CC}:=
\begin{cases}
	\Delta_{n,\CC} &\mbox{if $n$ is odd}\\
	\Delta_{n,\CC}^+&\mbox{if $n$ is even}
\end{cases}
\]
\[
\dd_n^h:=
\begin{cases}
	\Delta_n^h &\mbox{if $n\not\equiv 0\bmod 4$}\\
	\Delta_n^{h,+}&\mbox{if $n\equiv 0\bmod 8$}\\
	\Delta_n^{h,-}&\mbox{if $n\equiv 4\bmod 8$}
\end{cases}
\quad
\dd_{n,\CC}^h:=
\begin{cases}
	\Delta_{n,\CC}^h &\mbox{if $n$ is odd}\\
	\Delta_{n,\CC}^{h,+}&\mbox{if $n$ is even}
\end{cases}
\]
If $n$ is clear in the context, we will suppress $n$ and write, for example, $\dd$ for $\dd_n$. The choices are made so that
\begin{align*}
	\Ind_\RR^\CC(\dd_{8k})=\dd_{8k,\CC},&\quad \Res_\RR^\CC(\dd_{8k,\CC}^h)=\dd_{8k}^h,\\
	\Res_\RR^\CC(\dd_{8k+4,\CC})=\dd_{8k+4},&\quad \Ind_\RR^\CC(\dd_{8k+4}^h)=\dd_{8k+4,\CC}^h.
\end{align*}

\subsubsection{Right modules} The Clifford algebra $\Cl_n$ carries a \textbf{transpose} endomorphism $(-)^t:\Cl_n\to \Cl_n$ determined by
\[
(e_{i_1} e_{i_2}\cdots e_{i_k})^t=e_{i_k}\cdots e_{i_2} e_{i_1}.
\]
The transpose satisfies $(a^t)^t=a$ and $(ab)^t=b^t a^t$ for all $a,b\in \Cl_n$, and thus it is an isomorphism between $\Cl_n$ and its \textbf{opposite algebra} $\Cl_n^{op}$. Hence left and right modules over $\Cl_n$ are equivalent under the transpose as follows. Let $V$ be a left module over $\Cl_n$, we can define a right module $\widetilde{V}$ over $\Cl_n$ whose underlying vector space is $V$ on which the right $\Cl_n$-multiplication is defined by $v\cdot a:=(a^t)\cdot v$ for all $v\in V$, $a\in \Cl_n$.

Now the transpose on $\Cl_n$ extends to a transpose endomorphism on $\Cl_{n}^h$ by putting together the transpose on $\Cl_n$ and the conjugation on $\HH$:
\[
(a\otimes z)^t:=a^t\otimes \overline{z}
\]
for all $a\in \Cl_n$, $z\in \HH$. The extended transpose is an isomorphism of algebras $(-)^t:\Cl_n^h\cong(\Cl_n^h)^{op}$. Therefore left and right modules over $\Cl_n^h$ are also equivalent through the construction $V\mapsto \widetilde{V}$ as before.

\subsubsection{Bimodules} So there is nothing new by considering right modules. It is more interesting to consider bimodules. As we will see, the canonical bimodule over $\Cl_n^h$, which is $\Cl_n^h$ itself via left and right multiplications on itself, can be written as a tensor product of left and right modules. We focus on the dimensions $n\equiv 4,5,6,8\bmod 8$.

For $n=8k+4$, $\Cl_{8k+4}^h$ is of form $\RR(N)$ and recall $\Delta_{8k+4}^{h,\pm}$ are constructed from left matrix multiplication of $\RR(N)$ on $\RR^N$ and distinguished by the volume element. Here we think of $\Cl_n$ as the subalgebra $\Cl_n\otimes 1$ of $\Cl_{n}^h=\Cl_n\otimes_\RR\HH$ and thus think of the volume element as an element of $\Cl_{n}^h$. Let $\Delta_{8k+4}^h$ be the ungraded $\Cl_{8k+4}^h$-module that underlies either $\Delta_{8k+4}^{h,+}$ or $\Delta_{8k+4}^{h,-}$, then since the matrix representation is faithful, we have an injective homomorphism
\[
\Cl_{8k+4}^h\to \End_\RR(\Delta_{8k+4}^h)
\]
which must be an isomorphism since both algebras have the same dimension. It is well-known $\End_\RR(V)\cong V\otimes_\RR V^*$ as bimodules over $\End_\RR(V)$. Therefore we have an isomorphism of $\Cl_{8k+4}^h$-bimodules
\[
\Cl_{8k+4}^h\cong \Delta_{8k+4}^h\otimes_\RR (\Delta_{8k+4}^h)^*.
\]
Now equipping $\Delta_{8k+4}^h$ with either of the gradings from $\Delta_{8k+4}^{h,\pm}$, the induced $\ZZ_2$-gradings on $\Delta_{8k+4}^h\otimes_\RR (\Delta_{8k+4}^h)^*$ are the same, and it is clear the above isomorphism is a $\ZZ_2$-graded one. Next observe that $\widetilde{\Delta}_{8k+4}^h$ is equivalent to $(\Delta_{8k+4}^h)^*$ as $\ZZ_2$-graded right modules over $\Cl_{8k+4}^h$. Indeed for dimension reasons both are irreducible and the actions of the volume element on them are the same. Therefore we have an isomorphism of $\ZZ_2$-graded real $\Cl_{8k+4}^h$-bimodules
	\begin{equation*}\label{biiso4}
		\Cl_{8k+4}^h\cong\Delta_{8k+4}^h\otimes_\RR\widetilde{\Delta}_{8k+4}^h.
	\end{equation*}
	
For $n=8k+5$, the volume element $\omega_{8k+5}\in \Cl_{8k+5}^h$ is central and satisfies $\omega_{8k+5}^2=-1$. Thus $\omega_{8k+5}$ generates a central $\ZZ_2$-graded subalgebra $\CC^\omega$ of $\Cl_{8k+5}^h$ that is isomorphic to $\CC$. It follows we have an algebra isomorphism
\[
\Cl_{8k+5}^h=(\Cl_{8k+5}^{h})^0\oplus\omega_{8k+5}(\Cl_{8k+5}^{h})^0= (\Cl_{8k+5}^{h})^0\otimes_\RR\CC^\omega\cong\Cl_{8k+4}^h\otimes_\RR\CC
\]
where the $\ZZ_2$-grading on $\Cl_{8k+4}^h\otimes_\RR\CC$ is only contributed from the decomposition $\CC=\RR+\mathbf{i}\RR$. Similarly every $\ZZ_2$-graded $\RR$-module $V$ over $\Cl_{8k+5}^h$ can be written as
\[
V=V^0\oplus \omega_{8k+5} V^0= V^0\otimes_\RR\CC.
\]
In particular $\Delta_{8k+5}^h\cong\Delta_{8k+4}^h\otimes_\RR\CC$ and similarly $\widetilde{\Delta}_{8k+5}^h\cong\widetilde{\Delta}_{8k+4}^h\otimes_\RR\CC$. Then using the result in the $8k+4$ case, we get an isomorphism of $\ZZ_2$-graded real $\Cl_{8k+5}^h$-bimodules
\[
\Cl_{8k+5}^h\cong \Delta_{8k+5}^h\otimes_\CC\widetilde{\Delta}_{8k+5}^h.
\]
Here the $\ZZ_2$-grading on $\Delta_{8k+5}^h\otimes_\CC\widetilde{\Delta}_{8k+5}^h$ is given as follows. Write $\Delta_{8k+5}^h=(\Delta_{8k+5}^h)^0\otimes_\RR\CC$ then $\Delta_{8k+5}^h\otimes_\CC\widetilde{\Delta}_{8k+5}^h=(\Delta_{8k+5}^h)^0\otimes_\RR\widetilde{\Delta}_{8k+5}^h$ has a $\ZZ_2$-grading inherited from $\widetilde{\Delta}_{8k+5}^h$. Or equivalently one can write $\Delta_{8k+5}^h\otimes_\CC\widetilde{\Delta}_{8k+5}^h$ as $\Delta_{8k+5}^h\otimes_\RR (\widetilde{\Delta}_{8k+5}^h)^0$ and use the $\ZZ_2$-grading on $\Delta_{8k+5}^h$.

For $n=8k+6$ there is a strong analogy. The volume element $\omega_{8k+6}$ generates a central subalgebra $\CC^\omega$ of $(\Cl_{8k+6}^h)^0$, moreover $\omega_{8k+6}$ and $e_{8k+6}$ together generate a $\ZZ_2$-graded subalgebra $\HH^{e,\omega}$ that is isomorphic to $\HH$. Notice under the isomorphism $(\Cl_{8k+6}^h)^0\cong \Cl_{8k+5}^h$, the subalgebra $\CC^\omega$ here coincides with the one just discussed in the $8k+5$ case. Consider now $\Cl_{8k+4}^h$ as a subalgebra of $(\Cl_{8k+6}^h)^0$ via $\Cl_{8k+4}^h\cong (\Cl_{8k+5}^h)^0\subset (\Cl_{8k+6}^h)^0$. It is easy to check $\HH^{e,\omega}$ commutes with $\Cl_{8k+4}^h$ within $\Cl_{8k+6}^h$ and there is an isomorphism of $\ZZ_2$-graded algebras
\[
\Cl_{8k+6}^h=\Cl_{8k+4}^h\otimes_\RR\HH
\]
where the $\ZZ_2$-grading on $\Cl_{8k+4}\otimes_\RR\HH$ is only contributed from the decomposition $\HH=\CC+\CC\mathbf{j}$. For a $\ZZ_2$-graded module $V$ over $\Cl_{8k+6}^h$, we can first write it as
\[
V=V^0\oplus e_{8k+5}V^0=V^0\otimes_\CC\HH.
\]
Then by treating $V^0$ as a module over $\Cl_{8k+5}^h$, we can write it as $W\otimes_\RR\CC$ for some $\Cl_{8k+4}^h$-module $W$ and thus
\[
V=W\otimes_\RR \HH.
\]
In particular every $\Cl_{8k+6}^h$-module carries a right $\HH$-module structure that commutes with the left $\Cl_{8k+6}^h$-multiplication. Consequently $\widetilde{\Delta}_{8k+6}^h$ carries a left $\HH$-module structure commuting with the right $\Cl_{8k+6}^h$-multiplication. And similar as before, we have an isomorphism of $\ZZ_2$-graded real $\Cl_{8k+6}^h$-bimodules
\[
\Cl_{8k+6}^h\cong\Delta_{8k+6}^h\otimes_\HH\widetilde{\Delta}_{8k+6}^h.
\]
Here $\otimes_\HH$ means equating the right $\HH$-action on $\Delta_{8k+6}^h$ with the left one on $\widetilde{\Delta}_{8k+6}^h$, and the $\ZZ_2$-grading on $\Delta_{8k+6}^h\otimes_\HH\widetilde{\Delta}_{8k+6}^h$ is given by writing $\Delta_{8k+6}^h=(\Delta_{8k+6}^h)^0\otimes_\CC\HH$ and then $\Delta_{8k+6}^h\otimes_\HH\widetilde{\Delta}_{8k+6}^h=(\Delta_{8k+6}^h)^0\otimes_\CC\widetilde{\Delta}_{8k+6}^h$ inherits a $\ZZ_2$-grading from $\widetilde{\Delta}_{8k+6}^h$.

Finally for $n=8k$, consider first $\CCl_{8k}^h$ which is of form $\CC(N)$. Then we have an isomorphism of $\ZZ_2$-graded complex $\Cl_{8k}^h$-bimodules $\CCl_{8k}^h\cong\Delta_{8k,\CC}^h\otimes_\CC\widetilde{\Delta}_{8k,\CC}^h$.
Since $\CCl_{8k}^h=\Cl_{8k}^h\oplus\Cl_{8k}^h$ as real bimodules over $\Cl_{8k}^h$, we may write
\[
\Cl_{8k}^h\cong \frac{1}{2}\Delta_{8k,\CC}^h\otimes_\CC\widetilde{\Delta}_{8k,\CC}^h.
\]

In summary we have proved:
\begin{proposition}\label{bimodule}
	For $n\equiv 0,4,5,6 \bmod 8$, there are isomorphisms of $\ZZ_2$-graded real $\Cl_{n}^h$-bimodules:
	\begin{align*}
		\Cl_{8k}^h&\cong \frac{1}{2}\Delta_{8k,\CC}^h\otimes_\CC\widetilde{\Delta}_{8k,\CC}^h\\
		\Cl_{8k+4}^h&\cong\Delta_{8k+4}^h\otimes_\RR\widetilde{\Delta}_{8k+4}^h\\
		\Cl_{8k+5}^h&\cong\Delta_{8k+5}^h\otimes_\CC\widetilde{\Delta}_{8k+5}^h\\
		\Cl_{8k+6}^h&\cong \Delta_{8k+6}^h\otimes_\HH\widetilde{\Delta}_{8k+6}^h
	\end{align*}
\end{proposition}

This proposition and its proof will be used in computing the index of certain Dirac operator, see \Cref{cliffindexcomputed}.

\section{Spin$^h$ vector bundles}\label{sec4}
In this chapter, we develop a topological theory of spin$^h$ vector bundles. In particular, we construct "Thom classes" for spin$^h$ vector bundles in real and complex K-theories and exhibit versions of Thom isomorphisms for spin$^h$ vector bundles. Using these "Thom classes", we prove a version of Riemann-Roch theorem for spin$^h$ maps which picks out a special characteristic class for spin$^h$ vector bundles analogous to the $\hat{A}$-class for spin vector bundles. Finally we compute the cohomology of the classifying space of (stable) spin$^h$ vector bundles.
\subsection{Spin$^h$ structures on vector bundles}
Recall the spin group $\Spin(n)$ is a subgroup of $\Cl_n^\times$, the multiplicative group of $\Cl_n$. Let $\Sp(1)$ be the group of unit quaternions, then we have a natural group homomorphism
\[
\Spin(n)\times \Sp(1)\to (\Cl_n^h)^\times=(\Cl_n\otimes_\RR\HH)^\times,
\]
whose kernel is the "diagonal" $\ZZ_2$ generated by $(-1,-1)$. By modding out the kernel, we obtain the group
\[
\Spin^h(n):=\Spin(n)\times \Sp(1)/\ZZ_2\subset (\Cl_n^h)^\times.
\]
Since $\Spin(n)\subset \Cl_n^0$, we see $\Spin^h(n)\subset(\Cl_n^h)^0$. From here we can see the representation theory of $\Spin^h(n)$ is closely related to that of $\Cl_n^h$. Let $V$ be a $\ZZ_2$-graded $\RR$-module (resp. $\CC$-module) over $\Cl_n^h$, then $V^0$ is a module over $(\Cl_n^h)^0$ and therefore a representation of $\Spin^h(n)$. If $V$ is irreducible over $\Cl_n^h$, then $V^0$ is irreducible over $\Spin^h(n)$ because $\Spin(n)$ contains a set of generators of $\Cl_n^0$ and $\Sp(1)$ contains a set of generators of $\HH$.
%\begin{proposition}
%	Let $V$ be an irreducible $\ZZ_2$-graded real (resp. complex) module over $\Cl_n^h$. Then $V^0$ is an irreducible real (resp. complex) representation of $\Spin^h(n)$. 
%\end{proposition}
%\begin{proof}
% Since $V$ is an irreducible $\ZZ_2$-graded module over $\Cl_n^h$, $V^0$ must be an irreducible module over $(\Cl_n^h)^0$, otherwise $V^0$ contains a non-trivial proper submodule $W^0$ which then extends to a non-trivial $\ZZ_2$-graded proper submodule $W=\Cl_n^h\otimes_{(\Cl_n^h)^0}W^0\subset V$. Now we note $\Spin(n)$ contains a set of generators of $\Cl_n^0$, namely $e_{i_1}e_{i_2}\cdots e_{i_k}$ for $k$ even, and meanwhile $\Sp(1)$ contains a set of generators of $\HH$. Therefore $\Spin^h(n)$ contains a set of generators of $(\Cl_n^h)^0$. This implies $V^0$, being irreducible over $(\Cl_n^h)^0$, is also irreducible over $\Spin^h(n)$.
%\end{proof}

But $\Spin^h(n)$ owns more irreducible representations than $\Cl_n^h$. For instance, through projections onto its two factors, $\Spin^h(n)$ admits two natural orthogonal representations
\begin{equation}\label{elementrep}
	\begin{aligned}
		&\Spin^h(n)\to \Spin(n)/\ZZ_2=\SO(n)\\
		&\Spin^h(n)\to \Sp(1)/\ZZ_2=\SO(3)
	\end{aligned}
\end{equation}
Thus irreducible representations of $\SO(n)$ and $\SO(3)$ also become irreducible representations of $\Spin^h(n)$. By contrast, $\Cl_n^h$ has only one or two irreducible representations.

Now putting the two projections in \labelcref{elementrep} together, we get a short exact sequence of groups:
\begin{equation}\label{fundseqspinh}
	1\to \ZZ_2\to \Spin^h(n)\to \SO(n)\times \SO(3)\to 1
\end{equation}
where $\ZZ_2$ corresponds to $\pm 1$ in $\Cl_n^h$. This means $\Spin^h(n)$ is a \textit{central} extension of $\SO(n)\times \SO(3)$ by $\ZZ_2$. Group extensions of this type are classified by $$H^2(\BSO(n)\times \BSO(3);\ZZ_2)=\{0, w_2,w_2',w_2+w_2'\}$$ where $w_2\in H^2(\BSO(n);\ZZ_2)$ and $w_2'\in H^2(\BSO(3);\ZZ_2)$ stand for the corresponding second Stiefel-Whitney classes. Clearly $\Spin^h(n)$ is the extension that corresponds to $w_2+w_2'$; the other three elements $0, w_2, w_2'$ correspond to $\ZZ_2\times \SO(n)\times \SO(3)$, $\Spin(n)\times \SO(3)$ and $\SO(n)\times \Sp(1)$ respectively.
\begin{definition}\label{spinhstructure}
	Let $E$ be an oriented Riemannian vector bundle of rank $n$, and let $P_{\SO}(E)$ denote its oriented frame bundle. A \textbf{spin$^h$ structure} on $E$ is a principal $\Spin^h(n)$-bundle $P_{\Spin^h}(E)$ and a map of principal bundles $P_{\Spin^h}(E)\to P_{\SO}(E)$ which is equivariant with respect to $\Spin^h(n)\to \SO(n)$ in \labelcref{elementrep}. Or equivalently, in view of \labelcref{elementrep} and \labelcref{fundseqspinh}, a spin$^h$ structure is a rank $3$ oriented Riemannian vector bundle $\mathfrak{h}_E$ with $w_2(\mathfrak{h}_E)=w_2(E)$.
	
	With a fixed choice of $\mathfrak{h}_E$ or $P_{\Spin^h}(E)$, we say $E$ is a \textbf{spin$^h$ vector bundle}. The bundles $\mathfrak{h}_E$ and $P_{\Spin^h}(E)$ are called the \textbf{canonical bundle} and the \textbf{structure bundle} of $E$ respectively.
\end{definition}

\begin{remark}
The primary obstruction to the existence of spin$^h$ structures is the fifth integral Stiefel-Whitney class $W_5$ (see \cite{MAAM}). There are non-trivial secondary obstructions.
\end{remark}

%We insist on including metrics in our discussion for it will be convenient later to construct Dirac operators.

\begin{definition}\label{spinhmanifold}
	A smooth manifold $M$ is a \textbf{spin$^h$ manifold} if its tangent bundle is equipped with a spin$^h$ structure. Spin$^h$ manifolds with boundary and spin$^h$ cobordism can be defined in the usual way.
\end{definition}
\begin{example}\label{existence}
	Every compact oriented Riemannian manifold of dimension $\le 7$ admits a spin$^h$ structure \cite{MAAM}. Every oriented Riemannian $4$-manifold (including non-compact ones) admits two natural spin$^h$ structures whose canonical bundles are the bundle of self-dual two forms and the bundle of anti-self-dual two forms.
\end{example}
\begin{example}\label{bundleproduct}
	Let $F$ be a spin vector bundle of rank $m$ over $Y$ and $E$ a spin$^h$ vector bundle of rank $n$ over $X$, then $F\times E$ is a spin$^h$ vector bundle over $Y\times X$ with canonical bundle $\mathfrak{h}_{F\times E}=\pi_X^*\mathfrak{h}_E$ where $\pi_X:Y\times X\to X$ is the projection onto $X$. Let $P_{\Spin}(F)$ denote the structural principal $\Spin(m)$-bundle associated to $F$, then the structure bundle $P_{\Spin^h}(F\times E)$ of $F\times E$ is derived from the principal bundle $P_{\Spin}(F)\times P_{\Spin^h}(E)$ through the natural homomorphism
	\[
	\Spin(m)\times\Spin^h(n)\to\Spin^h(m+n)
	\]
	induced from the isomorphism given in \Cref{z2gradediso}(i). 
\end{example}
%\begin{caution}
%	In order to avoid confusion, the letters $F,E$ will often be reserved to stand for spin and spin$^h$ vector bundles respectively.
%\end{caution}
\subsection{Quaternionic Clifford and $^h$spinor bundles}
Recall for a spin vector bundle $F\to Y$ of rank $m$, its \textbf{Clifford bundle} is defined to be the bundle of $\ZZ_2$-graded $\RR$-algebra
\begin{equation*}
	\Cl(F)=P_{\Spin}(F)\times_{\Ad}\Cl_m
\end{equation*}
with the natural inherited $\ZZ_2$-grading, where $P_{\Spin}(F)$ is the principal $\Spin(m)$-bundle associated to $F$ and $\Spin(m)$ acts on $\Cl_m$ through the adjoint representation
\[
\Ad: \Spin(m)\to \Aut(\Cl_m), \quad g\mapsto \Ad_g(x):=gxg^{-1}, \text{ for }x\in\Cl_m.
\]
Since $-1\in\ker\Ad$, the adjoint representation descends to a representation $\Ad:\SO(m)\to\Aut(\Cl_n)$. As such, the Clifford bundle only relies on the metric on $F$.
%Alternatively $\Cl(F)$ can be described as
%\[
%\Cl(F)=\left(\bigoplus_{r\ge 0} F^{\otimes r}\right)\Big/I(F),
%\]
%where $I(F)$ is the bundle of ideals, whose fiber at $y\in Y$ is the two-sided ideal $I(F_y)$ in $\left(\bigoplus_{r\ge 0} F^{\otimes r}\right)$, generated by elements $e\otimes e+\|e\|^2$ for $e\in F_y$. In particular $\Cl(F_y)$ is the Clifford algebra generated by $F_y$ with respect to the inner product on $F_y$.
%\begin{remark}
%	We caution the reader (see \cite{LM89}) that $\Cl(F)$ differs from the bundle
%\[
%\Cl_{\Spin}(F)=P_{\Spin}(F)\times_l\Cl_m
%\]
%where $l$ is the left multiplicaition of $\Spin(m)$ on $\Cl_m$ through the natural inclusion $\Spin(m)\subset\Cl_m^\times$.
%\end{remark}

\begin{definition}
The \textbf{quaternionic Clifford bundle} of a spin$^h$ vector bundle $E\to X$ of rank $n$ is the bundle of $\ZZ_2$-graded $\RR$-algebra
\begin{equation*}
	\Cl^h(E)=P_{\Spin^h}(E)\times_{\Ad^h} \Cl_n^h
\end{equation*}
with the natural inherited $\ZZ_2$-grading, where $\Spin^h(n)$ acts on $\Cl_n^h$ through the adjoint representation 
\[
\Ad^h: \Spin^h(n)\to \Aut(\Cl_n^h), \quad g\mapsto \Ad^h_g(x):=gxg^{-1}, \text{ for }x\in\Cl_n^h.
\]
% Note that since $\RR^n\subset\Cl^1_{n,\HH}\subset\Cl_n^h$, we have $E\subset\Cl_{\HH}^1(E)\subset\Cl^h(E)$.
\end{definition}
This adjoint representation can be lifted to the adjoint representation of $\Spin(n)\times\Sp(1)$ on $\Cl_n^h$, which is the tensor product of the adjoint representations of $\Spin(n)$ on $\Cl_n$ and $\Sp(1)=\Spin(3)$ on $\HH=\Cl_{3}^0$. So the lifted representation (and therefore $\Ad^h$ as well) descends to a representation of $\SO(n)\times \SO(3)$ which is the tensor product of the (descended) adjoint representations of $\SO(n)$ on $\Cl_n$ and $\SO(3)$ on $\Cl_3^0$. Therefore $\Cl^h(E)$ depends only on the metrics on $E$ and $\mathfrak{h}_E$, and
\[
\Cl^h(E)=\Cl(E)\otimes_\RR\Cl^0(\mathfrak{h}_E).
\]
So the construction of the quaternionic Clifford bundle does not really require a spin$^h$ structure. However, the presence of the spin$^h$ structure will allow us to construct interesting bundles of modules over the quaternionic Clifford bundle.

\begin{definition}
	Let $E\to X$ be a spin$^h$ vector bundle of rank $n$. A \textbf{real $^h$spinor bundle} of $E$ is a bundle of the form
\[
	S^h(E,V):=P_{\Spin^h}(E)\times_{\mu}V,
\]
where $V$ is a $\RR$-module over $\Cl_n^h$ and $\mu$ is the composition $\Spin^h(n)\subset(\Cl_n^h)^\times\to \Aut(V)$. Similarly a \textbf{complex $^h$spinor bundle} of $E$ is a bundle of the form
\[
S^h_{\CC}(E,V_\CC):=P_{\Spin^h}(E)\times_\mu V_\CC,
\]
where $V_\CC$ is a $\CC$-module over $\Cl_n^h$. If the module $V$ (or $V_\CC$) is $\ZZ_2$-graded, the corresponding bundle is said to be $\ZZ_2$-graded. It is clear (real or complex, ungraded or $\ZZ_2$-graded) $^h$spinor bundles are bundles of modules over the quaternionic Clifford bundle.
\end{definition}
\begin{example}[fundamental $\ZZ_2$-graded $^h$spinor bundle]\label{fundspinor}
	We denote the corresponding $\ZZ_2$-graded real $^h$spinor bundle constructed from the $\ZZ_2$-graded modules $\Delta_{n}^h$ (resp. $\Delta_{n}^{h,\pm}$ if $n\equiv 0$ mod $4$) by $\SSS^h(E)$ (resp. $\SSS^{h,\pm}(E)$). Similarly, $\SSS_{\CC}^h(E)$ (resp. $\SSS_{\CC}^{h,\pm}(E)$ if $n$ is even) denotes the $\ZZ_2$-graded complex $^h$spinor bundle that corresponds to $\Delta_{n,\CC}^h$ (resp. $\Delta_{n,\CC}^{h,\pm}$). We call them the \textbf{fundamental $\ZZ_2$-graded (real or complex) $^h$spinor bundles} of $E$. Every $\ZZ_2$-graded $^h$spinor bundle of $E$ can be decomposed into a direct sum of fundamental ones.
\end{example}

\subsection{Thom classes and Thom isomorphisms}
In \cite{ABS64}, $\varphi$ is upgraded, for each spin vector bundle $F\to Y$ of rank $n$, to a homomorphism $$\varphi_F:\hat{\frN}_n(\RR)\to \widetilde{\KO}(Th(F))$$ where $Th(F)$ is the Thom space of $F$. Similarly if $F$ is spin$^c$, we have a homomorphism
\[
\varphi^c_F:\hat{\frN}_n(\CC)\to \widetilde{\KU}(Th(F)).
\]
%If $F'\to Y'$ is another spin vector bundle of rank $m'$, then we have a commutative diagram
%\begin{equation}
%	\begin{tikzcd}
%	\hat{\frN}_m(\RR)\otimes\hat{\frN}_{m'}(\RR)	\arrow[r,"\hat{\otimes}"]\arrow[d,"\varphi_F\otimes\varphi_{F'}"]& \hat{\frN}_{m+m'}(\RR)\arrow[d,"\varphi_{F\times F'}"]\\
%	\widetilde{\KO}(Th(F))\otimes \widetilde{\KO}(Th(F'))\arrow[r,"\boxtimes"]& \widetilde{\KO}(Th(F\times F'))
%	\end{tikzcd}
%\end{equation}
%where $\boxtimes$ is the external product in $\KO$.

Now we upgrade $\varphi^h$ for each spin$^h$ vector bundle $E$ of rank $n$ to a homomorphism
\[
\varphi^h_E:\hat{\frN}_n(\HH)\to \widetilde{\KO}(Th(E))
\]
as follows.
%\begin{lemma}\label{isocliffbund}
%	Let $E\to X$ be a spin$^h$ vector bundle of rank $n$ and $F\to Y$ be a spin vector bundle of rank $m$. Suppose $F\times E$ is given the spin$^h$ structure as in \Cref{bundleproduct}. Then there is an isomorphism of $\ZZ_2$-graded $\RR$-algebra bundles
%	\begin{equation*}
%		\Cl_{\HH}(F\times E)\cong \pi_Y^*\Cl(F)\hat{\otimes}_\RR\pi_X^*\Cl^h(E)
%	\end{equation*}
%	extending the natural identification of vector bundles $F\times E= \pi_Y^* F\oplus_\RR \pi_X^* E$, 
%	where $\pi_Y,\pi_X$ are the projections from $Y\times X$ to $Y$ and $X$ respectively and $\hat{\otimes}_\RR$ means fiberwise $\ZZ_2$-graded tensor product.
%\end{lemma}
%\begin{proof}
%	This follows easily from our definition, \Cref{gradedCliff} and \Cref{bundleproduct}.
%\end{proof}
Let $D(E), \partial D(E)$ denote the (closed) unit disk and sphere bundle of $E$ respectively. Let $\pi:D(E)\to X$ be the bundle projection. 

For a $\ZZ_2$-graded $\RR$-module $V$ over $\Cl_n^h$, we have the associated $\ZZ_2$-graded $^h$spinor bundle $S^h(E,V)$. Then the pull-backs of the degree $0$ and degree $1$ parts of $S^h(E,V)$ are canonically isomorphic over $\partial D(E)$ via the map
\[
\mu_e: \left(\pi^*S^h\left(E,V^0\right)\right)_e\to\left(\pi^*S^h\left(E,V^1\right)\right)_e\
\]
given at $e\in \partial D(E)$ by $\mu_e(\sigma)=e\cdot\sigma$.
That is, Clifford multiplication by $e$ itself. Since $e\cdot e=-\|e\|^2=-1$, each map $\mu_e$ is an isomorphism. This defines a difference class
\[
\varphi^h_E(V):=\left[\pi^*S^h\left(E,V^0\right),\pi^*S^h\left(E,V^1\right);\mu\right]\in\KO(D(E),\partial D(E))=\widetilde{\KO}(Th(E)).
\]

Clearly $\varphi^h_E(V)$ depends only on the equivalence class of $V$. If $V$ is restricted from a $\ZZ_2$-graded module over $\Cl_{n+1}^h$, i.e. $[V]$ belongs to $i^*\hat{\frM}_{n+1}(\HH)$, then we may embed $E$ into $E\oplus\RR$, where $\RR$ is the trivialized bundle with a nowhere zero cross-section $e_{n+1}$ and a metric so that $e_{n+1}$ is of norm one. This way the $^h$spinor bundle $S^h(E,V)$ is contained in $S^h(E\oplus\RR,V)$ and $\pi^* S^h(E,V)$ extends to a bundle over $D(E\oplus\RR)$. Then 
\[
\tilde{\mu}_e(\sigma)=(e+(\sqrt{1-\|e\|^2}) e_{n+1})\cdot\sigma\quad \text{for }e\in D(E)\subset D(E\oplus\RR)
\]
extends the isomorphism $\mu$ on $\partial D(E)$ to an isomorphism on $D(E)$. This means $\varphi^h_E$ descends to a group homomorphism
\[
\varphi^h_E:\hat{\frN}_n(\HH)\to \widetilde{\KO}(Th(E)).
\]

From definition $\varphi^h_E$ is functorial with respect to pull-backs of spin$^h$ bundles and if $E$ is the trivial bundle over a point, then $\varphi^h_E$ coincides with the composition
\[
\hat{\frN}_n(\HH)\xrightarrow{\varphi^h}\KSp^{-n}(\pt)\xrightarrow{\Res_\RR^\HH}\KO^{-n}(\pt).
\]

%\begin{remark}\label{functorial}
%	$\overline{\varphi}^h_E$ is functorial with respect to pull-backs of spin$^h$ vector bundles. That is, if $f:X'\to X$ is a continuous map, let $Th(f): Th(f^*E)\to Th(E)$ denote the map between Thom spaces induced from the bundle map $f^*E\to E$ covering $f$, then we have a commutative diagram:
%	\[
%	\begin{tikzcd}[row sep=small, column sep=small]
%		&\hat{\frN}_n(\HH)\arrow[dl,"\overline{\varphi}^h_{E}"']\arrow[dr,"\overline{\varphi}^h_{f^*E}"]\\
%		\widetilde{\KO}(Th(E))\arrow[rr,"Th(f)^*"]&&\widetilde{\KO}(Th(f^*E))
%	\end{tikzcd}
%	\]
%\end{remark}
Moreover the module structure of $\hat{\frN}_*(\HH)$ over $\hat{\frN}_*(\RR)$ is compatible with the external product in real K-theory.
\begin{proposition}\label{multiplicative}
	Let $E\to X$ be a spin$^h$ vector bundle of rank $n$ and $F\to X$ a spin vector bundle of rank $m$. Suppose $F\times E$ is given the spin$^h$ structure as in \Cref{bundleproduct}. Then the following diagram commutes:
	\begin{equation}
	\begin{tikzcd}
	\hat{\frN}_m(\RR)\otimes\hat{\frN}_n(\HH)	\arrow[r,"\hat{\otimes}"]\arrow[d,"\varphi_F\otimes\varphi^h_{E}"]& \hat{\frN}_{m+n}(\HH)\arrow[d,"\varphi^h_{F\times E}"]\\
	\widetilde{\KO}(Th(F))\otimes \widetilde{\KO}(Th(E))\arrow[r,"\boxtimes"]& \widetilde{\KO}(Th(F\times E))
	\end{tikzcd}
\end{equation}
\end{proposition}
\begin{proof}
	The proof is the same as that of \cite[Prop. 11.1]{ABS64}.
\end{proof}
As an application of the upgraded homomorphism, we have:
\begin{theorem}\label{fakeThomclass}
	Let $E\to X$ be a spin$^h$ vector bundle of rank $8k+4$. Then
	\begin{equation*}
		\varphi_E^h(\dd^h)\in \widetilde{\KO}(Th(E))
	\end{equation*}
	restricted to each fiber over $x\in X$ generates $\widetilde{\KO}(Th(E_x))\cong\KO^{-8k-4}(\pt)\cong\ZZ$. Moreover, multiplication by $\varphi_E^h(\dd^h)$ induces a Thom isomorphism
	\[
	\KO^*(X)[\frac{1}{2}]\xrightarrow{\cong}\widetilde{\KO}^*(Th(E))[\frac{1}{2}].
	\]
\end{theorem}

\begin{proof}
Since when restricted to the fiber $E_x$ over $x\in X$, $\varphi^h_{E_x}$ coincides with $\Res_\RR^\HH\circ \varphi^h$, the first assertion follows from $\Res_\RR^\HH(\dd_{8k+4}^h)=\dd_{8k+4}$ and $\Res_\RR^\HH\varphi^h=\varphi\Res_\RR^\HH$. The second assertion follows from a Mayer-Vietoris argument and that $\KO^{-8k-4}(\pt)[\frac{1}{2}]$ generates $\KO^*(\pt)[\frac{1}{2}]$ as a free $\KO^*(\pt)[\frac{1}{2}]$-module.
\end{proof}
\begin{remark}
	It is necessary to invert $2$ in order to obtain an isomorphism. Indeed, when $X$ is a point the map $\KO^*(\pt)\to\widetilde{\KO}^*(S^{8k+4})=\KO^{*-8k-4}(\pt)\cong\KO^{*-4}(\pt)$ is never an isomorphism since the 2-torsion on both sides are placed in different degrees.
\end{remark}

%\begin{remark}
%	For $E\to X$ a spin$^h$ vector bundle of rank $8k$, the $\KO$-class $\overline{\varphi}^h_E(\dd^h)$ in fact lifts to a $\KSp$-class using the intrinsic (right) quaternionic structure on $\dd_{8k}^h$. Then similarly $\overline{\varphi}^h_E(\dd^h)$ induces an isomorphism $\KO^*(X)[\frac{1}{2}]\xrightarrow{\cong}\widetilde{\KSp}^*(Th(E))[\frac{1}{2}]$. Again $2$ must be inverted in order for this homomorphism to be an isomorphism. 
%	
%%	It is for this Thom isomorphism that we mentioned in the introduction one can define Thom classes for spin$^h$ vector bundles in symplectic K-theory, however we will prefer to work with the $\KO$-class for rank $8k+4$ spin$^h$ vector bundles in this paper.
%\end{remark}

%One can apply the same construction to complex modules. In \cite{ABS64} it is shown:
%\begin{theorem}[\cite{ABS64}]
%	Let $F\to Y$ be a spin vector bundle of rank $2m$, then there is a homomorphism $$\varphi^c_F:\hat{\frN}_{2m}(\CC)\to \widetilde{\KU}(Th(F))$$
%	so that multiplication by the class $\varphi^c_F(\dd_\CC)$ induces a Thom isomorphism
%	\[
%	\KU^*(Y)\xrightarrow{\cong} \widetilde{\KU}^*(Th(F)).
%	\]\qed
%\end{theorem}

Analogously using complex modules over the quaternionic Clifford algebras we have:
\begin{theorem}\label{complexthomiso}
	Let $E\to X$ be a spin$^h$ vector bundle of rank $2n$, then there is a homomorphism
	\[
	\overline{\varphi}^c_E:\hat{\frN}_{2n}^h(\CC)\to \widetilde{\KU}(Th(E))
	\]
	so that $\overline{\varphi}^c_E(\dd^h_{\CC})\in \widetilde{\KU}(Th(E))$ restricted to each fiber over $x\in X$ is \textit{twice} the generator of $\widetilde{\KU}(Th(E_x))\cong\KU^{-2n}(\pt)$. Moreover, multiplication by $\overline{\varphi}^c_E(\dd^h_{\CC})$ induces a Thom isomorphism
	\[
	\KU^*(X)[\frac{1}{2}]\xrightarrow{\cong} \widetilde{\KU}^*(Th(E))[\frac{1}{2}].
	\]
\end{theorem}
\begin{proof}
	It suffices to show $\overline{\varphi}^c_E(\dd^h_{\CC})$ restricted to each fiber is twice the generator. Indeed, restricted to the fiber over $x\in X$, $\overline{\varphi}^c_{E_x}$ coincides with the composition
	\[
	\hat{\frN}_{2n}^h(\CC)\xrightarrow{\Res_{\CCl_{2n}}^{\CCl^h_{2n}}}\hat{\frN}_{2n}(\CC)\xrightarrow{\varphi^c}\KU^{-2n}(\pt),
	\]
	where the first map $\Res_{\CCl_{2n}}^{\CCl^h_{2n}}$ is the forgetful homomorphism induced by restricting the $\CCl_{2n}^h$-action to its subalgebra $\CCl_{2n}=\CCl_{2n}\otimes 1\subset \CCl_{2n}^h=\CCl_{2n}\otimes_\CC \CC(2)$. Recall $\dd^h_\CC=\dd_\CC\otimes_\CC\CC^2$ so  $\Res_{\CCl_{2n}}^{\CCl^h_{2n}}\dd^h_{\CC}=2\dd_{\CC}$. Therefore $\varphi^c_{E_x}(\dd^h_{\CC})=\varphi^c(2\dd_{\CC})=2\varphi^c(\dd_\CC)$; and $\varphi^c(\dd_\CC)$ generates $\KU^{-2n}(\pt)$ by \cite{ABS64}.
\end{proof}

These Thom isomorphisms motivate the following definition.

\begin{definition}
	Let $E$ be a spin$^h$ vector bundle of rank $n$. If $n\equiv 4$ (mod 8), then $$\dd^h(E):=\varphi_E^h(\dd^h)$$ is called the \textbf{weak KO-Thom class} of $E$. If $n \equiv 0$ (mod $2$), then $$\dd_\CC^h(E):=\overline{\varphi}^c_E(\dd_{\CC}^h)$$ is called the \textbf{weak KU-Thom class} of $E$. Note $\Ind_\RR^\CC(\dd^h(E))=\dd^h_\CC(E)$ for $n\equiv 4$ (mod 8).
\end{definition}

The weak Thom classes enjoy a nice multiplicative property. 
\begin{proposition}\label{mult}
	Let $E$ be a spin$^h$ vector bundle of rank $8k+4$ and $F$ a spin vector bundle of rank $8l$. Denote by $\dd(F)=\varphi_F(\dd)$ the $\KO$-Thom class of $F$. Then
	\[
	\dd^h(F\times E)=\dd(F)\cdot\dd^h(E).
	\]
	Similarly, denote by $\dd_\CC(F)=\varphi^c_F(\dd_\CC)$ the $\KU$-Thom class of $F$, then
	\[
	\dd^h_{\CC}(F\times E)=\dd_\CC(F)\cdot\dd^h_\CC(E).
	\]
\end{proposition}
\begin{proof}
	The first identity follows from \Cref{multiplicative}, the second follows from a similar property for $\varphi^c$ and $\overline{\varphi}^c$.
\end{proof}

The Chern character of the weak KU-Thom class is calculated below. The Pontryagin character of $\dd^h(E)$ is given by $\ph(\dd^h(E))=\ch(\dd^h_\CC(E))$ since $\Ind_\RR^\CC(\dd^h(E))=\dd^h_\CC(E)$ for $n\equiv 4$ mod $8$.
\begin{proposition}\label{cherncharacterdefect}
	Let $E\to X$ be a spin$^h$ vector bundle of rank $2n$. Then
	\[
	\ch(\dd^h_\CC(E))=(-1)^n U_E\cdot \left(2\cosh\left(\frac{\sqrt{p_1\left(\mathfrak{h}_E\right)}}{2}\right)\hat{\mathfrak{A}}(E)^{-1}\right)
	\]
	where $U_E\in \widetilde{H}^{2n}(Th(E);\ZZ)$ is the singular cohomology Thom class of $E$ and $\hat{\mathfrak{A}}(E)$ is the \textbf{total $\hat{A}$-class} of $E$.
\end{proposition}
\begin{proof}
	Since the map $f: \BSpin(2n)\times \mathrm{BSp(1)}\to \BSpin^h(2n)$ induced by the double cover $\Spin(2n)\times \Sp(1)\to \Spin^h(2n)$ induces an isomorphism on rational cohomology, without loss of generality we may assume $E$ is spin and its canonical bundle $\mathfrak{h}_E$ is also spin. Then
	$\dd^h_\CC(E)=\dd_\CC(E)\otimes_\CC \xi$ where $\xi$ is the rank $2$ complex vector bundle associated to the principal structural $\mathrm{Sp}(1)$-bundle of $\mathfrak{h}_E$ via the fundamental representation of $\Sp(1)=\mathrm{SU}(2)$ on $\CC^2$. The proposition now follows from $\ch\left(\dd^h_\CC(E)\right)=\ch(\dd_\CC(E))\cdot \ch(\xi)$ by $\ch\left(\dd_\CC(E)\right)=(-1)^n U_E\cdot\hat{\mathfrak{A}}(E)^{-1}$ (see \cite{AH59}) and the lemma below.
\end{proof}
	
\begin{lemma}\label{cherncharacter}
Let $\xi$ be the complex vector bundle over $\mathrm{BSU(2)}$ associated to the the fundamental representation of $\mathrm{SU}(2)$ on $\CC^2$ and let $\gamma$ be the real vector bundle over $\mathrm{BSU}(2)$ associated to the rotation representation of $\mathrm{SU}(2)$ on $\RR^3$ via the double covering map $\mathrm{SU}(2)=\Spin(3)\to \SO(3)$. Then $\ch(\xi)=2\cosh\left(\frac{\sqrt{p_1(\gamma)}}{2}\right)$.
\end{lemma}
\begin{proof}
The fundamental representation is irreducible with weights $1,-1$. By spitting principle, we may write $c(\xi)=(1+x)(1-x)=1-x^2$. The complexification of the rotation representation is the adjoint representation of $\mathrm{SU}(2)$, which is irreducible with weights $2,0,-2$. Therefore we can write $c(\gamma\otimes_\RR\CC)=(1+2x)(1-2x)=1-4x^2$. Now $p_1(\gamma)=-c_2(\gamma\otimes_\RR\CC)=4x^2$, hence symbolically $x=\frac{\sqrt{p_1(\gamma)}}{2}$. So $\ch(\xi)=e^x+e^{-x}=2\cosh(x)=2\cosh\left(\frac{\sqrt{p_1(\gamma)}}{2}\right)$. This expression makes sense since hyperbolic cosine is an even function.
\end{proof}
\subsection{Riemann-Roch theorem for spin$^h$ maps}
%In this subsection, we prove a Riemann-Roch theorem for spin$^h$ maps. Along the way, we pick out a special characteristic class for spin$^h$ manifolds, whose role is analogous to the $\hat{A}$-class for spin manifolds.
\begin{definition}
	Let $M$ and $N$ be closed oriented smooth manifolds. A continuous map $f:M\to N$ is called a \textbf{spin$^h$ map} if there exists an oriented rank $3$ real vector bundle $\mathfrak{h}_f$ on $M$ so that $$w_2(M)+f^*w_2(N)=w_2(\mathfrak{h}_f).$$ The bundle $\mathfrak{h}_f$ is called the canonical bundle of the spin$^h$ map $f$.
\end{definition}

\begin{theorem}\label{RR}
\begin{enumerate}[label=(\roman*)]
    \item Let $\dim M\equiv \dim N \bmod 2$. Then a spin$^h$ map $f:M\to N$ induces a direct image homomorphism $f_!: \KU(M)\to \KU(N)$ such that
    \begin{equation*}\label{Gysinspinh}
    	\ch f_!(\xi)\cdot\hat{\mathfrak{A}}(TN)=f_*\left(\ch \xi\cdot 2\cosh\left({\frac{\sqrt{p_1\left(\mathfrak{h}_f\right)}}{2}}\right)\hat{\mathfrak{A}}(TM)\right)
	\end{equation*}
	where $TM,TN$ are the tangent bundles of $M,N$ respectively, and $$f_*:H^*(M;\QQ)\to H^*(N;\QQ)$$ is the Umkehr homomorphism.
    \item If moreover $\dim M-\dim N\equiv 4 \bmod 8$, then there is a direct image homomorphism $\tilde{f}_!:\KO(M)\to\KO(N)$ so that the following diagram commutes.
    \begin{equation*}
    	\begin{tikzcd}
    		\KO(M)\arrow[r,"\tilde{f}_!"]\arrow[d,"\Ind_\RR^\CC"]&\KO(N)\arrow[d,"\Ind_\RR^\CC"]\\
    		\KU(M)\arrow[r,"f_!"]&\KU(N)
    	\end{tikzcd}
    \end{equation*}
	\end{enumerate}
\end{theorem}

\begin{proof}
	Since only the homotopy class of $f$ is relevant to the theorem, we may assume $f$ is smooth. Let $g: M\to S^{2n}$ be a smooth embedding of $M$. Then $f:M\to N$ can be factored into a smooth embedding $\iota:f\times g: M\to N\times S^{2n}$ followed by a projection $\pi: N\times S^{2n}\to N$. Since $w_2(S^{2n})=0$, $\iota$ is a spin$^h$ map with $\mathfrak{h}_\iota=\mathfrak{h}_f$; meanwhile $\pi$ is a spin map, that is $\pi^*w_2(N)=w_2(N\times S^{2n})$. Suppose we have proved (i) for $\iota$. Then since $\pi$ is spin, by Riemann-Roch theorem for spin maps (from \cite{Hir95}, enhanced in \cite{ABS64}), $\pi$ incudes a homomorphism $\pi_!:\KU(N\times S^{2n})\to \KU(N)$ satisfying $\ch\pi_!(-)\cdot\hat{\mathfrak{A}}(TN)=\pi_*(\ch(-)\cdot\hat{\mathfrak{A}}(T(N\times S^{2n}))$. Hence $f_!=\pi_! \iota_!$ is as desired. So we can assume $f$ is an embedding. 
	
	Now let $E$ be the normal bundle of $M$ in $N$ whose rank is even, then $w_2(E)=w_2(N)+f^*w_2(M)=w_2(\mathfrak{h}_f)$ and thus $E$ is spin$^h$ with canonical bundle $\mathfrak{h}_f$. Identify a closed tubular neighborhood of $M$ with $D(E)$, then we have $\KU(D(E),\partial D(E))\cong \KU(N,N-M)$ and $H^*(D(E),\partial D(E);\QQ)\cong H^*(N,N-M;\QQ)$ by excision. Recall the Umkehr homomorphism $f_*$ is the composition
	\[
	H^*(M;\QQ)\xrightarrow{\cdot U_E} H^*(D(E),\partial D(E);\QQ)\cong H^*(N,N-M;\QQ)\to H^*(N;\QQ).
	\]
	Define $f_!$ to be the composition
	\[
	\KU(M)\xrightarrow{\cdot\dd^h_\CC(E)}\KU(D(E),\partial D(E))\cong \KU(N,N-M)\to\KU(N).
	\]
	Using \Cref{cherncharacterdefect} and the multiplicative property of the $\hat{A}$-class: $f^*\hat{\mathfrak{A}}(TN)=\hat{\mathfrak{A}}(TM\oplus E)=\hat{\mathfrak{A}}(TM)\cdot\hat{\mathfrak{A}}(E)$, we conclude $f_!$ satisfies (i).
	
	For (ii), using the same embedding trick so that $\pi$ has relative dimension divisible by $8$, and noting from \cite{Hir95} and \cite{ABS64} $\pi_!$ in this case lifts to a homomorphism between $\KO$-groups, we may again assume $f$ is an embedding. Now $E$ is of rank $8k+4$, as such $\dd^h_\CC(E)=\Ind_\RR^\CC(\dd^h(E))$. Define $\tilde{f}_!$ to be the composition
	\[
	\tilde{f}_!:\KO(M)\xrightarrow{\cdot\dd^h(E)}\KO(D(E),\partial D(E))\cong \KO(N,N-M)\to\KO(N).
	\]
	Then $\tilde{f}_!$ clearly is as required.
\end{proof}

\begin{remark}
	$f_!$ is independent of the embedding $M\hookrightarrow S^{2n}$ due to the multiplicative property (\Cref{mult}) of the weak KO-Thom class. Indeed had we chosen two different embeddings, we may find a common larger embedding. So we can assume $M^d\subset S^{d+8k+4}\subset S^{d+8k+8l+4}$, then the normal bundle of $S^{d+8k+4}$ in $S^{d+8k+8l+4}$ is spin of rank $8l$. By \Cref{mult} and that $\dd_{8l}\in\KO^{-8k}(\pt)$ is the Bott generator, we conclude $f_!$ is independent of the choice of embedding.
\end{remark}
\begin{definition}
	Let $M$ be a closed spin$^h$ manifold of even dimension. We define the \textbf{$\hat{A}^h$-class} of $M$ to be rational cohomology class
	\[
	\hat{\mathfrak{A}}^h(M):=2\cosh\left({\frac{\sqrt{p_1\left(\mathfrak{h}_M\right)}}{2}}\right)\hat{\mathfrak{A}}(M).
	\]
	This cohomology class has Chern-Weil form representatives once we put Riemannian connections on $TM$ and $\mathfrak{h}_M$. Such a Chern-Weil representative is called an \textbf{$\hat{A}^h$-form} of $M$. Note $\hat{A}^h$-forms can also be defined for non-closed spin$^h$ manifolds through Chern-Weil theory, but now they are not necessarily closed forms.
\end{definition}
\begin{corollary}\label{integrality}Let $M$ be a closed spin$^h$ manifold of dimension $n\equiv 0\bmod 2$ with canonical bundle $\mathfrak{h}_M$.
\begin{enumerate}[label=(\roman*)]
	\item Suppose $\xi$ is a complex vector bundle over $M$. Then $\int_M\ch \xi\cdot \hat{\mathfrak{A}}^h(M)$ is an integer.
	\item If further $n\equiv 0 \bmod 8$ and $\gamma$ is a real vector bundle over $M$, then $\int_M\ph \gamma\cdot \hat{\mathfrak{A}}^h(M)$ is an even integer.
\end{enumerate}
\end{corollary}
\begin{proof}
	For (i), apply \Cref{RR}(i) to the spin$^h$-map $f:M\to \pt$. Then we have
	\begin{equation*}
		\int_M\ch \xi\cdot \hat{\mathfrak{A}}^h(M)=\int_{\pt}\ch f_!\xi\in\ZZ.
	\end{equation*}
	For (ii), apply \Cref{RR}(ii) to the spin$^h$-map $f:M\to \pt \hookrightarrow S^4$. Then we have
	\begin{equation*}
		\int_M\ph \gamma\cdot \hat{\mathfrak{A}}^h(M)=\int_{S^4}\ph \tilde{f}_!\gamma\in 2\ZZ.
	\end{equation*}
	The asserted integralities follows from Bott's theory (see \cite{Hir95}).
\end{proof}
\begin{remark}
	These integrality results are first obtained by Mayer \cite{Mayer} in studying immersions of manifolds into spin manifolds. They are also used to construct non-spin$^h$ 8-manifolds \cite{MAAM}.
\end{remark}
\begin{definition}
	The integer $\hat{A}^h(M):=\int_M \hat{\mathfrak{A}}^h(M)$ is called the \textbf{$\hat{A}^h$-genus} of $M$, and for a complex vector bundle $\xi$, the integer number $\hat{A}^h(M,\xi):=\int_M\ch \xi\cdot \hat{\mathfrak{A}}^h(M)$ is called the \textbf{$\hat{A}^h$-genus of $M$ twisted by $\xi$}.
\end{definition}

%\begin{remark}[cf. \cite{Mayer}]
%	It follows from \Cref{integrality} that $\hat{A}^h(M)$ is an even integer when $\dim M$ is divisible by $8$. This is the $\spin^h$ counterpart of Rokhlin's theorem since Rokhlin's theorem can be restated as that the $\hat{A}$-genus of a smooth closed spin $4$-manifold is even. This form of Rokhlin's theorem is later generalized to all smooth closed spin $(8k+4)$-manifold by Atiyah and Hirzebruch. The original form of Rokhlin's theorem is also generalized to all smooth closed spin $(8k+4)$-manifold by Ochanine.
%\end{remark}

\begin{example}[self-dual and anti-self dual spin$^h$ structures]\label{HP1}
	Let $M$ be a closed oriented Riemannian 4-fold. We furnish $M$ into a spin$^h$ manifold by putting $\mathfrak{h}_M=\Lambda^+_M$ (resp. $\Lambda_M^-$) where $\Lambda^+_M$ (resp. $\Lambda_M^-$) is the bundle of self-dual (resp. anti-self-dual) two forms, and denote the resulting spin$^h$ manifold by $M_+$ (resp. $M_-$). Then since $p_1(\Lambda_M^\pm)=p_1(M)\pm 2e(M)$ (see e.g. \cite[pp. 195]{Walschap}) where $e(M)$ is the Euler class of $M$, we have
	\begin{align*}
		\hat{A}^h(M_\pm)&=\int_M (2+\frac{p_1(\Lambda_M^\pm)}{4})(1-\frac{p_1(M)}{24})\\
		&= \int_M\frac{p_1(M)}{6}\pm\frac{e(M)}{2}\\
		&= (\text{Signature}\pm\text{Euler characteristic})/2.
	\end{align*}
In particular $\hat{A}^h(\HH\PP^1_\pm)=\pm 1$, $\hat{A}^h(\CC\PP^2_+)=2$ and $\hat{A}^h(\CC\PP^2_-)=-1$.
\end{example}

\subsection{Characteristic classes of spin$^h$ bundles}\label{BSpinh}	

We calculate the cohomology of (the classifying space of) the stable spin$^h$ group (\Cref{mod2cohomology}), which serves as an input for applying Adams spectral sequence to analyze the spin$^h$ cobordism groups, especially at prime 2. The cohomology for unstable spin$^h$ groups can be obtained using the beautiful algebro-geometric method of Quillen \cite{Quillen}, however we do not pursue  it here.

To begin with, recall that $\Spin^h$ is a central extension of $\SO\times \SO(3)$ by $\ZZ_2$, which is classified by $w_2+w_2'\in H^2(\BSO\times \BSO(3);\ZZ_2)$. Therefore we have a pull-back diagram
\[
\begin{tikzcd}
 \BSpin^h \arrow[r]\arrow[d,"\pi"'] & PK(\ZZ_2,2)\arrow[d]\\
 \BSO\times \BSO(3)\arrow[r,"f"] & K(\ZZ_2,2)
\end{tikzcd}
\]
where $PK(\ZZ_2,2)\to K(\ZZ_2,2)$ is the path space fiberation, and $f$ is induced by $w_2+w_2'$. Let $i_2$ denote the generator of $H^2(K(\ZZ_2,2);\ZZ_2)\cong\ZZ_2$. It is well-known the mod 2 cohomology of $K(\ZZ_2,2)$ is a polynomial algebra generated by $i_2$ and $Sq^I(i_2)$ where $I$ runs over all multi-indices $(2^r,2^{r-1},\dots,1)$.
\begin{lemma}\label{monicity}
$f^*: H^*(K(\ZZ_2,2);\ZZ_2)\to H^*(\BSO\times \BSO_3;\ZZ_2)=\ZZ_2[w_2, w_3,\dots; w_2', w_3']$ is monic.
\end{lemma}
\begin{proof}
For oriented bundles $Sq^1 w_2=w_3$. Then inductively using $Sq^{n-1}w_n=w_{2n-1}+\text{decomposables}$ (see \cite[pp. 291]{Stong}), we get
\begin{align*}
	Sq^0(w_2+w_2')&=w_2+w_2'\\
	Sq^1(w_2+w_2')&=w_3+w_3'\\
	Sq^I(w_2+w_2')&=w_{2^{r+1}+1}+\text{decomposables}\quad (r\ge 1).
\end{align*}
It is clear these are algebraically independent, thus proving $f^*$ is monic.
\end{proof}

\begin{lemma}\label{firstdescription} $\pi^*: H^*(\BSO\times \BSO(3);\ZZ_2)\to H^*(\BSpin^h;\ZZ_2)$ maps the subalgebra $$\ZZ_2[w_i| i\ge 2, i\neq 2^{r+1}+1, r\ge 1]$$ isomorphically onto $H^*(\BSpin^h;\ZZ_2)$.
\end{lemma}
\begin{proof}	
Let $E$ denote the Serre spectral sequence for $\pi: \BSpin^h\to \BSO\times \BSO(3)$ and $E'$ that of $PK(\ZZ_2,2)\to K(\ZZ_2,2)$. The map $f$ induces a map $f^*: E'\to E$ between spectral sequences. Since $E$ is an $H^*(\BSO\times \BSO(3);\ZZ_2)$-module, one has an induced spectral sequence map
\[
\ZZ_2[w_i| i\ge 2, i\neq 2^{r+1}+1, r\ge 1]\otimes E'\to E
\]
by means of $f^*$ and module multiplication. This is an isomorphism on the second page by the calculations done in the proof of \Cref{monicity}. Therefore the proposition follows from Zeeman's comparison theorem and that the path space $PK(\ZZ_2,2)$ is contractible.
\end{proof}
\begin{remark}
	In $H^*(\BSpin^h;\ZZ_2)$, the classes $w_{2^{r+1}+1}$ are not identically zero but decomposable. For instance, using \Cref{mod2cohomology} below one can prove $w_9=w_2 w_7+w_3 w_6$. 
\end{remark}
Our next step is to apply Bockstein spectral sequence to recover the 2-local cohomology of $\BSpin^h$, so first of all we must understand the action of $Sq^1$. 
%Recall for oriented bundles $Sq^1(w_{2i})=w_{2i+1}$, so $\ZZ_2[w_{2i}, w_{2i+1}]$ is a subalgebra invariant under $Sq^1$. However, in the mod 2-cohomology of $\BSpin^h$, the class $w_{2^{r+1}+1}$ is not an algebraic generator, so we would like to replace $w_{2^{r+1}}$ by another indecomposable class of the same degree, i.e. a class of the form $(w_{2^{r+1}}+\text{decomposables})$, on which $Sq^1$ vanishes.

%The Wu class $\nu_{2^{r+1}}$ is known to be indecomposable (see \cite[pp. 315]{Stong}), we shall verify $Sq^1\nu_{2^{r+1}}=0$ in $H^*(\BSpin^h;\ZZ_2)$, therefore $\nu_{2^{r+1}}$ is exactly the class we are looking for.

\begin{proposition}\label{Wuclass}
In $H^*(\BSpin^h;\ZZ_2)$ we have $Sq^1 \nu_{2^{r+1}}=0$ for $r\ge 1$ where $\nu_i$ is the $i^{th}$ Wu class.
\end{proposition}
\begin{proof}
Let $w,\nu$ denote the total Stiefel-Whitney class, the total Wu class respectively, and let $Sq$ denote the total Steenrod square. They are related by Wu's relation $Sq(\nu)=w$. Suppose $\overline{U}$ is the Thom class of the stable normal bundle to the bundle in question, then $Sq(\overline{U})=\overline{w}\cdot \overline{U}$ where $\overline{w}$ is the total Stiefel-Whitney class of the stable normal bundle, satisfying $w\cdot \overline{w}=1$.

Applying $Sq$ to $\nu\cdot \overline{U}$ and using Cartan's formula we get
\[
Sq(\nu\cdot \overline{U})=Sq(\nu)\cdot Sq(\overline{U})=w\cdot \overline{w}\cdot \overline{U}=\overline{U}.
\]
Then since $\chi(Sq)$ is the inverse to $Sq$, where $\chi$ is the canonical involution of the Steenrod algebra, we get $\nu\cdot \overline{U}=\chi(Sq)\overline{U}$. Now from Adem's relation $Sq^2 Sq^{4k-1}=Sq^{4k}Sq^1$
we get
\begin{align*}
    (Sq^1\nu_{4k})\cdot \overline{U} &=Sq^1(\nu_{4k} \overline{U})= Sq^1 \chi(Sq^{4k})\overline{U}=\chi(Sq^1)\chi(Sq^{4k})\overline{U}\\
    &=\chi(Sq^{4k}Sq^1)\overline{U} = \chi(Sq^2 Sq^{4k-1}) \overline{U}=\chi(Sq^{4k-1})\chi(Sq^2)\overline{U}\\
    &= \chi(Sq^{4k-1})Sq^2 \overline{U}= \chi(Sq^{4k-1})(w_2 \overline{U}).
\end{align*}
Here we used $Sq^1 \overline{U}=0$ and $w_2=\overline{w}_2$ since the bundles in question are orientable. Next we note from \cite{Davis}
\[
\chi(Sq^{2^{r+1}-1})=Sq^{2^{r}}Sq^{2^{r-1}}\cdots Sq^2 Sq^1,
\]
therefore
\begin{align*}
    Sq^1 \nu_{2^{r+1}} \cdot \overline{U} &= \chi(Sq^{2^{r+1}-1})(w_2 \overline{U})\\
    &= Sq^{2^{r}}Sq^{2^{r-1}}\cdots Sq^2 Sq^1 (w_2 \overline{U})\\
    &= Sq^{2^{r}}Sq^{2^{r-1}}\cdots Sq^4 (Sq^1 \nu_4 \cdot \overline{U}).
\end{align*}

By Thom isomorphism, we are reduced to proving $Sq^1\nu_4=0$. For oriented bundles, $\nu_4=w_4+w_2^2$ and thus $Sq^1\nu_4=Sq^1 w_4=w_5$. But the integral fifth Stiefel-Whitney class vanishes for spin$^h$ bundles \cite[Corollary 2.5]{MAAM}, so its mod 2 reduction $w_5$ must also vanish for spin$^h$ bundles. This completes the proof.
\end{proof}

We now derive a better description of the mod 2 cohomology of $\BSpin^h$.
\begin{theorem}\label{mod2cohomology}
	$\pi^*: H^*(\BSO\times \BSO(3);\ZZ_2)\to H^*(\BSpin^h;\ZZ_2)$ is surjective, with kernel generated by $w_2+w_2'$, $w_3+w_3'$, and $Sq^1\nu_{2^{r+1}}$ for all $r\ge 1$. In particular, $\pi^*$ induces an isomorphism
	\[
	H^*(\BSpin^h;\ZZ_2)\cong H^*(\BSO;\ZZ_2)/(Sq^1 \nu_{2^{r+1}}, r\ge 1).
	\]
\end{theorem}
\begin{proof}
	\Cref{firstdescription} shows that $\pi^*$ restricted to $H^*(\BSO;\ZZ_2)$ is surjective. Then by \Cref{Wuclass} this restriction descends to a surjection
	\[
	H^*(\BSO;\ZZ_2)/(Sq^1 \nu_{2^{r+1}}, r\ge 1)\to H^*(\BSpin^h;\ZZ_2).
	\]
	Now it is well-known that Wu classes in degrees powers of $2$ are indecomposable, that is to say $\nu_{2^{r+1}}=w_{2^{r+1}}+\text{decomposables}$. Consequently $Sq^1\nu_{2^{r+1}}=w_{2^{r+1}+1}+\text{decomposables}$. From here and \Cref{firstdescription} we see the domain and target of the above map have the same dimension in each degree, forcing the map to be an isomorphism. 
\end{proof}
\begin{remark}
Using the same method, we can prove
	\begin{align*}
		H^*(\BSpin;\ZZ_2)&=H^*(\BSO;\ZZ_2)/(\nu_2,Sq^1\nu_{2}, Sq^1\nu_4, Sq^1\nu_8,\dots);\\
		H^*(\BSpin^c;\ZZ_2)&=H^*(\BSO;\ZZ_2)/(Sq^1\nu_{2}, Sq^1\nu_4, Sq^1\nu_8,\dots).
	\end{align*}
	However, the (mod $2$) cohomology of $\BSpin^k$ does not seem to follow the same pattern.
\end{remark}
\begin{corollary}\label{2-torsion}
$H\left(H^*(\BSpin^h;\ZZ_2),Sq^1\right)\cong\ZZ_2[w_{2}^2, w_{2k}^2, \nu_{2^{r+1}}|k\neq 2^j, r\ge 1]$. 
\end{corollary}
\begin{proof}
	From the above theorem, the mod 2 cohomology of $\BSpin^h$ is isomorphic to $$\ZZ_2[w_2, Sq^1 w_2; w_{2k}, Sq^1 w_{2k}; \nu_{2^{r+1}}|k\neq 2^j, r\ge 1]$$ whose cohomology with respect to $Sq^1$ can now be easily obtained by applying K\"{u}nneth theorem. The result is as claimed. 
\end{proof}
\begin{corollary}
All torsion in $H^*(\BSpin^h;\ZZ)$ has order $2$.
\end{corollary}
\begin{proof}
Since $H\left(H^*(\BSpin^h;\ZZ_2),Sq^1\right)$ is concentrated in even degrees, all higher Bocksteins vanish. Hence by Bockstein spectral sequence all 2-primary torsion of $H^*(\BSpin^h;\ZZ)$ has order $2$. Meanwhile we see $$H^*(\BSpin^h;\ZZ[1/2])\cong H^*(\BSO\times \BSO(3);\ZZ[1/2])$$ is torsion-free.
\end{proof}

At this point, we have a rather complete description of characteristic classes for spin$^h$ vector bundles. Putting torsion aside, the integral characteristic classes are the Pontryagin classes of the bundle together with the first Pontryagin class of its canonical bundle. The mod 2 characteristic classes are the Stiefel-Whitney classes subject to universal relations generated by $Sq^1 \nu_{2^{r+1}}=0$ for $r\ge 1$. Certain mod 2 classes admit integral lifts. The square of the even Stiefel-Whitney classes are lifted to the Pontryagin classes. The odd Stiefel-Whitney classes are lifted to their integral counterparts. Finally the Wu classes in degrees power of two except for $\nu_2$ all have integral lifts.

\section{Dirac operators on spin$^h$ manifolds}\label{sec5}
In this chapter, we develop a geometric theory of indices of Dirac operators on spin$^h$ manifolds with and without boundary.
\subsection{Dirac operator}
Let $M$ be a closed $\spin^h$ manifold of dimension $n$ with canonical bundle $\mathfrak{h}_M$. We choose, once and for all, a Riemannian connection on $P_{\SO}(\mathfrak{h}_M)$. Then $P_{\Spin^h}(TM)$ inherits a connection from the Levi-Civita connection on $P_{\SO}(TM)$ and the Riemannian connection on $P_{\SO}(\mathfrak{h}_M)$. Suppose $S$ is a $^h$spinor bundle of $TM$, then $S$ is a bundle of $\Cl^h(M)$-module, and consequently a bundle of $\Cl(M)$-module. Moreover, $S$ is equipped with a connection $\nabla^S$ induced from $P_{\Spin^h}(TM)$. As usual we define the Dirac operator $D:\Gamma(S)\to\Gamma(S)$ to be the first order elliptic differential operator
\begin{equation}\label{diracoperator}
	D:=\sum_{i=1}^n e_i\cdot\nabla^S_{e_i}
\end{equation}
where $\{e_i\}_{i=1}^n$ is a local orthonormal frame of $M$, and $\cdot$ means Clifford multiplication.

If $S=S^0\oplus S^1$ is $\ZZ_2$-graded, then $D$ clearly interchanges the two factors. Written in matrix form
\begin{equation*}
	D=\begin{pmatrix}
		0& D^1\\
		D^0 &0
	\end{pmatrix}
\end{equation*}
where $D^0:\Gamma(S^0)\to\Gamma(S^1)$ and $D^1: \Gamma(S^1)\to \Gamma(S^0)$. As usual the Dirac operator is formally self-adjoint, namely $(D^0)^*=D^1$ and $(D^1)^*=D^0$. In particular $\ker D^1=\coker D^0$.

Recall all $^h$spinor bundles are direct sums of the fundamental ones.
\begin{definition}
Let $M$ be a closed $\spin^h$ manifold of dimension $n$. We define its \textbf{fundamental $\ZZ_2$-graded real $^h$spinor bundle} to be
\[
\SSS^h(M):=
\begin{cases}
	\SSS^h(TM) &\mbox{if $n\not\equiv 0$ mod $4$}\\
	\SSS^{h,+}(TM) &\mbox{if $n\equiv 0$ mod $8$}\\
	\SSS^{h,-}(TM) &\mbox{if $n\equiv 4$ mod $8$}
\end{cases}
\]
and denote the corresponding Dirac operator to be $\DDD_{M}^h$. If $\gamma$ is a Riemannian real vector bundle over $M$ with an orthogonal connection, we denote by $\DDD_{M}^h(\gamma)$ the Dirac operator on $\SSS^h(M)\otimes_\RR\gamma$ (defined by replacing the connection in \labelcref{diracoperator} with the tensor product connection).

Similarly we define the \textbf{fundamental $\ZZ_2$-graded complex $^h$spinor bundle} to be
\[
\SSS_{\CC}^h(M):=
\begin{cases}
	\SSS_{\CC}^h(TM) &\mbox{if $n$ is odd}\\
	\SSS_{\CC}^{h,+}(TM) &\mbox{if $n$ is even}
\end{cases}
\]
and denote the corresponding Dirac operator to be $\DDD_{M,\CC}^h$. If $\xi$ is a Hermitian complex vector bundle over $M$ with a unitary connection, we denote by $\DDD_{M,\CC}^h(\xi)$ the Dirac operator on $\SSS^h_\CC(M)\otimes_\CC\xi$.
\end{definition}

\begin{theorem}\label{complexindex}
	Let $M$ be a closed $\spin^h$ manifold of dimension $2n$ and $\xi$ a complex vector bundle over $M$. Then $$\ind\left(\DDD^{h}_{M,\CC}(\xi)\right)^0=\hat{A}^h(M,\xi).$$ In particular $\ind(\DDD_{M,\CC}^{h})^0=\hat{A}^h(M)$.
\end{theorem}
\begin{proof}
Let $x_1,\dots, x_n$ be virtual Chern roots of $M$ then from Atiyah-Singer index theorem we have
	\[
	\ind\left(\DDD^{h}_{M,\CC}(\xi)\right)^0=\int_M\left(\ch\left(\SSS_{\CC}^h\left(M\right)\right)^0-\ch\left(\SSS_{\CC}^h(M)\right)^1\right)\cdot\ch\xi\cdot\prod_{i=1}^n\frac{x_i}{1-e^{-x_i}}\cdot\frac{1}{1-e^{x_i}}.
	\]
	Meanwhile from \Cref{cherncharacterdefect} we have
	\[
	\ch\left(\SSS_{\CC}^h\left(M\right)\right)^0-\ch\left(\SSS_{\CC}^h(M)\right)^1=(-1)^n 2\cosh\left(\frac{\sqrt{p_1\left(\mathfrak{h}_M\right)}}{2}\right)\prod_{i=1}^n x_i\cdot\frac{\sinh(x_i/2)}{x_i/2}.
	\]
	Here we used that, when restricted to $M$, the weak KU-Thom class $\dd_\CC^h(TM)$ becomes $(\SSS_{\CC}^h(M))^0-(\SSS_{\CC}^h(M))^1$ and the singular cohomology Thom class $U_{TM}$ becomes the Euler class $\prod x_i$. The rest is a straightforward computation.
	\end{proof}
\subsection{$\Cl_{k}^h$-linear operator}
\begin{definition}
	By a \textbf{$\Cl_{k}^h$-Dirac bundle} over a Riemannian manifold $M$ we mean a real Dirac bundle $\mathfrak{S}$ over $M$, together with a left action $\Cl_{k}^h\to\Aut(\mathfrak{S})$ which is parallel and commutes with the multiplication by elements of $\Cl(M)$. Note as such the Dirac operator $\mathfrak{D}$ on $\mathfrak{S}$ commutes with the $\Cl_k^h$-action.

	A $\Cl_{k}^h$-Dirac bundle $\mathfrak{S}$ is said to be $\ZZ_2$-graded if it carries a $\ZZ_2$-grading $\mathfrak{S}=\mathfrak{S}^0\oplus\mathfrak{S}^1$ as a Dirac bundle, which is simultaneously a $\ZZ_2$-grading for the $\Cl_{k}^h$-action, that is
	\[
	(\Cl_{k}^h)^\alpha\cdot\mathfrak{S}^\beta\subseteq \mathfrak{S}^{\alpha+\beta}
	\]
	for all $\alpha,\beta\in\ZZ_2$. The Dirac operator $\mathfrak{D}$ with respect to the $\ZZ_2$-grading can be written as
\begin{equation*}
	\mathfrak{D}=
	\begin{pmatrix}
	0& \mathfrak{D}^1\\
	\mathfrak{D}^0& 0	
	\end{pmatrix} 
\end{equation*}
where $\mathfrak{D}^0:\Gamma(\mathfrak{S}^0)\to \Gamma(\mathfrak{S}^1)$ and $\mathfrak{D}^1:\Gamma(\mathfrak{S}^1)\to \Gamma(\mathfrak{S}^0)$. Therefore $$\ker\mathfrak{D}=\ker\mathfrak{D}^0\oplus\ker\mathfrak{D}^1$$ is a $\ZZ_2$-graded $\Cl_{k}^h$-module.
\end{definition}

\begin{definition}\label{cliffindexdef}
	Let $\mathfrak{S}$ be a $\ZZ_2$-graded $\Cl_{k}^h$-Dirac bundle over a closed manifold. The \textbf{analytic index} $\ind^h_k(\mathfrak{D})$ of the Dirac operator $\mathfrak{D}$ on $\mathfrak{S}$ is the residue class
	\[
	\ind^h_k(\mathfrak{D}):=[\ker\mathfrak{D}]\in\hat{\frN}_k(\HH)\cong\KSp^{-k}(\pt).
	\] 
\end{definition}
\begin{example}[$\Cl_k^h$-ification]\label{clification}
	Let $S$ be any ordinary real $\ZZ_2$-graded Dirac bundle over a closed manifold $M$, and let $D$ be its Dirac operator. We now consider an irreducible $\ZZ_2$-graded module $V$ over $\Cl_{k}^h$, and take the tensor product
	\[
	\mathfrak{S}=S\hat{\otimes}_\RR V
	\]
	where $V$ is considered as the trivialized bundle $V\times M\to M$. This bundle is naturally a $\ZZ_2$-graded $\Cl_{k}^h$-Dirac bundle. The associated Dirac operator $\mathfrak{D}$ on $\mathfrak{S}$ is simply $D\hat{\otimes} \mathrm{Id}_V$. Consequently we have that 
	\[
	\ker \mathfrak{D}=(\ker D)\hat{\otimes}V
	\]
	and in particular $\ker \mathfrak{D}^0=(\ker D^0\otimes V^0)\oplus(\ker D^1\otimes V^1)$. To determine the residue class $[\ker \mathfrak{D}]$ in $\hat{\frN}_{k}(\HH)=\hat{\frM}_{k}(\HH)/i^*\hat{\frM}_{k+1}(\HH)$ we recall the isomorphism $\hat{\frM}_{k}(\HH)\xrightarrow{\cong}\frM_{k-1}(\HH)$ by taking the degree zero part, so it suffices to determine $[\ker\mathfrak{D}^0]$ in $\frM_{k-1}(\HH)/i^*\frM_{k}(\HH)$. Since $V^0\oplus V^1$ is a $\Cl_{k}^h$-module, we have $[V^0]+[V^1]=0$ in $\frM_{k-1}(\HH)/i^*\frM_{k}(\HH)$. Therefore
	$$[\ker\mathfrak{D}^0]=(\dim_\RR \ker D^0-\dim_\RR\ker D^1)[V^0]=(\ind D^0)\cdot[V^0].$$
	It follows $[\ker\mathfrak{D}]=(\ind D^0)\cdot[V]$. Now that $[V]$ generates $\hat{\frN}_{k}(\HH)$, we conclude
	\begin{equation*}
		[\ker\mathfrak D]=
		\begin{cases}
			\ind D^0 &\mbox{if $k\equiv 0\bmod 4$}\\
			\ind D^0 (\bmod 2) &\mbox{if $k\equiv 5,6\bmod 8$}\\
			0&\mbox{otherwise}
		\end{cases}
	\end{equation*}
\end{example}

\begin{example}[The fundamental case]\label{cliffindexcomputed}
Let $M$ be a closed spin$^h$ manifold of dimension $n$. Consider the $^h$spinor bundle
\[
\slashed{\mathfrak{S}}(M):=P_{\Spin^h}(M)\times_l\Cl_n^h
\]
whose Dirac operator is denoted by $\slashed{\mathfrak{D}}$, where $\Spin^h(n)\subset(\Cl_n^h)^\times$ acts on $\Cl_n^h$ through the left multiplication. Clearly $\slashed{\mathfrak{S}}(M)$ admits a right $\Cl_n^h$-action that commutes with $\slashed{\mathfrak{D}}$, we can turn this into a left one by transpose. This way, $\slashed{\mathfrak{S}}(M)$ is a $\Cl_{n}^h$-Dirac bundle, and it follows from \Cref{bimodule} that
\begin{equation*}
		\slashed{\mathfrak{S}}(X)\cong
		\begin{cases}
			\SSS^h(X)\otimes_\RR\Delta_{n}^h&\mbox{if $n\equiv 4$ mod 8}\\
			\SSS^h(M)\otimes_\CC\Delta_{n}^h&\mbox{if $n\equiv 5$ mod 8}\\
			\SSS^h(M)\otimes_\HH\Delta_{n}^h&\mbox{if $n\equiv 6$ mod 8}\\
			\frac{1}{2}\SSS_{\CC}^h(M)\otimes_\CC\Delta_{n,\CC}^h&\mbox{if $n\equiv 0$ mod 8}
		\end{cases}
\end{equation*}
Note that the tilde's are removed for we have turned right $\Cl_n^h$-actions into left ones. Also we remark that for $n\equiv 6$ mod $8$, the tensor $\otimes_\HH$ is equating the right $\HH$-multiplications on $\SSS^h(M)$ and $\Delta_{n}^h$. From here we can extract the analytic index of $\slashed{\mathfrak{D}}$ as follows.

For $n=8k+4$, this is exactly the case of \Cref{clification} hence $\ind^h_{n}(\slashed{\mathfrak{D}})=\ind(\DDD_M^h)^0$. Now recall that $\Ind_\RR^\CC:\hat{\frN}_{8k+4}(\HH)\to\hat{\frN}_{8k+4}^h(\CC)$ is an isomorphism and $\Ind_\RR^\CC(\dd_{8k+4}^h)=\dd_{8k+4,\CC}^h$, we see $\ind(\DDD_M^h)^0=\ind(\DDD_{M,\CC}^h)^0=\hat{A}^h(M)$, and thus $\ind^h_n(\slashed{\mathfrak{D}})=\hat{A}^h(M)$.

For $n=8k+5$, the situation is similar to \Cref{clification}, we analogously have  $\ker\slashed{\mathfrak{D}}=\ker\DDD_M^h\otimes_\CC\Delta_{8k+5}^h$. Recall every $\ZZ_2$-graded module $V$ over $\Cl_{8k+5}^h$ can be written as $V^0\otimes_\RR\CC$. Therefore
\begin{align*}
	\ker\slashed{\mathfrak{D}}&=\ker\DDD_{M}^h\otimes_\CC\Delta_{8k+5}^h\\
	&= \left((\ker \DDD_M^h)^0\otimes_\RR\CC\right)\otimes_\CC\Delta_{8k+5}^h\\
	&\cong (\ker \DDD_M^h)^0\otimes_\RR\Delta_{8k+5}^h.
\end{align*}
Thus $\ind^h_n(\slashed{\mathfrak{D}})=\dim_\RR\ker(\DDD_{M}^h)^0=\dim_\CC \ker\DDD_{M}^h (\bmod 2)$.

For $n=8k+6$, we similarly have $\ker\slashed{\mathfrak{D}}\cong(\ker \DDD_M^h)^0\otimes_\CC\Delta_{8k+6}^h$ and therefore $\ind^h_n(\slashed{\mathfrak{D}})=\dim_\CC\ker(\DDD_{M}^h)^0=\dim_\HH \ker\DDD_{M}^h (\bmod 2)$.

Finally for $n=8k$, recall the forgetful morphism $\Res_\RR^\CC:\hat{\frN}_{8k}^h(\CC)\to\hat{\frN}_{8k}(\HH)$ is an isomorphism, and $\Delta_{8k}^h(\CC)$ generates $\hat{\frN}_{8k}^h(\CC)$. The argument of \Cref{clification} extended to the complex case yields $\ind^h_n(\slashed{\mathfrak{D}})=\frac{1}{2}\ind(\DDD_{M,\CC}^h)^0=\frac{1}{2}\hat{A}^h(M)$.
\end{example}

%\begin{remark}\label{lrtwist}
%	We have implicitly identified left and right modules over the Clifford algebras thanks to the isomorphisms $\Cl_n\cong\Cl_n^{op}$ for all $n$.
%\end{remark}
\subsection{Index of closed spin$^h$ manifold}
Let $M$ be a closed spin$^h$ manifold of dimension $n$ and let $\slashed{\mathfrak{D}}$ be the Dirac operator on $\slashed{\mathfrak{S}}(M)$ as in \Cref{cliffindexcomputed}. We define the \textbf{analytic index} of $M$ to be
\[
\aind(M):=\ind^h_n(\slashed{\mathfrak{D}})\in \KSp^{-n}(\pt).
\]
If $\gamma$ is a Riemannian real vector bundle over $M$ with an orthogonal connection, denote by $\slashed{\mathfrak{D}}(\gamma)$ the Dirac operator on $\slashed{\mathfrak{S}}(M)\otimes \gamma$. We define the \textbf{analytic index of $M$ twisted by $\gamma$} to be
\[
\aind(M,\gamma):=\ind^h_n\left(\slashed{\mathfrak{D}}(\gamma)\right)\in \KSp^{-n}(\pt).
\]

It follows from \Cref{cliffindexcomputed} that
\begin{theorem}\label{cliffindex}
	Let $M$ be a closed spin$^h$ manifold of dimension $n$. When $n\equiv 5$ or $6$ mod $8$, let $\mathcal{H}=\ker\DDD_{M}^h$ denote the kernel of the Dirac operator on the $\ZZ_2$-graded fundamental real $^h$spinor bundle of $M$. Then
	\begin{equation}
	\aind(M)=
		\begin{cases}
		\hat{A}^h(M)/2&\mbox{if $n\equiv 0$ mod $8$}\\
		\hat{A}^h(M)&\mbox{if $n\equiv 4$ mod $8$}\\
		\dim_\CC\mathcal{H} \bmod 2 &\mbox{if $n\equiv 5$ mod $8$}\\
		\dim_\HH\mathcal{H} \bmod 2 &\mbox{if $n\equiv 6$ mod $8$}
		\end{cases}
	\end{equation}
\end{theorem}

The same argument goes through with $\slashed{\mathfrak{S}}(M)$ replaced by $\slashed{\mathfrak{S}}(M)\otimes_\RR\gamma$, so we have
\begin{theorem}[cf. {\cite[\S 2. Theorem 7.13]{LM89}}]\label{twistedcliffindex}
	Let $M$ be a closed spin$^h$ manifold of dimension $n$ and $\gamma$ a real vector bundle over $M$ whose complexification is denoted by $\gamma_\CC$. When $n\equiv 5$ or $6$ mod $8$, let $\mathcal{H}_\gamma=\ker\DDD_{M}^h(\gamma)$ denote the kernel of the Dirac operator on the $\ZZ_2$-graded fundamental real $^h$spinor bundle of $M$ twisted by $\gamma$. Then
	\begin{equation}
	\aind(M,\gamma)=
		\begin{cases}
		\hat{A}^h(M,\gamma_\CC)/2&\mbox{if $n\equiv 0 \bmod 8$}\\
		\hat{A}^h(M,\gamma_\CC)&\mbox{if $n\equiv 4 \bmod 8$}\\
		\dim_\CC\mathcal{H}_\gamma \bmod 2 &\mbox{if $n\equiv 5 \bmod 8$}\\
		\dim_\HH\mathcal{H}_\gamma \bmod 2 &\mbox{if $n\equiv 6 \bmod 8$}
		\end{cases}
	\end{equation}
\end{theorem}

We can define the topological index of $M$ using the direct image map from Riemann-Roch theorem as follows. Let $f: M\to S^m$ be a smooth embedding of codimension $8k+4$. Then from \Cref{RR} we have a homomorphism
\[
\widetilde{f}_!: \KO(M)\to \KO(S^m).
\]
A closer look at the construction of $\widetilde{f}_!$ shows it maps into $\widetilde{\KO}(S^m)$ since the weak $\KO$-Thom class has virtual rank zero. We define the \textbf{topological index} of $M$ to be
	\[
	\tind(M):=q_!\widetilde{f}_!(1)\in \KSp^{-n}(\pt),
	\]
	where $q_!: \widetilde{\KO}(S^m)\to\KSp^{-n}(\pt)$ is the periodicity isomorphism. If $\gamma$ is a real vector bundle over $M$, we define the \textbf{topological index of $M$ twisted by $\gamma$} to be
	\[
	\tind(M,\gamma):=q_!\widetilde{f}_!(\gamma)\in \KSp^{-n}(\pt).
	\]

Of course topological index and analytic index coincide:
\begin{theorem}\label{spinhindexthm}
	Let $M$ be a closed spin$^h$ manifold. Then
	\[
	\aind(M)=\tind(M).
	\]
	If moreover $\gamma$ is a Riemannian real vector bundle over $M$ with an orthogonal connection, then $\aind(M,\gamma)=\tind(M,\gamma)$.
\end{theorem}
\begin{remark}
	In particular the analytic index does not depend on the chosen geometric data involved in its definition, such as the Riemannian metrics on $M$, $\mathfrak{h}_M$ and $\gamma$.
\end{remark}

We postpone the proof of this theorem to \Cref{sec7}.

\subsection{Boundary defect}
Let $M$ be a spin$^h$ manifold with boundary $\partial M$ of dimension $2n$, so that $\partial M$ has dimension $2n-1$. Assume the Riemannian metric on $M$ coincides with a product metric on $\partial M\times[0,1]$ in a neighborhood of the boundary. Recall $M$ carries a fundamental $\ZZ_2$-graded complex $^h$spinor bundle $\SSS_{\CC}^h(M)$ that admits a Dirac operator $\DDD^h_{M,\CC}$. The restriction of the bundle $(\SSS^h_{\CC}(M))^0$ onto $\partial M$ can be identified with the fundamental ungraded complex $^h$spinor bundle over $\partial M$
\begin{equation}\label{spinoronboundary}
	S^h_{\CC}(\partial M):=P_{\Spin^h}(\partial M)\times_\mu (\dd^h_{2n,\CC})^0
\end{equation}
where $(\dd^h_{2n,\CC})^0$ is viewed as a $\Cl^h_{2n-1}$-module through the isomorphism $\Cl^h_{2n-1}\cong(\Cl^h_{2n})^0$. Note $(\dd^h_{2n,\CC})^0$ is the unique, up to equivalence, irreducible ungraded $\CC$-module over $\Cl_{2n-1}^h$.

Choose, in a neighborhood of the boundary, a local framing $e_1,\dots,e_{2n}$ for $M$ so that $e_{2n}$ is the inward normal direction. In local terms
\[
(\DDD^h_{M,\CC})^0=e_{2n}\cdot(\nabla_{e_{2n}}+\sum_{i=1}^{2n-1} e_i e_{2n}\cdot \nabla_{e_i})=e_{2n}\cdot(\nabla_{e_{2n}}+D)
\]
where $D$, through the identification $(\SSS^h_{\CC}(M))^0|_{\partial M}=S^h_\CC(\partial M)$ is the Dirac operator on $S^h_\CC(\partial M)$. In particular $D$ is a first order self-adjoint elliptic operator. As such, $D$ has a discrete spectrum with real eigenvalues.

Two spectral invariants are attached to $D$: the multiplicity of the eigenvalue $0$, i.e. $\dim_\CC\ker D$, and the eta-invariant $\eta(0)$ where $\eta$ is the analytic continuation of $$\eta(s)=\sum_{\lambda\neq 0}(sign\lambda)|\lambda|^{-s}$$ where the $\lambda$ runs over the non-zero eigenvalues of $D$ (counted with multiplicity) and $sign\lambda=\pm 1$ is the sign of $\lambda$.

Abusing the notation, let us define $$\eta^h(\partial M):=\frac{1}{2}\left(\dim_\CC\ker D+\eta(0)\right)$$ and call it the \textbf{eta-invariant} of $\partial M$. Note that the definition of the eta-invariant involves only the Dirac operator on the fundamental ungraded complex $^h$spinor bundle, as such the eta-invariant can be defined for any odd dimensional spin$^h$ manifold even if it is not a spin$^h$ boundary. Also note it is clear the eta-invariant is additive with respect to disjoint union.

Now if we impose the following global boundary condition for $(\DDD^h_{M,\CC})^0$
\begin{equation}\label{boundarycondition}
	P(f|_{\partial M})=0, \quad f\in \Gamma\left(M,(\SSS^h_{\CC}(M))^0\right)
\end{equation}
where $P$ is the spectral projection of $D$ corresponding to eigenvalues $\ge 0$, then the Atiyah-Patodi-Singer index theorem \cite{APS} asserts:
\[
\ind(\DDD^h_{M,\CC})^0=\int_M \alpha_0(x)-\eta^h(\partial M).
\]
where $\alpha_0(x)$ is certain locally defined differential form on $M$. To determine $\alpha_0(x)$, it suffices to do a local computation, so we can assume $M$ is a spin manifold and its canonical bundle $\mathfrak{h}_M$ is reduced from a $\Sp(1)$-bundle through the covering map $\Sp(1)\to \SO(3)$. Now that $\dd^h_{2n,\CC}$, when viewed as a representation of $\Spin(2n)\times \Sp(1)$, is the tensor product $\dd_{2n,\CC}\otimes \CC^2$ where $\CC^2$ is considered the fundamental irreducible representation for $\Sp(1)=\mathrm{SU}(2)$, the $\ZZ_2$-graded complex $^h$spinor bundle $\SSS^h_{\CC}(X)$ can be written as $\SSS_{\CC}(X)\otimes \xi$ where $\SSS_{\CC}(M)$ is the usual $\ZZ_2$-graded complex spinor bundle for spin manifolds that corresponds to the complex Clifford module $\dd_{2n,\CC}$, and where $\xi$ is the rank $2$ complex vector bundle associated to the fundamental irreducible representation of $\Sp(1)$. This is exactly the twisted situation considered in \cite[4.3]{APS}, therefore $\alpha_0$ is the Chern-Weil form representative of $\ch(\xi)\hat{\mathfrak{A}}(M)$. By \Cref{cherncharacter} this form is identical to the $\hat{A}^h$-form $\hat{\mathfrak{A}}^h(M)$. Therefore we have proved:
\begin{theorem}\label{APS}
	Let $M$ be a $2n$-dimensional spin$^h$ manifold with boundary $\partial M$. Let $\DDD_{\CC,M}^h$ be the Dirac operator on the fundamental $\ZZ_2$-graded complex $^h$spinor bundle over $M$. Then the index of $(\DDD_{\CC,M}^h)^0$ with the global boundary condition \labelcref{boundarycondition} is given by
	\[
	\ind(\DDD_{\CC,M}^h)^0=\int_M \hat{\mathfrak{A}}^h(M)-\eta^h(\partial M).
	\]
\end{theorem}
	The corresponding statement naturally holds for Dirac operators with coefficients in a Hermitian vector bundle. 
\begin{theorem}\label{APStwisted}
	Let $M$ be a $2n$-dimensional spin$^h$ manifold with boundary $\partial M$. Suppose $\xi$ be a Hermitian vector bundle with a unitary connection and that, near the boundary, the metric and connection are constant in the normal direction. Let $\DDD_{\CC,M}^h(\xi)$ be the Dirac operator on the fundamental $\ZZ_2$-graded complex $^h$spinor bundle tensored with $\xi$ over $M$.
	Then the index of $(\DDD^h_{\CC,M}(\xi))^0$ with the global boundary condition
\[
	P_\xi(f|_{\partial M})=0, \quad f\in \Gamma\left(M,(\SSS^h_{\CC}(M)\otimes\xi)^0\right)
\]
is given by
\[	
	\ind(\DDD^h_{\CC,M}(\xi))^0=\int_M \ch(\xi)\cdot \hat{\mathfrak{A}}^h(M)-\eta^h(\partial M,\xi).
\]
Here $P_\xi$ is the spectral projection of the Dirac operator $D_\xi$ on $S^h_\CC(\partial M)\otimes_\CC\xi$ corresponding to eigenvalues $\ge 0$, and $\eta^h(\partial M,\xi)$ is defined similar to $\eta^h(\partial M)$ using $D_\xi$.
\end{theorem}	
	
\begin{remark}\label{extradivisibility}
	If $n\equiv 0\bmod 4$, i.e. $\dim M\equiv 0\bmod 8$, the $\ZZ_2$-graded complex $\Cl_{2n}^h$-module $\dd^h_{2n,\CC}$ carries further a right $\HH$-multiplication making it into a $\ZZ_2$-graded $\HH$-module. Therefore $\SSS^h_{\CC}(M)$ is a $\ZZ_2$-graded $\HH$-vector bundle on which $\DDD^h_{\CC,M}$ is $\HH$-linear, and the spectral projection operator $P$ is $\HH$-linear as well. As such the index of $(\DDD^h_{\CC,M})^0$ is an even integer. If further $\xi$ is the complexification $\gamma_\CC$ of a Riemannian real vector bundle $\gamma$ with an orthogonal connection, then $\DDD^h_{\CC,M}\otimes_\CC\gamma_\CC=\DDD^h_{\CC,M}\otimes_\RR \gamma$ is again a $\ZZ_2$-graded $\HH$-vector bundle on which $\DDD^h_{\CC,M}$ is $\HH$-linear. So the index of $(\DDD^h_{\CC,M}(\gamma_\CC))^0$ is an even integer. 
\end{remark}

\subsection{Index of $\ZZ_k$-spin$^h$ manifold} Let $\overline{M}$ be a $\ZZ_k$ spin$^h$ manifold of dimension $n\equiv 0\bmod 4$. We \textit{define} its \textbf{analytic $\ZZ_k$-index} to be
\[
\aind_{\ZZ_k}(\overline{M}):=\int_M \hat{\mathfrak{A}}^h(M)-k\cdot\eta^h(\beta M) \pmod k\in\ZZ_k
\]
If $\xi$ is a Hermitian complex vector bundle over $\overline{M}$ with a unitary connection, then the \textbf{analytic $\ZZ_k$-index of $\overline{M}$ twisted by $\xi$} is
\[
\aind_{\ZZ_k}(\overline{M},\xi):=\int_M \hat{\mathfrak{A}}^h(M)\ph(\xi)-k\cdot\eta^h(\beta M,\xi) \pmod k\in\ZZ_k.
\]
We note that even though the integral and the eta invariant in our definition involve geometric data, such as curvature, connection and spectrum of Dirac operator, the resulting analytic $\ZZ_k$-index is independent of those geometric data. This is because the $\ZZ_k$-index depends continuously on those geometric data but takes discrete values. Also note that from the previous section, the analytic index as defined is indeed the index of the Dirac operator on $M$ with Atiayh-Patodi-Singer boundary condition, provided that the geometric data on $\partial M$ is induced from the geometric data on $\beta M$ according to the identification $\partial M=\underbrace{\beta M\sqcup \beta M\sqcup\cdots\sqcup \beta M}_\text{$k$ times}$.

On the other hand we can define a topological $\ZZ_k$-index for $\overline{M}$ as follows. Let $\overline{S^m}$ be the $\ZZ_k$-$m$-sphere (\Cref{zksphereexample}). We can find $m$ large enough so that $\overline{M}$ embeds into $\overline{S^m}$ as a $\ZZ_k$ submanifold. The normal (vector) bundle of $\overline{M}$ in $\overline{S^m}$ carries an induced spin$^h$ structure. Then our construction in the Riemann-Roch theorem using weak $\KU$- and $\KO$-Thom classes yields:
\begin{enumerate}[label=(\roman*)]
	\item If $m\equiv n\bmod 2$, then we have a homomorphism 
	 $$f_!: \KU(\overline{M})\to \widetilde{\KU}(\overline{S^m}).$$
	 \item If moreover $m-n\equiv 4\bmod 8$, then we have a homomorphism
	 $$\widetilde{f}_!: \KO(\overline{M})\to \KO(\overline{S^m})$$
	 so that the following diagram commutes
	 \begin{equation*}
    	\begin{tikzcd}
    		\KO(\overline{M})\arrow[r,"\tilde{f}_!"]\arrow[d,"\Ind_\RR^\CC"]&\widetilde{\KO}(\overline{S^m})\arrow[d,"\Ind_\RR^\CC"]\\
    		\KU(\overline{M})\arrow[r,"f_!"]&\widetilde{\KU}(\overline{S^m})
    	\end{tikzcd}
    \end{equation*}
\end{enumerate}
Now an easy computation shows
\begin{lemma}\label{cohomologyofzksphere}
	Let $h^*$ be a generalized cohomology theory. Then we have the following short exact sequence
	\[
	0\to h^{*-1}(\pt)^{\oplus (k-1)}\oplus(h^{*-m}(\pt)\otimes \ZZ_k)\to \widetilde{h}^*(\overline{S^m})\to \mathrm{Tor}(h^{*-m+1}(\pt),\ZZ_k)\to 0
	\]
\end{lemma}
\begin{proof}
	Let $X$ be the manifold with boundary obtained from $S^m$ by removing $k$ disjoint open disks. Then $\partial X$ is the disjoint union of $k$ copies of $S^{m-1}$. Choose a base point $x\in \partial X$. The gluing map $X\to \overline{S^m}$ yields a map of triples $\pi:(X,\partial X,x)\to (\overline{S^m},S^{m-1},x)$ which in turn induces a commutative diagram
	\[
	\begin{tikzcd}[column sep=small, row sep=small]
		\cdots\ar[r]& h^*(\overline{S^m},S^{m-1})\ar[r]\ar[d,"\pi^*"] & \widetilde{h}^*(\overline{S^m})\ar[r]\ar[d,"\pi^*"] & \widetilde{h}^*(S^{m-1})\ar[r]\ar[d,"\pi^*"] &h^{*+1}(\overline{S^m},S^{m-1})\ar[r] \ar[d,"\pi^*"]&\cdots \\
		\cdots\ar[r]& h^*(X,\partial X)\ar[r] & \widetilde{h}^*(X) \ar[r]& \widetilde{h}^*(\partial X) \ar[r]&h^{*+1}(X,\partial X) \ar[r]& \cdots
	\end{tikzcd}
	\]
	where the rows are the long exact sequence of the triples. Here the reduced cohomology group $\widetilde{h}^*$ is the cohomology group relative to the base point $x$, and treated as a subgroup of the unreduced cohomology group. The lemma now follows easily from the following facts. First, the pair $(X,\partial X)$ is equivalent to $(\overline{S^m},S^{m-1})$ through $\pi$, so $\pi^*:h^*(X,\partial X)\to h^*(\overline{S^m},S^{m-1})$ is an isomorphism. Second, $\widetilde{h}^*(\partial X)=\widetilde{h}^*(S^{m-1})\oplus h^*(S^{m-1})^{\oplus (k-1)}$ and $\widetilde{h}^*(S^{m-1})\to \widetilde{h}^*(\partial X)$ is the diagonal map. Finally, $X$ is equivalent to the wedge sum of $(k-1)$-copies of $S^{m-1}$ (the boundary components of $X$ not containing $x$) and thus $\widetilde{h}^*(X)=\widetilde{h}^*(S^{m-1})^{\oplus (k-1)}$. Moreover the restriction map $\widetilde{h}^*(X)\to \widetilde{h}^*(\partial X)$ is identified with the injective homomorphism
		\[
		\widetilde{h}^*(S^{m-1})^{\oplus (k-1)}\to h^*(S^{m-1})^{\oplus (k-1)}\oplus\widetilde{h}^*(S^{m-1})
		\]
		that takes $(a_i)_{1\le i\le k-1}$ to $(a_i;-\sum a_i)$ (recall the definition of addition in homotopy group).
	\end{proof}

\begin{corollary}[cf. {\cite[Proposition 1.7]{modkindex}}{\cite[Proposition 3.1]{mod2refinement}}]\label{zksphere}
	We have
	\begin{align*}
		\widetilde{\KU}(\overline{S^m})&=\ZZ_k \quad \text{for }m\equiv 0\bmod 2,\\
		\widetilde{\KO}(\overline{S^m})&=\ZZ_k \quad \text{for }m\equiv 0\bmod 4.
	\end{align*}
	Moreover the map $$\Ind_\RR^\CC:\widetilde{\KO}(\overline{S^m})\to\widetilde{\KU}(\overline{S^m})$$ is an isomorphism for $m\equiv 0\bmod 8$, and multiplication by $2$ for $m\equiv 4\bmod 8$.
\end{corollary}
\begin{proof}
Apply the above lemma to $\KU$ and $\KO$ and use that the short exact sequence commutes with natural transformations between cohomology theories.
\end{proof}
\begin{remark}\label{bundlegeneratorofzksphere}
	For $m$ even, the generator of $\widetilde{\KU}(\overline{S^m})$ can be obtained as follows. From the commutative diagram in the proof of \Cref{cohomologyofzksphere}, we see the generator of $\widetilde{\KU}(\overline{S^m})$ is reduction mod $k$ of the generator of $\KU(X,\partial X)=\ZZ$. Meanwhile, by excision we have $\KU(X,\partial X)=\KU(S^m, k D^m)\cong \widetilde{\KU}(S^m)$. Therefore the generator of $\widetilde{\KU}(\overline{S^m})$ is simply the reduction mod $k$ of the (virtual) bundle $\dd_{m,\CC}$ over $S^m$ restricted to $X$ which is trivialized over $\partial X$ (since it is trivialized over $kD^n$). Similarly $\widetilde{\KO}(\overline{S^m})$ is generated by the restriction of $\dd_{m}$ mod $k$ for $m\equiv 0\bmod 4$.
\end{remark}

We define the \textbf{topological $\ZZ_k$-index} of $\overline{M}$ to be
 \[
 \tind_{\ZZ_k}(\overline{M}):=f_!(1)\in\ZZ_k.
 \]
 If $\xi$ is a complex vector bundle over $\overline{M}$, then we define the \textbf{topological $\ZZ_k$-index of $\overline{M}$ twisted by $\xi$} to be
 \[
 \tind_{\ZZ_k}(\overline{M},\xi):=f_!(\xi)\in \ZZ_k.
 \]
 As usual, the topological $\ZZ_k$-index is independent of choice of embedding due to the multiplicative property of the weak $\KU$-Thom class.

 Now the mod $k$ index theorem of Freed and Melrose \cite{modkindex} applied to the Dirac operator on the fundamental complex $^h$spinor bundle implies the analytic and topological $\ZZ_k$-indices of $\overline{M}$ agree:
 \begin{proposition}\label{zkindextheorem}
 	$\aind_{\ZZ_k}(\overline{M})= \tind_{\ZZ_k}(\overline{M})$.
 \end{proposition}
 \begin{proof}
 	The proof is the same as that in the spin$^c$ case presented in \cite[Corollary 5.4]{modkindex}. It suffices to replace the spin$^c$ $\KU$-Thom class by the spin$^h$ weak $\KU$-Thom class and replace the Dirac operator accordingly. A crucial computation is made therein using Atiyah-Patodi-Singer index theorem, in our case this is done by \Cref{APS}.
 \end{proof}
 
 The same of course holds in the twisted case: $\aind_{\ZZ_k}(\overline{M},\xi)= \tind_{\ZZ_k}(\overline{M},\xi)$ for complex vector bundle $\xi$ over $\overline{M}$. 
 
 Now since we are concerned with real vector bundles, we are more interested in the quantity
 \[
 \widetilde{\tind}_{\ZZ_k}(\overline{M},\gamma):=\widetilde{f}_!(\gamma)
 \]
 where $\gamma$ is a real vector bundle over $\overline{M}$. We wish to find an analytic formula for $\widetilde{\tind}_{\ZZ_k}(\overline{M},\gamma)$. In view of \Cref{zksphere}, if $\dim \overline{M}\equiv 4\bmod 8$, then
 \[
 \widetilde{\tind}_{\ZZ_k}(\overline{M},\gamma)=\tind_{\ZZ_k}(\overline{M},\gamma_\CC)
 \]
 where $\gamma_\CC$ is the complexification of $\gamma$. Then \Cref{zkindextheorem} gives the desired analytic formula. On the other hand if $\dim \overline{M}\equiv 0\bmod 8$, then $$\tind_{\ZZ_k}(\overline{M},\gamma_\CC)=2\cdot\widetilde{\tind}_{\ZZ_k}(\overline{M},\gamma).$$
Since $2$ is not necessarily invertible in $\ZZ_k$, $\widetilde{\tind}_{\ZZ_k}(\overline{M},\gamma)$ cannot be directly deduced from $\tind_{\ZZ_k}(\overline{M},\gamma_\CC)$. However recall (from \Cref{extradivisibility}) that if $\overline{M}\equiv 0\bmod 8$, then the analytic index (twisted by a real bundle) before modulo $k$ is divisible by $2$, so it is natural to expect
\[
\widetilde{\tind}_{\ZZ_k}(\overline{M},\gamma)=\frac{1}{2}\int_M \hat{\mathfrak{A}}^h(M)\ph(\gamma)-\frac{1}{2}k\cdot\eta^h(\beta M,\gamma_\CC) \pmod k.
\]
This is indeed true and can be proved using the method of \cite[Theorem 3.2]{mod2refinement} in which a similar formula for $\ZZ_k$-spin manifold of dimension $8k+4$ is obtained. The crucial ingredient of that proof is the quaternionic structure on the fundamental $\ZZ_2$-graded complex Clifford module in dimension $8k+4$. So that proof, adapted in our case using the quaternionic structure on the fundamental $\ZZ_2$-graded complex module over $\Cl_{8k}^h$, yields the above formula.

To summarize, we have proved:
\begin{theorem}\label{zkindex}
	Let $\overline{M}$ be a $\ZZ_k$-spin$^h$ manifold of dimension $n\equiv 0\bmod 4$ and let $\gamma$ be real vector bundle over $M$. Then after equipping appropriate geometric data on $M$ and $\gamma$ as before, we have
	\[
	\widetilde{\tind}_{\ZZ_k}(\overline{M},\gamma)=\frac{1}{\epsilon}\left(\int_M\hat{\mathfrak{A}}^h(M)\ph(\gamma)-k\cdot\eta^h(\beta M,\gamma_\CC)\right)\pmod k
	\]
	where $\epsilon=1$ if $n\equiv 4\bmod 8$ and $\epsilon=2$ if $n\equiv 0\bmod 8$.
\end{theorem}

\section{Characteristic variety theorems}\label{sec6}
In this chapter, we prove our main theorem. We start by showing spin$^h$ manifolds provide enough cycles for symplectic K-theory. Then we define invariants of real vector bundles over spin$^h$ cycles and spin$^h$ $\ZZ_k$-cycles using indices of Dirac operators. These invariants descend to periods of real vector bundles over symplectic K-homology, and therefore by \Cref{thm2} determine real vector bundles up to stable equivalence.

Henceforth, we will use the notation $\Omega_*^G(-)$ to denote the bordism theory of $G$-manifolds. For instance, $\Omega_*^{fr}(-)$, $\Omega_*^{\spin}(-)$ and $\Omega_*^{\spin^h}(-)$ stand for the framed, spin and spin$^h$ bordism theory respectively.

\subsection{Spin$^h$ orientation of KSp} Using the universal weak-KO-Thom class of the universal bundle $\mathsf{E}_{8k+4}$ over $\BSpin^h(8k+4)$ we can construct a spectrum map from the Thom spectrum of spin$^h$ cobordism to the  $\Omega$-spectrum of symplectic K-theory. This is a consequence of the following commutative diagram
	\begin{center}
		\begin{tikzcd}
		S^8\wedge \MSpin^h(8k+4)\arrow[r]\arrow[d,"id\wedge\dd^h(\mathsf{E}_{8k+4})"]& \MSpin^h(8k+12)\arrow[d,"\dd^h(\mathsf{E}_{8k+12})"]\\
		S^8\wedge (\BO\times\ZZ) \arrow[r,"Bott"]& \BO	\times\ZZ
		\end{tikzcd}
	\end{center}
	where the top map is induced by the bundle $\RR^8\oplus\mathsf{E}_{8k+4}$ and the bottom map is the Bott periodicity map. The commutativity follows from the multiplicative property of the weak-KO-Thom class (\Cref{mult}) and that $\dd_8\in\widetilde{\KO}(S^8)=\KO^{-8}(\pt)$ is exactly the Bott generator. Thus we obtain a spectrum map $\hat{\mathcal{A}}^h:\underline{\MSpin}^h\to \underline{\KSp}$. The multiplicative property of the weak-KO-Thom class further implies $\hat{\mathcal{A}}^h$ is a module map over the ring spectra homomorphism $\hat{\mathcal{A}}: \underline{\MSpin}\to \underline{\KO}$ where $\hat{\mathcal{A}}$ is the Atiyah-Bott-Shapiro spin-orientation of $\KO$, defined using the KO-Thom class for spin vector bundles. Indeed, the following diagram commutes:
\begin{center}
		\begin{tikzcd}
		\MSpin(8l)\wedge \MSpin^h(8k+4)\arrow[r]\arrow[d,"\dd\wedge\dd^h"]& \MSpin^h(8l+8k+4)\arrow[d,"\dd^h"]\\
		(\BO\times\ZZ)\wedge (\BO\times\ZZ) \arrow[r,"\boxtimes"]& (\BO\times\ZZ)
		\end{tikzcd}
\end{center}
where the top map is induced by the Whitney sum of the universal bundles.

We can think of $\hat{\mathcal{A}^h}$ as a sort of spin$^h$ orientation for $\KSp$ in view of the following theorem.

\begin{theorem}\label{CF-KSp}
	Let $X$ be a finite CW-complex, and $\Lambda$ an abelian group. Then the natural transformation
	\[
	\hat{\mathcal{A}}^h:\Omega_*^{\spin^h}(X;\Lambda)\to \KSp_*(X;\Lambda),
	\]
	upon a base change with respect to $\hat{\mathcal{A}}:\Omega_*^{\spin}(\pt)\to \KO_*(\pt)$, induces a surjection
	\[
	\Omega_*^{\spin^h}(X;\Lambda) \otimes_{\Omega_*^{\spin}(\pt)} \KO_*(\pt)\to \KSp_*(X;\Lambda).
	\]
	with a canonical splitting that is natural in $X$ and functorial in $\Lambda$. 
\end{theorem}

\begin{proof}
Consider the following commutative diagram:
\[
\begin{tikzcd}
\Omega_*^{\spin}(X;\Lambda)\ar[r,"\hat{\mathcal{A}}"]\ar[d,"\times\HH\PP^1_+"]& \KO_*(X;\Lambda)\ar[d,"\times \hat{\mathcal{A}}^h(\HH\PP^1_+)"]\\
\Omega_*^{\spin^h}(X;\Lambda)\ar[r,"\hat{\mathcal{A}}^h"]& \KSp_*(X;\Lambda)
\end{tikzcd}
\]
where vertical maps are of degree $+4$. After base change, we get
\[
\begin{tikzcd}
\Omega_*^{\spin}(X;\Lambda)\otimes_{\Omega_*^{\spin}(\pt)}\KO_*(\pt)\ar[r]\ar[d,"\times\HH\PP^1_+"]& \KO_*(X;\Lambda)\ar[d,"\times \hat{\mathcal{A}}^h(\HH\PP^1_+)"]\\
\Omega_*^{\spin^h}(X;\Lambda)\otimes_{\Omega_*^{\spin}(\pt)}\KO_*(\pt)\ar[r]& \KSp_*(X;\Lambda)
\end{tikzcd}
\]
The right vertical map is an isomorphism since $\hat{\mathcal{A}}^h(\HH\PP^1_+)$ generates $\KSp_4(\pt)$ (\Cref{HP1}). The top horizontal map is an isomorphism thanks to Hopkins and Hovey \cite{HH92}. Therefore the bottom horizontal map is surjective with a splitting given by composing the inverses of the two isomorphisms followed by $\times \HH\PP^1$. It is clear this splitting is natural in $X$ and functorial in $\Lambda$.
\end{proof}

\begin{remark}
	The map in this theorem is not an isomorphism. For instance take $X=\pt$ and $\Lambda=\ZZ[1/2]$, the left hand side is $\KO_*(\pt)\otimes_\ZZ H_*(\HH\PP^\infty;\ZZ[1/2])$ which is strictly bigger than $\KSp_*(\pt;\ZZ[1/2])$.
\end{remark}

%\begin{definition}
%	Let $X$ be a finite CW-complex and $\Lambda$ an abelian group. Define the \textbf{Conner-Floyd reduction} of $\Omega_*^{\spin^h}(X;\Lambda)$ to be the $\KO_*(\pt)$-module
%	\[
%	\overline{\Omega}_*^h(X;\Lambda):=\Omega_*^{\spin^h}(X;\Lambda) \otimes_{\Omega_*^{\spin}(\pt)} \KO_*(\pt).
%	\]
%\end{definition}

\subsection{Invariants of real vector bundles}
\subsubsection{Integer and parity invariants}
Let $f:M\to X$ be a spin$^h$ cycle in $X$ and let $E$ be a real vector bundle over $X$, we define a pairing
\begin{equation}\label{Z-pairing}
	\langle M\xrightarrow{f} X|X\gets E\rangle:=\tind(M,f^*E)\in \KSp_n(\pt).
\end{equation}
This pairing can be computed using \Cref{twistedcliffindex}. For simplicity of notation, we denote $\langle M\xrightarrow{f} X|X\leftarrow E\rangle$ by $\langle M| E\rangle$.
\begin{proposition}\label{Z-pairingprop}
	The pairing \labelcref{Z-pairing} has the following properties.
	\begin{enumerate}[label=(\roman*)]
		\item (biadditivity) For spin$^h$ cycles $(M,f)$ and $(M',f')$ we have
		\[
		\langle M\sqcup M'|E\rangle=\langle M|E\rangle+\langle M'|E\rangle;
		\]
		and for vector bundles $E,E'$ we have
		\[
		\langle M|E\oplus E' \rangle=\langle M|E\rangle+\langle M|E'\rangle.
		\]
		\item (naturality) Let $(M,f)$ be a spin$^h$ cycle in $X$ and $g: X\to Y$ a continuous map, $F$ a vector bundle over $Y$. Denote the spin$^h$ cycle $g\circ f: M\to Y$ by $g_*M$. Then
		\[
		\langle g_*M|F\rangle=\langle M|g^*F\rangle.
		\]
		\item (cobordism invariance) If $(M,f)$ is a spin$^h$ boundary, then $\langle M| E\rangle=0$.
		\item (slant product) From (i)(ii)(iv) the pairing \labelcref{Z-pairing} descends to a pairing $$\Omega_n^{\spin^h}(X)\otimes \KO(X)\xrightarrow{\langle-|-\rangle} \KSp_n(\pt).$$
		We have the following commutative diagram:
	\[
	\begin{tikzcd}[row sep=small, column sep=small]
		\Omega_n^{\spin^h}(X)\otimes\KO(X)\ar[dr,"{\langle -|-\rangle}"]\ar[dd,"\hat{\mathcal{A}}^h\otimes 1"' ]& \\
		& \KSp_n(\pt)\\
		\KSp_n(X)\otimes\KO(X)\ar[ur,"-\backslash-"] & 
	\end{tikzcd}
	\]
	where $-\backslash-$ is the slant product defined using that $\underline{\KSp}$ is a $\underline{\KO}$-module.
	\item (multiplicativity) For a spin manifold $N$ and a spin$^h$ cycle $(M,f)$ we have
		\[
		\langle N\times M|E\rangle=\hat{\mathcal{A}}(N)\cdot\langle M|E\rangle
		\]
		where $N\times M$ is the spin$^h$ cycle given by projection onto $M$ followed by $f$. Here $\cdot$ means the action of $\KO_*(\pt)$ on $\KSp_*(\pt)$.
	\end{enumerate}
\end{proposition}
\begin{proof}
	(i) and (ii) are straightforward. Now to prove (iii) and (iv), we choose an embedding $M\subset S^{n+8k+4}$ with $k$ large, then the Pontryagin-Thom construction yields a map $S^{n+8k+4}\to X\wedge\MSpin^h(8k+4)$. The universal weak KO-Thom class yields a map $\MSpin^h(8k+4)\to \BO$ and the bundle $E$ over $X$ yields a map $X\to \BO\times \ZZ$. Putting these maps together, we obtain
	\begin{equation}\label{slantprod1}
		S^{n+8k+4}\to X\wedge \MSpin^h(8k+4)\xrightarrow{E\wedge \dd^h}(\BO\times \ZZ)\wedge \BO\xrightarrow{\boxtimes}\BO.
	\end{equation}
	By definition, $\langle M|E\rangle$ is exactly the homotopy class of the composition of the above maps. Now if $f:M\to X$ is a spin$^h$ boundary, then by Pontryagin-Thom argument, the first map is null-homotopic, and therefore $\langle M|E\rangle=0$. Meanwhile unwrapping the definition of $\hat{\mathcal{A}^h}$ and slant product, (iv) follows immediately from \labelcref{slantprod1}.  Finally (v) follows from (iv) by that $\hat{\mathcal{A}}^h$ is a module map over $\hat{\mathcal{A}}$ and that the slant product is a map of $\KO_*(\pt)$-modules.
%	For (iii) choose embeddings $\iota_N:N\subset \RR^p$ of codimension $8k$ and $\iota_{M}:M\subset \RR^q$ of codimension $8l+4$, consequently we have an embedding $\iota_{N\times M}=\iota_N\times \iota_M$. Then we have direct image maps ${(\iota_N)}_!: \KO(N)\to \KO_{cpt}(\RR^p)$, ${(\iota_M)}_!: \KO(M)\to \KO_{cpt}(\RR^q)$ and ${(\iota_{N\times M})}_!: \KO(N\times M)\to \KO_{cpt}(\RR^{p+q})$. The multiplicative property of weak KO-Thom class (\Cref{mult}) implies $(\iota_{N\times M})_!=(\iota_{N})_!\boxtimes(\iota_{M})_!$, and therefore
%	\begin{align*}
%		\langle N\times M|E\rangle&=(\iota_{N\times M})_!(1\boxtimes f^*E)\\
%		&=(\iota_N)_!(1)\boxtimes (\iota_M)_!(f^*E)\\
%		&=\hat{\mathcal{A}}(N)\boxtimes \langle M|E\rangle.
%	\end{align*}
\end{proof}
%\begin{remark}\label{enhancedcobordinv}
%	Since this pairing takes values in $\ZZ$ and $\ZZ_2$, the cobordism invariance can be slightly enhanced using linearity. If $(M,f)$ is a spin$^h$ cycle of dimension $n\equiv 0\bmod 4$ (resp. $n\equiv 5,6\bmod 8$), and some integer (resp. odd integer) multiple of $(M,f)$ is a spin$^h$ boundary, then $\langle M| E\rangle=0$.
%\end{remark}

\subsubsection{Angle invariants}
Let $f:\overline{M}\to X$ be a spin$^h$ $\ZZ_k$-cycle in $X$ of dimension $n\equiv 0\pmod 4$ and let $E$ be a real vector bundle over $X$, we define a pairing
\begin{equation}\label{Q/Z-pairing}
	\langle\overline{M}\xrightarrow{f}X|X\leftarrow E\rangle :=\widetilde{\tind}_{\ZZ_k}(\overline{M},f^*E)\in \ZZ_k\subset \QQ/\ZZ
\end{equation}
This pairing can be computed using \Cref{zkindex}. For simplicity of notation, we denote $\langle \overline{M}\xrightarrow{f} X|X\leftarrow E\rangle$ by $\langle \overline{M}| E\rangle$. To proceed further, we need an alternative definition of $\widetilde{\tind}_{\ZZ_k}$ valued in $\KSp_n(\pt;\ZZ_k)$. Note for $n\equiv 0\bmod 4$, $\KSp_n(\pt;\Lambda)=\Lambda$ by universal coefficient theorem.

Choose an embedding $i: \overline{M}\hookrightarrow \overline{S^m}$ of $\ZZ_k$-manifolds so that $m-n\equiv 4\bmod 8$ (in particular $m\equiv 0\bmod 4$) and further choose an embedding $j: \overline{S^m}\hookrightarrow S^{m+r}$ for $r$ sufficiently large. Then $i$ yields a direct image map $\widetilde{i}_!:\KO(\overline{M})\to \widetilde{\KO}(\overline{S^m})$ using the induced spin$^h$ structure on the normal bundle of $\overline{M}$. And recall $\widetilde{\tind}_{\ZZ_k}(\overline{M}, -)=\widetilde{i}_!(-)$. Meanwhile, we have
\begin{lemma}\label{thomisoforzksphere}
	There is an induced direct image isomorphism
	\[
	j_!:\widetilde{\KO}(\overline{S^m})\xrightarrow{\cong} \widetilde{\KO}^r(S^{m+r};\ZZ_k).
	\]
\end{lemma}
\begin{proof}
Consider the map $\phi:[\overline{S^m},\BO]\to \Omega_m^{fr}(\BO;\ZZ_k)$ that takes $g: \overline{S^m}\to \BO$ to the framed $\ZZ_k$-bordism class of $g$. We claim $\phi$ is an isomorphism. Indeed, the generator of $\Omega_m^{fr}(\BO;\ZZ_k)=\widetilde{\KO}(S^m;\ZZ_k)=\ZZ_k$ is reduction mod $k$ of the Bott generator $S^m\xrightarrow{\dd_m} \BO$ of $\Omega_{m}^{fr}(\BO)=\widetilde{\KO}(S^m)=\ZZ$. We can deform $\dd_m$ so that it is trivial over $k$ disjoint closed disks, and consequently $\dd_m$ is cobordant to $\dd_m|_{X}$ as framed $\ZZ_k$-cycles (recall \Cref{zksphereexample}). On the other hand, from \Cref{bundlegeneratorofzksphere}, $\dd_m|_X$ exactly is the generator of $[\overline{S^m},\BO]=\widetilde{\KO}(\overline{S^m})=\ZZ_k$. This proves $\phi$ is surjective and therefore an isomorphism.

	Now we interpret $\phi$ as a direct image homomorphism $j_!$ as follows. The embedding $j$ yields a map $S^{n+r}\to \overline{S^m}\wedge (S^r\cup_k D^{r+1})=\overline{S^m}\wedge (\underline{S}_{\ZZ_k})_r$ by a Pontryagin-Thom type construction as in \Cref{zkexample}. Then for any given map $g: \overline{S^m}\to \BO$ we obtain the composition
	\[
	S^{m+r}\to \overline{S^m}\wedge (\underline{S}_{\ZZ_k})_r\xrightarrow{g\wedge 1} \BO\wedge (\underline{S}_{\ZZ_k})_r
	\]
	whose homotopy class defines an element $j_!(f)$ in $\widetilde{\KO}^r(S^{m+r};\ZZ_k)$. Clearly under the isomorphism $\widetilde{\KO}^r(S^{m+r};\ZZ_k)\cong\Omega_m^{fr}(\BO;\ZZ_k)$ by a transversality argument as in \Cref{zkexample}, we see $j_!$ coincides with $\phi$.
\end{proof}

We \textit{redefine} $\widetilde{\tind}_{\ZZ_k}(\overline{M},-)$ to be $q_!j_!\widetilde{i}_!(-)\in \KSp_n(\pt;\ZZ_k)$ where $q_!$ is the suspension isomorphism $\widetilde{\KO}^r(S^{m+r};\ZZ_k)\xrightarrow{\cong} \widetilde{\KO}(S^m;\ZZ_k)$ followed by the periodicity isomorphism $\widetilde{\KO}(S^m;\ZZ_k)\xrightarrow{\cong} \KSp_n(\pt;\ZZ_k)$.
\begin{proposition}\label{Q/Z-pairingprop}
	The pairing \labelcref{Q/Z-pairing} has the following properties.
	\begin{enumerate}[label=(\roman*)]
		\item (biadditivity) For spin$^h$ $\ZZ_k$-cycles $(\overline{M},f)$ and $(\overline{M}',f')$ we have
		\[
		\langle \overline{M}\sqcup \overline{M}'|E\rangle=\langle \overline{M}|E\rangle+\langle \overline{M}'|E\rangle;
		\]
		and for vector bundles $E,E'$ we have
		\[
		\langle \overline{M}|E\oplus E' \rangle=\langle \overline{M}|E\rangle+\langle \overline{M}|E'\rangle.
		\]
		\item (naturality) Let $(\overline{M},f)$ be a spin$^h$ $\ZZ_k$-cycle in $X$, $g: X\to Y$ a continuous map, and $F$ a vector bundle over $Y$. Denote the $\ZZ_k$-spin$^h$ cycle $g\circ f: \overline{M}\to Y$ by $g_*\overline{M}$. Then
		\[
		\langle g_*\overline{M}|F\rangle=\langle \overline{M}|g^*F\rangle.
		\]
%		\item (compatibility) The pairing is unchanged when a $\ZZ_k$-cycle $(\overline{M},f)$ is viewed as a $\ZZ_{kl}$-cycle by multiplication $(M,\beta M)\mapsto (l M, \beta M)$.
		\item (cobordism invariance) If $(\overline{M},f)$ is a spin$^h$ $\ZZ_k$-boundary, then $\langle \overline{M}| E\rangle=0$.
		\item (slant product) From (i)(ii)(iii) the pairing \labelcref{Q/Z-pairing} descends to a pairing
		\[
		\Omega^{\spin^h}_n(X;\ZZ_k)\otimes \KO(X)\xrightarrow{\langle-|-\rangle}\KSp_n(\pt;\ZZ_k)
		\]
		We have the following commutative diagram:
	\[
	\begin{tikzcd}[row sep=small, column sep=small]
		\Omega_n^{\spin^h}(X;\ZZ_k)\otimes\KO(X)\ar[dr,"{\langle -|-\rangle}"]\ar[dd,"\hat{\mathcal{A}}^h\otimes 1"' ]& \\
		& \KSp_n(\pt;\ZZ_k)\\
		\KSp_n(X;\ZZ_k)\otimes\KO(X)\ar[ur,"-\backslash-"] & 
	\end{tikzcd}
	\]
	\item (multiplicativity) For a spin manifold $N$ and a spin$^h$ $\ZZ_k$-cycle $(\overline{M},f)$ we have
		\[
		\langle N\times \overline{M}|E\rangle=\hat{\mathcal{A}}(N)\cdot\langle \overline{M}|E\rangle.
		\]
		where $N\times \overline{M}$ is the spin$^h$ $\ZZ_k$-cycle given by projection onto $M$ followed by $f$, and $\cdot$ means the action of $\KO_*(\pt)$ on $\KSp_{*}(\pt;\ZZ_k)$.
	\end{enumerate}
\end{proposition}
\begin{proof}
The proof is similar to \Cref{Z-pairingprop}. (i)(ii)(iii) are straightforward. To prove (iv), choose an embedding $\overline{M}\subset \overline{S^m}$ so that $m-n\equiv 4\bmod 8$ and further choose an embedding $\overline{S^m}\subset S^{m+r}$ for $r$ sufficiently large. Then by a Pontryagin-Thom type argument, we get a map
	\[
	S^{m+r}\to X\wedge \MSpin^h(m-n)\wedge (\underline{S}_{\ZZ_k})_r
	\]
	using the normal spin$^h$ bundle of $\overline{M}$ in $\overline{S^m}$ and the "normal bundle" of $\overline{S^m}$ in $S^{m+r}$. Then using the universal weak KO-Thom class and the bundle $E$, we obtain the composition
	\begin{equation}\label{slantprod2}
		S^{m+r}\to X\wedge \MSpin^h(m-n)\wedge (\underline{S}_{\ZZ_k})_r\xrightarrow{(E\boxtimes \dd^h)\wedge 1}\BO\wedge (\underline{S}_{\ZZ_k})_r.
	\end{equation}
	From the proof of \Cref{thomisoforzksphere}, the homotopy class of this composition is exactly $\langle \overline{M}|E\rangle$. So if $(\overline{M},f)$ is a $\ZZ_k$-boundary then the first map is null-homotopic and therefore $\langle \overline{M}|E\rangle=0$. Similar to \Cref{Z-pairingprop}, (v) follows from \labelcref{slantprod2} and (vi) follows from (v).
\end{proof}

\subsection{Proof of \Cref{thm1}}
When defining the pairing \labelcref{Q/Z-pairing} we assumed $\overline{M}$ has dimension $n\equiv 0\bmod 4$. However, part (iv) of \Cref{Q/Z-pairingprop} allows us to extend the pairing \labelcref{Q/Z-pairing} to all dimensions by $\hat{\mathcal{A}}^h$ and slant product.
\begin{definition}\label{Lambda-invariant}
	Let $E$ be a real vector bundle over $X$ and $\Lambda$ an abelian group, we define $\inv^h_\Lambda(E)$ to be the composition
	\[
	\inv^h_\Lambda(E): \Omega_*^{\spin^h}(X;\Lambda)\xrightarrow{\hat{\mathcal{A}}^h}\KSp_*(X;\Lambda)\xrightarrow{-\backslash E}\KSp_*(\pt;\Lambda).
	\] 
\end{definition}

Under this definition, the pairing \labelcref{Z-pairing} is $\inv^h_\ZZ(E)$ evaluated on a spin$^h$ cycle and \labelcref{Q/Z-pairing} is $\inv^h_{\ZZ_k}(E)$ evaluated on a spin$^h$ $\ZZ_k$-cycle. 

To prove \Cref{thm1}, we are concerned with $\inv^h_\QQ$ and $\inv^h_{\QQ/\ZZ}$ (recall \Cref{periods}). They can be computed using $\inv^h_\ZZ$ and $\inv^h_{\ZZ_k}$ since $\inv^h_\QQ=\inv^h_\ZZ\otimes \QQ$ and $\inv^h_{\QQ/\ZZ}=\varinjlim_k \inv^h_{\ZZ_k}$. Note that $\KSp_*(\pt;\QQ)=\KSp_*(\pt)\otimes\QQ$ is concentrated in degrees multiples of four, so $\inv_\QQ^h$ involves only the integer invariants. On the other hand, an easy computation using the coefficient long exact sequence associated to $0\to \ZZ\to \QQ\to \QQ/\ZZ\to 0$ shows
\[
\KSp_n(\pt;\QQ/\ZZ)=
\begin{cases}
	\QQ/\ZZ & n\equiv 0\bmod 4\\
	\ZZ_2 & n\equiv 6,7\bmod 8\\
	0 &\mbox{otherwise}
\end{cases}
\]
Moreover, $\KSp_n(\pt;\QQ/\ZZ)$ is isomorphic to $\KSp_{n-1}(\pt)$ for $n\equiv 6,7\bmod 8$ via Bockstein. Therefore $\inv_{\QQ/\ZZ}^h$ involves
\begin{enumerate}[label=(\roman*)]
	\item the angle invariants for $4m$-dimensional spin$^h$ $\ZZ_k$-cycles; and
	\item the parity invariants for spin$^h$-cycles of dimension $5,6\bmod 8$: if $(\overline{M},f)$ is a spin$^h$ $\ZZ_k$-cycle in $X$ of dimension $n\equiv 6,7\bmod 8$, then the evaluation of $\inv^h_{\QQ/\ZZ}(E)$ on $(\overline{M},f)$, via Bockstein isomorphism $\KSp_n(\pt;\QQ/\ZZ)\cong\KSp_{n-1}(\pt)$, is $\langle \beta M|E\rangle$ since both $\hat{\mathcal{A}}^h$ and slant product commute with Bockstein.
\end{enumerate}
From here we see that both $\inv^h_\QQ$ and $\inv^h_{\QQ/\ZZ}$ can be computed analytically. It is worth pointing out that $\inv_{\QQ}^h$ is \textbf{local} in the sense that the evaluation of $\inv_\QQ^h(E)$ on a spin$^h$-cycle can be expressed as an integral of a locally defined differential form. In contrast $\inv_{\QQ/\ZZ}^h$ is \textbf{global} since neither the eta invariant nor the parity invariant can be locally expressed.

The maps $\inv_{\QQ}^h$ and $\inv_{\QQ/\ZZ}^h$ are compatible in the sense that for any real vector $E$ over $X$, the diagram
\begin{equation}\label{invariantdiagram}
	\begin{tikzcd}
	\Omega_*^{\spin^h}(X;\QQ)\ar[r,"\inv_{\QQ}^h(E)"]\ar[d,"\bmod \ZZ"] & \KSp_*(\pt;\QQ)\ar[d,"\bmod \ZZ"]\\
	\Omega_*^{\spin^h}(X;\QQ/\ZZ)\ar[r,"\inv_{\QQ/\ZZ}^h(E)"] & \KSp_*(\pt;\QQ/\ZZ)
\end{tikzcd}
\end{equation}
commutes because the slant products are compatible with change of coefficients. Consider the map
\[
\inv^h: \KO(X)\to
\Biggl\{
\begin{tikzcd}[row sep=scriptsize, column sep=scriptsize]
	\Omega_*^{\spin^h}(X;\QQ)\ar[r]\ar[d,"\bmod \ZZ"] & \KSp_*(\pt;\QQ)\ar[d,"\bmod \ZZ"]\\
	\Omega_*^{\spin^h}(X;\QQ/\ZZ)\ar[r] & \KSp_*(\pt;\QQ/\ZZ)
\end{tikzcd}
\Biggl\}
\]
that sends $E$ to the diagram \labelcref{invariantdiagram}, where the target is the group of commutative diagrams of such form. We have
\begin{proposition}\label{injectivity!}
	$\inv^h$ is injective.
\end{proposition}
\begin{proof}
	 Observe that both $\inv_{\QQ}^h(E)$ and $\inv_{\QQ/\ZZ}^h(E)$ are equivariant with respect to $\hat{\mathcal{A}}:\Omega_*^{\spin}(\pt)\to \KO_*(\pt)$. Therefore $\inv^h$ maps into the subgroup of diagrams that are equivariant with respect to $\hat{\mathcal{A}}$. This subgroup, by tensor-hom adjunction, is isomorphic to the group
\begin{equation}\label{group-1}
	\Biggl\{
\begin{tikzcd}[row sep=scriptsize, column sep=scriptsize]
	\overline{\Omega}_*^{\spin^h}(X;\QQ)\ar[r]\ar[d,"\bmod \ZZ"] & \KSp_*(\pt;\QQ)\ar[d,"\bmod \ZZ"]\\
	\overline{\Omega}_*^{\spin^h}(X;\QQ/\ZZ)\ar[r] & \KSp_*(\pt;\QQ/\ZZ)
\end{tikzcd}
\Biggl\}_{\KO_*(\pt)}
\end{equation}
of commutative diagrams of $\KO_*(\pt)$-modules of the above form, where $\overline{\Omega}_*^{\spin^h}(X;\Lambda)$ denotes $\Omega_*^{\spin^h}(X;\Lambda)\otimes_{\Omega_*^{\spin}(\pt)}\KO_*(\pt)$ for $\Lambda=\QQ,\QQ/\ZZ$. From \Cref{CF-KSp}, the natural transformation $\hat{\mathcal{A}^h}$ induces split surjections $\overline{\Omega}_*^{\spin^h}(X;\Lambda)\to \KSp_*(X;\Lambda)$. Therefore by pulling back along these split surjections, we obtain an embedding $i$ of the group \labelcref{group-1}, as a direct summand, into the group
\begin{equation}\label{group-2}
	\Biggl\{
\begin{tikzcd}[row sep=scriptsize, column sep=scriptsize]
	\KSp_*(X;\QQ)\ar[r]\ar[d,"\bmod \ZZ"] & \KSp_*(\pt;\QQ)\ar[d,"\bmod \ZZ"]\\
	\KSp_*(X;\QQ/\ZZ)\ar[r] & \KSp_*(\pt;\QQ/\ZZ)
\end{tikzcd}
\Biggl\}_{\KO_*(\pt)}
\end{equation}
of commutative diagrams of $\KO_*(\pt)$-modules of the above form. Taking the degree zero components yields a map $\pi$ from the group \labelcref{group-2} to the group
\begin{equation}\label{group-3}
	\Biggl\{
\begin{tikzcd}[row sep=scriptsize, column sep=scriptsize]
	\KSp_0(X;\QQ)\ar[r]\ar[d,"\bmod \ZZ"] & \QQ\ar[d,"\bmod \ZZ"]\\
	\KSp_0(X;\QQ/\ZZ)\ar[r] & \QQ/\ZZ
\end{tikzcd}
\Biggl\}_{\ZZ}
\end{equation}
of commutative diagrams of abelian groups of the above form. Note by \Cref{periodKO}, the group \labelcref{group-3} is isomorphic to $\KO(X)$.

We claim the composition of $\inv^h$ and the maps $i,\pi$, namely
\[
\KO(X)\xrightarrow{\inv^h}\labelcref{group-1}\xrightarrow{i}\labelcref{group-2}\xrightarrow{\pi} \labelcref{group-3},
\]
is an isomorphism. Indeed this composition $\pi\circ i\circ\inv^h$ by construction takes $x\in\KO(X)$ to slant products with $x$, which by the proof of \cite{Anderson} is an isomorphism. This in particular implies $\inv^h$ is injective.
\end{proof}

The above proposition means that two real vector bundles are stably equivalent if and only if they have the same integer, parity and angle invariants over \textit{all} spin$^h$-cycles and torsion spin$^h$-cycles. We can improve this to

\begin{theorem}[Characteristic variety theorem, spin$^h$ version]\label{thm3} For each finite CW-complex $X$, there exists a finite set of spin$^h$ cycles and torsion spin$^h$ cycles\footnote{a torsion cycle is a $\ZZ_k$-cycle for some $k$.} in $X$ such that every real vector bundle on $X$ can be determined, up to stable equivalence, by the corresponding integer, parity and angle invariants.
\end{theorem}
\begin{proof}
We can write
	 \begin{align*}
	 	\Omega_n^{\spin^h}(X;\QQ)&=(\Omega_n^{\spin^h}(X)/{\text{torsion}})\otimes\QQ\\
	 	\Omega_n^{\spin^h}(X;\QQ/\ZZ)&=(\Omega_n^{\spin^h}(X)/{\text{torsion}})\otimes{\QQ/\ZZ}\oplus \text{torsion}(\Omega_{n-1}^{\spin^h}(X))
	 \end{align*}
	 where the latter is be obtained from the coefficient long exact sequence associated to $0\to\ZZ\to \QQ\to \QQ/\ZZ\to 0$ by noting $\QQ/\ZZ$ is injective. Under this notation we can write $\inv_{\QQ}^h=\inv_\ZZ^h\otimes\QQ$ and $\inv^h_{\QQ/\ZZ}=\inv^h_{\ZZ}\otimes\QQ/\ZZ\oplus \phi$ where $\phi$ is the restriction of $\inv_{\QQ/\ZZ}^h$ onto $\text{torsion}(\Omega_{*-1}^{\spin^h}(X))$. Then it is clear that $\inv^h$ is determined by $\inv_\ZZ^h$ restricted to $\Omega_{4*}^{\spin^h}(X)/\text{torsion}$ and $\phi$. For $n\equiv 0\bmod 4$, $\phi$ encodes the angle invariants and for $n\equiv 6,7\bmod 8$, $\phi$ encodes the parity invariants.
	 
So \Cref{injectivity!} can be restated as that the map
	 \begin{align*}
	 	\KO(X)\to & \Hom(\Omega_{4*}^{\spin^h}(X)/\text{torsion},\ZZ)\\
	 	 &\oplus \Hom(\text{torsion}(\Omega_{4*-1}^{\spin^h}(X)),\QQ/\ZZ)\\
	 	 &\oplus \Hom(\text{torsion}(\Omega_{8*+5,8*+6}^{\spin^h}(X)),\ZZ_2)
	 \end{align*}
	 given by $(\inv_\ZZ^h, \phi,\phi)$ is injective. Note the target group carries an increasing filtration by $*\le N$. Since $\KO(X)$ is Noetherian, the above map embeds $\KO(X)$ into some finite filtration level. Now choose a finite set of additive generators for $\Omega_{4*}^{\spin^h}(X)/\text{torsion}$, $\text{torsion}(\Omega_{4*-1}^{\spin^h}(X))$ and $\text{torsion}(\Omega_{8*+5,8*+6}^{\spin^h}(X))$ up to that filtration level, and then represent those generators by $4*$-dimensional spin$^h$-cycles, $4*$-dimensional torsion spin$^h$-cycles and $8*+5, 8*+6$-dimensional spin$^h$-cycles respectively. This is the finite set of spin$^h$ cycles and torsion spin$^h$ cycles as desired.
\end{proof}

The name, characteristic variety theorem, is borrowed from \cite{DS71}. We call the \textit{union} of the cycles and torsion cycles appearing in \Cref{thm3} a \textbf{characteristic variety} (for real vector bundles). Thus two real vector bundle are stably equivalent if and only if they have the same invariant (now valued in a direct sum of copies of $\ZZ$, $\ZZ_2$ and $\QQ/\ZZ$) over a characteristic variety.

\Cref{thm1} is the spin version of the above characteristic variety theorem, whose proof is similar. We sketch the proof of \Cref{thm1} below with an emphasis on its relation with \Cref{thm3}.

\begin{theorem}[\Cref{thm1}]\label{thm1'}
	For each compact manifold with corners (or finite CW-complex) $X$, there exists a \textit{finite} set of spin cycles and torsion spin cycles in $X$ such that every real vector bundle on $X$ can be determined, up to stable equivalence, by the corresponding integer, parity and angle invariants.
\end{theorem}

\begin{proof}
	Using the natural transformation $\hat{\mathcal{A}}$ and slant product in real K-theory, we can define $\inv_\Lambda(E):\Omega_*^{\spin}(X;\Lambda)\to \KO_*(\pt;\Lambda)$ similar to \Cref{Lambda-invariant}. Note $\inv_{\QQ}$ and $\inv_{\QQ/\ZZ}$ can be computed using the integer, parity and angle invariants defined in the introduction. Now $\inv_{\QQ}$ and $\inv_{\QQ/\ZZ}$ can be assembled into a map
	\[
	\inv: \KO(X)\to \Biggl\{
\begin{tikzcd}[row sep=scriptsize, column sep=scriptsize]
	\Omega_*^{\spin}(X;\QQ)\ar[r,"\inv_\QQ"]\ar[d,"\bmod \ZZ"] & \KO_*(\pt;\QQ)\ar[d,"\bmod \ZZ"]\\
	\Omega_*^{\spin}(X;\QQ/\ZZ)\ar[r,"\inv_{\QQ/\ZZ}"] & \KO_*(\pt;\QQ/\ZZ)
\end{tikzcd}
\Biggl\}_{\hat{\mathcal{A}}}
	\]	
which is related to our previous map $\inv^h$ by multiplication by $\HH\PP^1_+$. More precisely for any map $\varphi^h:\Omega_*^{\spin^h}(X;\Lambda)\to \KSp_*(\pt;\Lambda)$, we can form a map $\varphi:\Omega_*^{\spin}(X;\Lambda)\to \KO_*(\pt;\Lambda)$ by requiring the following diagram to commute:
\[
\begin{tikzcd}
	\Omega_*^{\spin}(X;\Lambda)\ar[r,"\varphi"]\ar[d,"\times \HH\PP^1_+"]& \KO_*(\pt;\Lambda)\ar[d,"\times\hat{\mathcal{A}}^h(\HH\PP^1_+)","\cong"']\\
	\Omega_*^{\spin^h}(X;\Lambda)\ar[r,"\varphi^h"]& \KSp_*(\pt;\Lambda)
\end{tikzcd}
\]
Denote the mapping $\varphi^h\mapsto\varphi$ by $(\HH\PP^1_+)^*$ and we will use the same notation for all consequent maps induced by multiplication by $\HH\PP^1_+$. It follows from the proof of \Cref{CF-KSp} that $\inv=(\HH\PP^1_+)^*\inv^h$.

Now a similar construction applied to $\inv$ as in the proof of \Cref{injectivity!} (with $\hat{\mathcal{A}}^h$ replaced by $\hat{\mathcal{A}}$, \Cref{CF-KSp} replaced by \cite{HH92}, and taking degree $0$ components replaced by taking degree $-4$ components) yields a map
\[
\KO(X)\to
\Biggl\{
\begin{tikzcd}[row sep=scriptsize, column sep=scriptsize]
	\KO_{-4}(X;\QQ)\ar[r]\ar[d,"\bmod \ZZ"] & \QQ\ar[d,"\bmod \ZZ"]\\
	\KO_{-4}(X;\QQ/\ZZ)\ar[r] & \QQ/\ZZ
\end{tikzcd}
\Biggl\}_\ZZ
\]
which coincides with $(\HH\PP^1_+)^*\circ \pi\circ i\circ \inv^h$. Notice now both $(\HH\PP^1_+)^*$ and $\pi\circ i\circ \inv^h$ are isomorphisms, so $\inv$ is injective. This combined with that $\KO(X)$ is finitely generated completes the proof.
\end{proof}
\begin{remark}\label{transferofinvariants}
	The proof in fact shows that the spin$^h$ cycles and torsion spin$^h$ cycles in \Cref{thm3} can be chosen to be products of spin cycles and torsion spin cycles with the map $\HH\PP^1_+\to \pt$.
\end{remark}

%by $E\mapsto |E\rangle$ using the pairing \labelcref{Z-pairing} and \labelcref{Q/Z-pairing}. These maps are compatible with $\QQ\to \QQ/\ZZ$ in the sense that
%\begin{equation}
%	\KO(X)\xrightarrow{\lambda}
%\Biggl\{
%\begin{tikzcd}
%\Omega_*^{\spin^h}(X;\QQ) \ar[r]\ar[d]& \KSp_*(\pt;\QQ)\ar[d]\\
%\Omega_*^{\spin^h}(X;\QQ/\ZZ) \ar[r]&\KSp_*(\pt;\QQ/\ZZ)
%\end{tikzcd}
%\Biggl\}_{\hat{\mathcal{A}}}
%\end{equation}
%By a base change we get
%\begin{equation}
%\KO(X)\to
%\Biggl\{
%\begin{tikzcd}
%\Omega_*^{\spin^h}(X;\QQ)\otimes_{\Omega_*^{\spin}(\pt)}\KO_*(\pt) \ar[r]\ar[d]& \KSp_*(\pt;\QQ)\ar[d]\\
%\Omega_*^{\spin^h}(X;\QQ/\ZZ)\otimes_{\Omega_*^{\spin}(\pt)}\KO_*(\pt) \ar[r]&\KSp_*(\pt;\QQ/\ZZ)
%\end{tikzcd}
%\Biggl\}_{\KO_*(\pt)}
%\end{equation}
%Using the splittings introduced in \Cref{CF-KSp}, we get a map
%\begin{equation}
%\KO(X)\to
%\Biggl\{
%\begin{tikzcd}
%\KSp_*(X;\QQ) \ar[r]\ar[d]& \KSp_*(\pt;\QQ)\ar[d]\\
%\KSp_*(X,\QQ/\ZZ) \ar[r]&\KSp_*(\pt;\QQ/\ZZ)
%\end{tikzcd}
%\Biggl\}_{\KO_*(\pt)}
%\end{equation}
%Next taking the degree zero component on the right hand side, we obtain
%\begin{equation}
%\KO(X)\to
%\Biggl\{
%\begin{tikzcd}
%\KSp_0(X;\QQ) \ar[r]\ar[d]& \QQ\ar[d]\\
%\KSp_0(X,\QQ/\ZZ) \ar[r]&\QQ/\ZZ
%\end{tikzcd}
%\Biggl\}_{\ZZ}
%\end{equation}
%By \Cref{periodKO}, the right hand side is isomorphic to $\KO(X)$, so eventually we get a natural transformation
%\begin{equation}
%	\KO(X)\to \KO(X).
%\end{equation}
%We prove this map is an isomorphism by verifying it is an isomorphism for spheres.
%
%\subsubsection{Verification for spheres}

\subsection{Examples}
\begin{example}[Spheres]
\begin{enumerate}[label=(\roman*)]
    \item For $S^8$, one can take the union of $\pt\to S^8$ and the identity map of $S^8$ as its spin characteristic variety. Indeed, $\pt\to S^8$ detects the rank of the bundle and the Dirac index of $S^8$ twisted by a real bundle $E$, according to index theorem, is
    \[
    \int_{S^8}\hat{A}(S^8) \ph(E)=\int_{S^8}\ph(E).
    \]
    Bott showed this integral is always an integer and further this integer together with the rank classify vector bundles on $S^8$ up to stable equivalence. Similarly for $S^4$, one can take $\pt\cup S^4\to S^4$. In this case, the Dirac index of $S^4$ twisted by a bundle $E$ is $\frac{1}{2}\int_{S^4} \ph(E)$. Again thanks to Bott, this is always an integer and further (together with rank) classifies stable vector bundles.
    \item For $S^2=\CC\PP^1$, putting rank aside, there are only two classes of stable vector bundles, the trivial complex line bundle $\mathcal{O}$ and the tautological complex line bundle $\mathcal{O}(-1)$. The tangent bundle of $\CC\PP^1$ is $\mathcal{O}(2)$, whose "square root" is $\mathcal{O}(1)$. The Dirac operator in this case is $\overline{\partial}$ and the Dirac index is the dimension of holomorphic sections modulo $2$. In the untwisted case, the Dirac index of $\CC\PP^1$ is the (complex) dimension of holomorphic sections of $\mathcal{O}(1)$ modulo $2$, which is zero. Meanwhile the Dirac index of $\CC\PP^1$ twisted by $\mathcal{O}(-1)$ is the dimension of holomorphic functions on $\CC\PP^1$ modulo $2$, which is one. So the twisted Dirac indices indeed distinguish real vector bundles on $S^2$ up to stable equivalence. 
    \item For $S^1=\RR\PP^1$, we need to use the non-trivial spin structure on $S^1$ which corresponds to the double cover of $\RR\PP^1$ by $S^1$. The spinor bundle in this case is the tautological line bundle. The Dirac index of $\RR\PP^1$ without twisting is the (real) dimension of locally constant sections of the tautological line bundle modulo $2$. However since every section of the tautological line bundle has a zero, the locally constant ones must be identically zero. So the untwisted Dirac index is zero. On the other hand, the Dirac index of $\RR\PP^1$ twisted by the tautological bundle is the dimension of the constant functions on $\RR\PP^1$ modulo $2$, which is one. So the twisted Dirac indices indeed distinguish trivial bundle and tautological bundle on $S^1=\RR\PP^1$ up to stable equivalence. But, putting rank aside, these are the only two stable equivalence class of real vector bundles.
    \item In general $\pt\cup S^n\to S^n$ is a spin characteristic variety for $S^n$. Indeed, the "integration over the fiber" map
    \[
    \widetilde{\KO}(S^n)\xrightarrow{[S^n]\backslash-}\KO_n(\pt)
    \]
    is an isomorphism, where $[S^n]$ is the fundamental class of $S^n$ in real K-theory.
\end{enumerate}
\end{example}
\begin{example}[Quaternionic projective spaces]\label{HPn} $\pt\cup\HH\PP^1\cup\HH\PP^2\cup\cdots\cup\HH\PP^n\to \HH\PP^n$ by inclusions of a flag of quaternionic projective subspaces is a spin characteristic variety. To see this, let $V_k$ denote the weight $k$ irreducible complex representation of $\Sp(1)=\mathrm{SU}(2)$ so that $\dim_\CC V_k=k+1$. Let $\xi_k$ be the complex vector bundle over $\mathrm{BSp}(1)=\HH\PP^\infty$ corresponding to $V_k$. Note that $V_k$ is of real type if $k$ is even and of quaternionic type if $k$ is odd. Let $\gamma_{2k}$ be the real bundle over $\HH\PP^\infty$ so that $\gamma_{2k}\otimes\CC=\xi_k$ and let $\gamma_{2k-1}$ be the underlying real bundle of $\xi_{2k-1}$. Then $\ph(\gamma_{2k})=\ch(\xi_{2k})$ and $\ph(\gamma_{2k-1})=\ch(\xi_{2k-1})+\ch(\overline{\xi_{2k-1}})=2\ch(\xi_{2k-1})$ ($V_{2k-1}$ is of quaternionic type and thus isomorphic to its complex conjugate). We now view $\HH\PP^j$ as a subspace of $\HH\PP^\infty$ by the standard embedding, and claim that
\[
\int_{\HH\PP^j}\hat{A}(\HH\PP^j)\ch(\xi_i)=\binom{i+j+1}{i-j}.
\]
In particular the integral is $0$ if $i<j$ and $1$ if $i=j$. From here the matrix
\[
\langle \gamma_i|\HH\PP^j\rangle_{i,j\le n}
\]
of integer invariants is an upper triangular matrix with diagonals $=1$. Note that for $i=j=\text{odd}$, the integer invariant is $\frac{1}{2}\int_{\HH\PP^{i}}\hat{A}(\HH\PP^{i})\ph(\gamma_{i})=\frac{1}{2}\times 2=1$. Therefore the integer invariants over a flag of quaternionic projective subspaces yield a surjective map
\[
\KO(\HH\PP^n)\to \ZZ^{n+1}
\]
which must be an isomorphism since $\KO(\HH\PP^n)$ is free of rank $n+1$ (by e.g. Atiyah-Hirzebruch spectral sequence).

To prove the claim, consider the Hopf-type fiberation $\pi:\CC\PP^{2j+1}\to \HH\PP^j$ with fiber $\CC\PP^1$. Since $\pi^*$ is injective on cohomology, we may express $H^*(\CC\PP^{2j+1};\ZZ)=\ZZ[x]/(x^{2j+2})$ and identify $H^*(\HH\PP^j;\ZZ)$ through $\pi^*$ with the subring generated by $x^2$. The $\hat{A}$-class of $\HH\PP^j$ is $F(x)^{2j+2}/F(2x)$ where $F(x)=x/2\sinh(x/2)$, and the Chern character of $\xi_i$ is $\sinh((i+1)x)/\sinh(x)$. Therefore by residue theorem, we need to evaluate the integral
\[
\frac{1}{2\pi \mathbf{i}}\int \frac{\sinh((i+1)x)}{\sinh(x)}\frac{F(x)^{2j+2}}{F(2x)}\frac{1}{x^{2j+1}} dx
\]
We will suppress the factor $1/2\pi\mathbf{i}$ for simplicity. Multiplying the above by $t^j$ and then summing over $j\ge 0$, we get a generating function \[
G_i(t)=\int \frac{\sinh((i+1)x)}{\sinh(x)}\frac{F(x)^2}{F(2x)x}\frac{dx}{1-\frac{F(x)^2}{x^2}t}.
\]
After making the substitution $y=2\sinh(x/2)=x+o(x)$, we have
\[
G_i(t)=\int \Theta_i(y)\frac{1}{y}\frac{dy}{1-t/y^2}=\int\Theta_i(y)\sum \frac{t^j}{y^{2j+1}}
\]
where $\Theta_i(y)$ is $\sinh((i+1)x)/\sinh(x)$ viewed as a function in $y$. So the coefficient of $t^j$ in $G_i(t)$ is exactly the coefficient of $y^{2j}$ in $\Theta_i(y)$. Finally let $U_i$ be the Chebychev polynomial of the second kind which satisfies $U_i(\cosh x)=\sinh((i+1)x)/\sinh(x)$, and then $\Theta_i(y)=U_i(1+\frac{1}{2}y^2)$. We leave it to the reader to use the well-known inductive relation $U_{i+1}(z)=2z U_{i}(z)-U_{i-1}(z)$ to derive $\Theta_i(y)=\sum_j\binom{i+j+1}{i-j}y^{2j}$.
\end{example}

\begin{example}[Classifying spaces]Let $G$ be a compact connected, simply-connected Lie group of rank $r$. Even though $BG$ is not a finite CW-complex, \Cref{thm1} still holds by \Cref{inverselimit} and $\KO(BG)=\varprojlim\KO(\text{finite skeleton of }BG)$. From Atiyah-Segal completion theorem, $\KO(BG)$ is the completed real representation of $G$. Now choose a maximal torus for $G$ and correspondingly obtain $r$ embeddings of $\Sp(1)=\mathrm{SU}(2)$ into $G$. From representation theory, a real representation of $G$ is trivial if and only if its restriction onto those $\Sp(1)$ subgroups are trivial. This implies $\KO(BG)$ embeds into the direct sum of $r$ copies of $\KO(\HH\PP^\infty)$ by pulling back along the induced maps between classifying spaces of those embeddings of groups. Therefore from the previous example, a spin characteristic variety of $BG$ can be chosen to be a union of maps from $\HH\PP^{i_1}\times\HH\PP^{i_2}\times \dots\times \HH\PP^{i_r}$ for $i_1,\dots,i_r\ge 0$. A special case is $G=\Sp(1)$ and $\KO(\mathrm{BSp}(1))=\varprojlim_n\KO(\HH\PP^n)$ is a power series ring over $\ZZ$. A real vector bundle over $\mathrm{BSp}(1)=\HH\PP^\infty$ is stably trivial if and only if its integer invariants over a flag of (positive dimensional) quaternionic projective subspaces are zero. It seems hard to explicitly find characteristic varieties if $G$ is not simply-connected (for example $U(1)$).
\end{example}

%Our characteristic variety theorem fails for general infinite CW-complexes.
%\begin{example}
%	Let $\pi: E\to S^2$ be a rank $3$ oriented real vector bundle over $S^2$ with $w_2(E)\neq 0$. Let $\iota: S^2\to E$ be a smooth embedding so that $\pi\circ\iota$ has degree $3$. Such a bundle $E$ exists, for instance one can take $E$ to be the Whitney sum (over $\RR$) of the tautological complex line bundle over $\CC\PP^1=S^2$ and the trivial real line bundle. Moreover, we note that rank $3$ oriented real vector bundles over $S^2$ are classified by its second Stiefel-Whitney class. Such an embedding $\iota$ also exists, because we can first find a degree $3$ map from $S^2$ to itself and then deform it into a smooth embedding into $E$ ($\dim E=5>4=2\cdot\dim S^2$). Now it is easy to see the normal bundle to $\iota$ has $w_2\neq 0$ and therefore isomorphic to $E$. Hence a tubular neighborhood of $S^2$ in
%\end{example}

\section{Index theorem for quaternionic operators}\label{sec7}
In this chapter, we state and sketch the proof of an index theorem for families of quaternionic operators and apply it to prove that the analytic and topological indices of a closed spin$^h$ manifold coincide.
\subsection{Bigraded ABS isomorphisms}%This subsection can be replaced by claiming the main results hold.
In this subsection we study Clifford algebras associated to (non-degenerate) indefinite quadratic forms as well as their modules. Since over $\CC$ all non-degenerate quadratic forms are isomorphic, we deal only with $\RR$- and $\HH$-modules. Our goal is to establish a quaternionic version of ABS isomorphism in the indefinite setting. Given our understanding of the positive definite case, this will be easy once we have the appropriate language.
\subsubsection{Clifford algebras associated to indefinite forms}
Let $\Cl_{r,s}$ be the Clifford algebra on $\RR^{r+s}=\RR^r\times \RR^s$ with respect to the quadratic form $\|x\|^2-\|y\|^2$ of signature $(r,s)$ where $x\in\RR^r$ and $y\in \RR^s$. In particular $\Cl_{n,0}=\Cl_n$. These algebras are also $\ZZ_2$-graded (induced from the antipodal map on $\RR^{r+s}$) and there are $\ZZ_2$-graded isomorphisms (see e.g. \cite[Proposition 3.2]{LM89}):
\begin{equation}\label{indefz2gradediso}
	\Cl_{r+r',s+s'}\cong\Cl_{r,s}\hat{\otimes}_\RR\Cl_{r',s'}
\end{equation}
These isomorphism can be deduced from the following general fact. Recall the Clifford algebra $\Cl(V,q)$ associated to a vector space $V$ (over a field of characteristic not $2$) with a quadratic form $q$ is the quotient of the tensor algebra generated by $V$ by the relations $v\otimes v=-q(v)$ for all $v\in V$. The antipodal map $v\mapsto -v$ extends to an involution of $\Cl(V,q)$ yielding a $\ZZ_2$-grading on $\Cl(V,q)$.
\begin{lemma}\label{generalz2gradediso}
	Let $(V,q)$ and $(V',q')$ be finite dimensional vector spaces with quadratic forms over a field (of characteristic not 2). Then $$\Cl(V\oplus V, q\oplus q')\cong\Cl(V,q)\hat{\otimes}\Cl(V',q').$$
\end{lemma}
\begin{proof}
	The linear map $V\oplus V'\to \Cl(V,q)\hat{\otimes}\Cl(V',q')$, $(v,v')\mapsto v\hat{\otimes}1+1\hat{\otimes}v'$ extends to an algebra map $\Cl(V\oplus V, q\oplus q')\to\Cl(V,q)\hat{\otimes}\Cl(V',q')$ which is seen to be an isomorphism since it maps onto a set of generators and the two algebras in question have the same dimension ($2^{\dim V+\dim V'}$).
\end{proof}

For $\KK=\RR$ or $\HH$, let $\hat{\frM}_{r,s}(\KK)$ denote the Grothendieck group of finite dimensional $\ZZ_2$-graded $\KK$-modules over $\Cl_{r,s}$, and set $\hat{\frN}_{r,s}(\KK)=\hat{\frM}_{r,s}(\KK)/i^*\hat{\frM}_{r+1,s}(\KK)$ where $i^*$ is induced by the inclusion $\RR^r\times \RR^s\hookrightarrow \RR^{r+1}\times \RR^s$, $(x,y)\mapsto (x,0,y)$. Then naturally $\hat{\frM}_{\bullet,\bullet}(\RR)=\bigoplus_{r,s}\hat{\frM}_{r,s}(\RR)$ is a bigraded ring with respect to direct sum and $\ZZ_2$-graded tensor product, and
$\hat{\frM}_{\bullet,\bullet}(\HH)=\bigoplus_{r,s}\hat{\frM}_{r,s}(\HH)$
is a bigraded module over $\hat{\frM}_{\bullet,\bullet}(\RR)$. These structures descend to make $\hat{\frN}_{\bullet,\bullet}(\HH)=\bigoplus_{r,s}\hat{\frN}_{r,s}(\HH)$ a bigraded module over the bigraded ring $\hat{\frN}_{\bullet,\bullet}(\RR)=\bigoplus_{r,s}\hat{\frN}_{r,s}(\RR)$.

In \cite{KR}, Atiyah showed $\hat{\frN}_{\bullet,\bullet}(\RR)$ is isomorphic to \textbf{Real K-theory} of a point:
\[
\hat{\frN}_{\bullet,\bullet}(\RR)\cong \KR^{\bullet,\bullet}(\pt).
\]
We will prove analogously $\hat{\frN}_{\bullet,\bullet}(\HH)$ is isomorphic to \textbf{Quaternionic K-theory} of a point:
\[
\hat{\frN}_{\bullet,\bullet}(\HH)\cong \KQ^{\bullet,\bullet}(\pt).
\]

\subsubsection{Real and Quaternionic K-theories}
The Real K-theory $\KR$ is a variant of K-theory invented by Atiyah in \cite{KR} partially for the purpose of analyzing indices for families of real elliptic operators. The Quaternionic K-theory $\KQ$ was invented by Dupont in \cite{KQ} to extend the work of Atiyah to the quaternionic case. As we will see later, $\KQ$ is suitable for analyzing indices for families of quaternionic operators. Both theories are defined on the category of \textbf{real spaces}.

A real space is simply a space equipped with an involution. For example, the set of complex points of a real algebraic variety equipped with conjugation is a real space. Another important example is $\RR^{r,s}$ whose underlying space is $\RR^r\times \RR^s$ equipped with the involution $(x,y)\mapsto (x,-y)$ for $x\in\RR^s$, $y\in \RR^s$. When $r=s$, we write $\RR^{r,r}=\CC^r$ where the involution becomes complex conjugation.
\begin{definition}
	Let $(X,f)$ be a real space and $\xi$ a \textit{complex} vector bundle over $X$ equipped with an involution $j$ covering the involution $f$ on $X$ such that $j: \xi_x\to \xi_{fx}$ is $\CC$-antilinear for all $x\in X$.
	\begin{enumerate}[label=(\roman*)]
		\item If $j^2\equiv 1$, then $(\xi,j)$ is called a \textbf{Real bundle}, or simply a \textbf{R-bundle}, over $(X,f)$.
		\item If $j^2\equiv -1$, then $(\xi,j)$ is called a \textbf{Quaternionic bundle}, or simply a \textbf{Q-bundle} over $(X,f)$.
	\end{enumerate}
	
\end{definition}

\begin{definition}
	For a real space $X$ (suppressing the involution), we define $\KR(X)$ and $\KQ(X)$ respectively to be the the Grothendieck groups of R-bundles and Q-bundles over $X$ with respect to direct sum. The reduced groups $\widetilde{\KR}$ and $\widetilde{\KQ}$ for pointed real spaces (the base point is fixed by involution) are defined to be the kernel of restriction to base point. The relative groups $\KR(X,Y)$ and $\KQ(X,Y)$ for a real pair $(X,Y)$ are defined to be $\widetilde{\KR}(X/Y)$ and $\widetilde{\KQ}(X/Y)$ respectively. The higher groups are defined by\footnote{the order of $(r,s)$ here coincides with \cite{LM89} but is the opposite of that in \cite{KR}.}
		\begin{align*}
			\KR^{r,s}(X,Y)&:=\KR(X\times D^{r,s},X\times S^{r,s}\cup Y\times D^{r,s}),\\
			\KQ^{r,s}(X,Y)&:=\KQ(X\times D^{r,s},X\times S^{r,s}\cup Y\times D^{r,s}).
		\end{align*}
		Here $D^{r,s}$ and $S^{r,s}$ are the unit disk and unit sphere in $\RR^{r,s}$ respectively with restricted involutions. For $X$ a locally compact Hausdorff real space, we define compactly supported groups to be
		\begin{align*}
			\KR^{r,s}_{cpt}(X)&:=\widetilde{\KR}^{r,s}(X^{cpt}),\\
			\KQ^{r,s}_{cpt}(X)&:=\widetilde{\KQ}^{r,s}(X^{cpt}).
		\end{align*}
		where $X^{cpt}=X\cup\{\pt\}$ is the one-point compactification of $X$. If $X$ is compact then $\KR^{r,s}_{cpt}(X)=\KR^{r,s}(X)$.
\end{definition}

We will need the following facts about $\KR$ and $\KQ$ theories drawn from \cite{KR} and \cite{KQ}.
\begin{enumerate}[label=(\roman*)]
	\item If the involution on $X$ is trivial (the identity map), then $\KR(X)=\KO(X)$ and $\KQ(X)=\KSp(X)$.
	\item $\KR$ is a multiplicative theory and $\KQ$ is a module theory over $\KR$. The multiplication on $\KR$ and the module multiplication of $\KR$ on $\KQ$ are both induced from tensor product over $\CC$. Given R-bundles $\pi:(\xi,j)\to X$ and $\pi':(\xi',j')\to X'$, the complex vector bundle $\pi^*\xi\otimes_\CC(\pi')^*\xi'$ equipped with $J=\pi^*j\otimes (\pi')^*j'$ is a Real bundle over $X\times X'$. If either $\xi$ or $\xi'$ is a $Q$-bundle then $\pi^*\xi\otimes_\CC(\pi')^*\xi'$ is a Q-bundle. In particular $\KR^{\bullet,\bullet}(\pt)=\bigoplus_{r,s\ge 0}\KR^{r,s}(\pt)$ is a bigraded ring and $\KQ^{\bullet,\bullet}(\pt)=\bigoplus_{r,s\ge 0}\KQ^{r,s}(\pt)$ is a bigraded module over $\KR^{\bullet,\bullet}(\pt)$.
	\item $\KR^{1,1}(\pt)=\ZZ$ and multiplication by the generator of $\KR^{1,1}(\pt)$ yields isomorphisms
	\begin{align*}
		&\KR^{r,s}(X)\xrightarrow{\cong}\KR^{r+1,s+1}(X),\\
		&\KQ^{r,s}(X)\xrightarrow{\cong}\KQ^{r+1,s+1}(X).
	\end{align*}
	\item $\KQ^{4,0}(\pt)=\ZZ$ and multiplication by the generator of $\KQ^{4,0}(\pt)$ yields an isomorphism
	\begin{equation*}
		\KR^{r,s}(X)\xrightarrow{\cong}\KQ^{r+4,s}(X).
	\end{equation*}
	\item For a locally compact Hausdorff real space $X$, $\KR^{r,s}_{cpt}(X)=\KR_{cpt}(X\times \RR^{r,s})$ and $\KQ^{r,s}_{cpt}(X)=\KQ_{cpt}(X\times \RR^{r,s})$. In particular $\KR^{r,0}(\pt)=\KR_{cpt}(\RR^r)=\KO_{cpt}(\RR^r)=\KO^{-r}(\pt)$. Similarly $\KQ^{r,0}(\pt)=\KSp^{-r}(\pt)$.
\end{enumerate}
In fact all the above are easy to prove except for (iii) and (iv). The isomorphisms in (iii) are called \textbf{(1,1)-periodicity} theorems for $\KR$- and $\KQ$-theories. The former is proved by Atiyah \cite[Theorem 2.3]{KR} and the latter by Dupont \cite{KQ}. The isomorphism in (iv) is proved by Dupont \cite{KQ} implicitly, we justify this in \Cref{appendixB}.

\subsubsection{R- and Q-modules} Now, to connect Clifford modules to $\KR$- and $\KQ$-theories, we consider the \textbf{algebra-with-involution} $\Cl(\RR^{r,s})$ which is $\Cl_{r,s}$ equipped with the involution $c:\Cl(\RR^{r,s})\to \Cl(\RR^{r,s})$ extended from the involution on $\RR^{r,s}$, $(x,y)\mapsto (x,-y)$ for $x\in\RR^r$ and $y\in \RR^s$. Since the involution on $\RR^{r,s}$ commutes with the antipodal map, $\Cl(\RR^{r,s})$ still carries a $\ZZ_2$-grading.

We note that the inclusion $i_*:\Cl_{r,s}\subset \Cl_{r+1,s}$ induced from $\RR^r\times\RR^s\hookrightarrow\RR^{r+1}\times\RR^s$, $(x,y)\mapsto (x,0,y)$ is compatible with the involutions and therefore gives rise to an inclusion of algebras-with-involution $i_*:\Cl(\RR^{r,s})\subset\Cl(\RR^{r+1,s})$. Also observe that the $\ZZ_2$-graded isomorphism $\Cl_{r,s}\hat{\otimes}\Cl_{r',s'}\cong\Cl_{r+r',s+s'}$ carries the involution $c\hat{\otimes}c'$ to the involution on $\Cl_{r+r',s+s'}$. So we have an isomorphism of $\ZZ_2$-graded algebras-with-involution $\Cl(\RR^{r,s})\hat{\otimes}\Cl(\RR^{r',s'})\cong\Cl(\RR^{r+r',s+s'})$.
\begin{definition}
	Let $V$ be a complex module over $\Cl_{r,s}$ together with a $\CC$-antilinear map $c:V\to V$ such that
	\[
	c(a\cdot v)=c(a)\cdot c(v)
	\]
	for all $a\in \Cl(\RR^{r,s})$ and $v\in V$.
	\begin{enumerate}[label=(\roman*)]
		\item If $c^2\equiv 1$, then $V$ is called a \textbf{R-module} over $\Cl(\RR^{r,s})$.
		\item If $c^2\equiv -1$, then $V$ is called a \textbf{Q-module} over $\Cl(\RR^{r,s})$.
	\end{enumerate}
	If further $V$ is a $\ZZ_2$-graded module over $\Cl_{r,s}$ and $c(V^\alpha)=V^\alpha$ for $\alpha=0,1$, then $V$ is called a $\ZZ_2$-graded R- or Q-module over $\Cl(\RR^{r,s})$. Denote by $\hat{\frM} R_{r,s}$ and $\hat{\frM} Q_{r,s}$ the Grothendieck groups of finite dimensional $\ZZ_2$-graded R- and Q-modules over $\Cl_{r,s}$ respectively. And denote by $\hat{\frN} R_{r,s}$ and $\hat{\frN} Q_{r,s}$ respectively the cokernels of
	$i^*:\hat{\frM} R_{r+1,s}\to \hat{\frM}R_{r,s}$ and $i^*:\hat{\frM} Q_{r+1,s}\to \hat{\frM}Q_{r,s}$ induced by $i_*:\Cl(\RR^{r,s})\subset \Cl(\RR^{r+1,s})$.
\end{definition}
We will see on the one hand R- and Q-modules over $\Cl(\RR^{r,s})$ can be identified with $\RR$- and $\HH$-modules over $\Cl_{r,s}$ respectively, and on the other hand they can be directly related to KR- and KQ-theories through Atiyah-Bott-Shapiro type constructions.

For a $\RR$-module $V$ over $\Cl_{r,s}$, we consider the $\CC$-vector space $V_\CC=V\otimes_\RR \CC=\Ind_\RR^\CC(V)$ endowed with the $\CC$-antilinear map given by complex conjugation and with the $\Cl(\RR^{r,s})$-action determined by
\[
(x,y)\cdot v:=xv+\mathbf{i}yv
\]
for all $(x,y)\in \RR^r\times \RR^s=\RR^{r,s}$. For an $\HH$-module $W$ over $\Cl_{r,s}$, we consider the underlying $\CC$-vector space $W_\CC=\Res_\CC^\HH(W)$ endowed with the $\CC$-antilinear map given by multiplication by $\mathbf{j}\in \HH=\CC+\CC \mathbf{j}$ and with the $\Cl(\RR^{r,s})$-action determined by
\[
(x,y)\cdot w=xw+\mathbf{i}yw
\]
for all $(x,y)\in \RR^r\times \RR^s=\RR^{r,s}$. It is straightforward to check $V_\CC$ (resp. $W_\CC$) is a R-module (resp. Q-module) over $\Cl(\RR^{r,s})$ and the functors
\begin{align*}
	&\text{$\RR$-modules over $\Cl_{r,s}$}\xrightarrow{\Ind_\RR^\CC} \text{R-modules over $\Cl(\RR^{r,s})$}\\
	&\text{$\HH$-modules over $\Cl_{r,s}$}\xrightarrow{\Res_\RR^\CC} \text{Q-modules over $\Cl(\RR^{r,s})$}
\end{align*}
are isomorphisms of categories that preserve $\ZZ_2$-gradings. Moreover, these functors are compatible with the inclusions of algebras $\Cl(\RR^{r,s})\subset\Cl(\RR^{r+1,s})$, $\Cl_{r,s}\subset\Cl_{r+1,s}$. Therefore they induce isomorphisms
\begin{align*}
	&\hat{\frM}_{r,s}(\RR)\xrightarrow{\cong}\hat{\frM}R_{r,s}, \quad \hat{\frN}_{r,s}(\RR)\xrightarrow{\cong}\hat{\frN}R_{r,s},\\
	&\hat{\frM}_{r,s}(\HH)\xrightarrow{\cong}\hat{\frM}Q_{r,s}, \quad \hat{\frN}_{r,s}(\HH)\xrightarrow{\cong}\hat{\frN}Q_{r,s}.
\end{align*}

Now $\ZZ_2$-graded tenor product can be defined for R- and Q-modules. Let $(V,c)$, $(V',c')$ be $\ZZ_2$-graded R- or Q-modules over $\Cl(\RR^{r,s})$ and $\Cl(\RR^{r',s'})$ respectively. Then $V\hat{\otimes}_\CC V'$ equipped with $c\hat{\otimes }c'$ is a $\ZZ_2$-graded module over $\Cl(\RR^{r+r',s+s'})$. If both $V$ and $V'$ are R-modules or both are Q-modules, then $V\hat{\otimes}V'$ is a R-module, and otherwise a Q-module. With these understood, $\hat{\frM}R_{\bullet,\bullet}=\bigoplus_{r,s\ge 0}\hat{\frM}R_{r,s}$ is a bigraded ring and $\hat{\frM}Q_{\bullet,\bullet}=\bigoplus_{r,s\ge 0}\hat{\frM}Q_{r,s}$ is a bigraded module over $\hat{\frM}R_{\bullet,\bullet}$. These structures descend to make $\hat{\frN}R_{\bullet,\bullet}=\bigoplus_{r,s\ge 0}\hat{\frN}R_{r,s}$ into a bigraded ring and $\hat{\frN}Q_{\bullet,\bullet}=\bigoplus_{r,s\ge 0}\hat{\frN}Q_{r,s}$ a bigraded module over $\hat{\frN}R_{\bullet,\bullet}$. It is now easy and left to the reader to check the above isomorphisms assemble to isomorphisms of bigraded rings and modules.

\subsubsection{Bigraded ABS isomorphisms}
The Atiyah-Bott-Shapiro construction applies to R- and Q-modules over $\Cl(\RR^{r,s})$ without difficulty and gives rise to homomorphisms
\begin{align*}
	&\varphi^R: \hat{\frN}_{\bullet,\bullet}(\RR)\cong\hat{\frN}R_{\bullet,\bullet}\to \KR^{\bullet,\bullet}(\pt),\\
	&\varphi^Q: \hat{\frN}_{\bullet,\bullet}(\HH)\cong\hat{\frN}Q_{\bullet,\bullet}\to \KQ^{\bullet,\bullet}(\pt).
\end{align*}

To prove these maps are isomorphisms, we need to better understand $\RR$- and $\HH$-modules over the algebras $\Cl_{r,s}$'s. Denote $\Cl_{r,s}\otimes_\RR\HH$ by $\Cl_{r,s}^h$ and think of $\HH$-modules over $\Cl_{r,s}$ as $\RR$-modules over $\Cl_{r,s}^h$.
\begin{proposition}\label{bigradedhalfbott}
	For $r,s\ge 0$, there are isomorphisms of $\RR$-algebras
	\begin{enumerate}[label=(\roman*)]
		\item $\Cl_{r+1,s+1}\cong \Cl_{r,s}\otimes_\RR\RR(2)$;
		\item $\Cl_{r+4,s}^h\cong \Cl_{r,s}\otimes_\RR \RR(8)$;
		\item $\Cl_{r+4,s}\cong \Cl_{r,s}^h\otimes_\RR\RR(2)$.
	\end{enumerate}
	These isomorphisms are compatible with the embeddings $\RR^r\times \RR^s\hookrightarrow\RR^{r+1}\times \RR^s$, $(x,y)\mapsto (x,0,y)$ for $r,s\ge 0$.
\end{proposition}
\begin{proof}
	For a proof of (i), see \cite[\S 1. Theorem 4.1]{LM89}. We note in particular $\Cl_{1,1}=\RR(2)$. Both (ii) and (iii) follow from \labelcref{indefz2gradediso} and \Cref{halfbott}.
\end{proof}

From here, we can quickly deduce the following
\begin{proposition}\label{bigradedz2iso}
	\begin{enumerate}[label=(\roman*)]
		\item Multiplication by the generator of $\hat{\frM}_{1,1}(\RR)$ gives isomorphisms
		\[
		\hat{\frM}_{r,s}(\KK)\xrightarrow{\cong}\hat{\frM}_{r+1,s+1}(\KK),\quad \hat{\frN}_{r,s}(\KK)\xrightarrow{\cong}\hat{\frN}_{r+1,s+1}(\KK)
		\]
		for $\KK=\RR$ or $\HH$ and $r,s\ge 0$.
		\item Multiplication by $\Delta_{4,\HH}^+\in \hat{\frM}_4(\HH)=\hat{\frM}_{4,0}(\HH)$ gives isomorphisms
		\[
		\hat{\frM}_{r,s}(\RR)\xrightarrow{\cong}\hat{\frM}_{r+4,s}(\HH),\quad \hat{\frN}_{r,s}(\RR)\xrightarrow{\cong}\hat{\frN}_{r+4,s}(\HH)
		\]
		for $r,s\ge 0$.
	\end{enumerate}
\end{proposition}
\begin{proof}
	The algebras $\Cl_{r,s}$ and $\Cl_{r,s}^h$ are of the form $\KK(N)$ for $r-s\not\equiv 3\bmod 4$, and of the form $\KK(N)\oplus\KK(N)$ for $r-s\not\equiv 3\bmod 4$ for some $N$ and $\KK\in\{\RR,\CC,\HH\}$. This from the classification of the algebras $\Cl_{r,s}$'s, see e.g. \cite[pp.27-29]{LM89}, and the identities \labelcref{fundamentalid}.
	
	 Now let $\KK=\RR$ or $\HH$. As long as $r-s\not\equiv 3\bmod 4$, $\Cl_{r,s}$ has a unique (up to equivalence) irreducible ungraded $\KK$-module. We can turn these ungraded modules into $\ZZ_2$-graded modules by applying the functor $\Cl_{r+1,s}\otimes_{\Cl_{r,s}}-$ where $\Cl_{r,s}$ is identified with $\Cl_{r+1,s}^0$ similar to the way $\Cl_{n}$ is identified with $\Cl_{n+1}^0$. So for $r-s\not\equiv 0\bmod 4$, we get $\hat{\frM}_{r,s}(\KK)=\ZZ$. As for $r-s\equiv 0\bmod 4$, since $\Cl_{r-1,s}$ has two inequivalent irreducible ungraded $\KK$-modules, $\Cl_{r,s}$ has two inequivalent $\ZZ_2$-graded $\RR$- or $\HH$-modules. So $\hat{\frM}_{r,s}(\KK)=\ZZ+\ZZ$. The two generators can also be explicitly constructed by looking the matrix multiplication of $\Cl_{r,s}\cong\KK(N)$ on $\KK^{N}$ and examining the the action of the volume element $\omega_{r,s}=e_1 e_2\cdots e_{r+s}\in\Cl_{r,s}$ as follows. Note that $\omega_{r,s}^2=(-1)^{[(r+s)^2+(r-s)]/2}$ so in particular $\omega_{r,s}^2=1$ when $r-s\equiv 0\bmod 4$. So similar to \Cref{sec:irredmod}, the two inequivalent $\ZZ_2$-graded modules are constructed by putting the $\pm 1$ eigenspaces of $\omega_{r,s}$ in even or odd degrees. Let us denote the two such irreducible $\ZZ_2$-graded $\RR$-modules over $\Cl_{1,1}$ by $\Delta_{1,1}^\pm$.
	 
	 The proposition now follows from simple dimension counts and from looking at the actions of the volume elements. To elaborate, we note that the real dimension of $\Delta_{1,1}^+$ is $2$ which implies the $\ZZ_2$-graded tensor product of $\Delta_{1,1}^+$ with an irreducible $\ZZ_2$-graded $\KK$-module over $\Cl_{r,s}$ is irreducible over $\Cl_{r+1,s+1}$ by dimension counts, in light of \Cref{bigradedhalfbott}(i). We also note that $\omega_{r,s}\hat{\otimes}\omega_{r',s'}$ is identified with $\omega_{r+r',s+s'}$ under the isomorphism $\Cl_{r,s}\hat{\otimes}\Cl_{r',s'}\cong \Cl_{r+r',s+s'}$, so by looking at the actions of the volume elements we know the map
	 \[
	 \hat{\frM}_{r,s}(\KK)\to{\hat{\frM}}_{r+1,s+1}(\KK)
	 \]
	 given by multiplication by $\Delta_{1,1}^+$ is onto and therefore is an isomorphism for all $r,s\ge 0$. This proves (i) by noting the above map commutes with $i^*$. The same argument proves (ii) as well.
\end{proof}

Now we are ready to prove
\begin{theorem}
	$\varphi^R$ is an isomorphism of bigraded rings and $\varphi^Q$ is an isomorphism of bigraded modules over the bigraded ring isomorphism $\varphi^R$.
\end{theorem}

\begin{proof}
	That $\varphi^R$ is an isomorphism of bigraded rings is proved by Atiyah in \cite{KR}. The argument therein that proves $\varphi^R$ is a bigraded ring homomorphism also proves $\varphi^Q$ is a map of bigraded modules over $\varphi^R$. To prove $\varphi^Q$ is an isomorphism we use the isomorphisms of algebras established in the previous proposition.
	
	First, the (1,1)-periodicities in \Cref{bigradedz2iso}(i) are compatible with the (1,1)-periodicities for $\KR$ and $\KQ$ since all of them are induced by multiplication with the generator of $\hat{\frN}_{1,1}(\RR)\cong\hat{\frN}R_{1,1}\cong\KR^{1,1}(\pt)$. So using (1,1)-periodicities it suffices to prove $\varphi^Q$ is an isomorphism in bidegrees $(r,s)$ for $r\ge 4, s\ge 0$. Next, \Cref{bigradedz2iso}(ii) says $\hat{\frN}_{\bullet+4,\bullet}(\HH)\cong\hat{\frN}Q_{\bullet+4,\bullet}$ is a free module of rank one over $\hat{\frN}_{\bullet,\bullet}(\RR)\cong\hat{\frN}R_{\bullet,\bullet}$. Since $\KQ^{\bullet+4,\bullet}(\pt)$ is a also free module of rank one over $\KR^{\bullet,\bullet}(\pt)$, it suffices to prove $\varphi^Q:\hat{\frN}_{4,0}(\HH)\to \KQ^{4,0}(\pt)$, which is already proved in \Cref{ABSH} since this map the same as $\varphi^h:\hat{\frN}_4(\HH)\to \KSp^{-4}(\pt)$.
\end{proof}

\subsection{Family of quaternionic operators}
Recall that a complex vector bundle $E$ is said to be quaternionic if $E$ is equipped with a real vector bundle automorphism $j:E\to E$ which is $\CC$-antilinear in each fiber and $j^2=-1$. The complex vector space of smooth sections $\Gamma(E)$ is equipped with a quaternionic structure given by $j^*$.

Suppose now $E,F$ are quaternionic vector bundles over a closed manifold $M$ and $P:\Gamma(E)\to\Gamma(F)$ is a complex elliptic differential operator of order $m$. We say $P$ is quaternionic if $Pj_E^*=j_F^*P$. In local terms
\[
P=\sum_{|\alpha|=m} A^\alpha(x)\partial^{|\alpha|}/\partial x^\alpha+\text{lower order terms}
\]
where the $A^\alpha$'s are complex-matrix-valued functions with $A^\alpha j_E=j_F A^\alpha$. The \textbf{principal symbol} $\sigma_\tau(P)=\sum A^\alpha(x)(\mathbf{i}\tau)^\alpha$ of $P$ in local terms is
\[
e\in E \mapsto \sum_{|\alpha|=m} \mathbf{i}^{|\alpha|} \tau^\alpha\left(A^\alpha(x)\cdot e\right)\in F
\]
for any tangent vector $\tau=\sum \tau^i \partial/\partial x^i$ of $M$. It is easy to see $\sigma(P)$ satisfies
\begin{equation}\label{symbolclass}
	\sigma_\tau(P)j_E=j_F\sigma_{-\tau}(P).
\end{equation}
The symbol class of a quaternionic elliptic differential operator therefore lands in $\KQ$-theory.
\begin{definition}
	Given a closed manifold $M$, consider the tangent bundle $\pi:TM\to M$ to be equipped with the canonical involution $f:TM\to TM$ defined by $f(\tau)=-\tau$, i.e. the fiberwise antipodal map. Given any quaternionic vector bundle $(E,j)$ over $M$, $\pi^*E$ is in a natural way a Quaternionic bundle over the real space $(TM,f)$ by setting
	$J:\pi^*E\to\pi^*E$ to be 
	\[
	J(x,\tau,e)=(x,-\tau,j(e)).
	\]
	Suppose now $E,F$ are quaternionic vector bundles over $M$, then for any quaternionic elliptic operator $P:\Gamma(E)\to\Gamma(F)$, the \textbf{Quaternionic symbol class} of $P$ is defined to be the element
	\[
	[\pi^* E,\pi^*F;\sigma(P)]\in \KQ_{cpt}(TM).
	\]
	Note \labelcref{symbolclass} says $\sigma(P)$ is an isomorphism of Quaternionic bundles outside the zero section of $TM$.
\end{definition}

To define the topological index of a quaternionic elliptic operator, we need a version of Thom isomorphism for KQ-theory.
\begin{theorem}[Atiyah, Dupont]\label{Atiyah-Dupont}
	Let $E$ be a Real bundle over a locally compact Hausdorff real space $X$. Then multiplication by $\dd_R(E)$ induces isomorphisms
	\begin{align*}
		\KR_{cpt}(X)&\xrightarrow{\cong}\KR_{cpt}(E)\\
		\KQ_{cpt}(X)&\xrightarrow{\cong}\KQ_{cpt}(E)
	\end{align*}
\end{theorem}
\begin{remark}
	Locally these Thom isomorphisms are compositions of (1,1)-periodicities. See also \Cref{appendixB}.
\end{remark}

Now we can define the topological index of a quaternionic elliptic operator $P$ as follows. We first choose an embedding $f:M\hookrightarrow \RR^m$. The associated embedding $TM\hookrightarrow T\RR^m$ is compatible with involutions, i.e. is a mapping of real spaces. If $N$ is the normal bundle to $M$ in $\RR^m$, then $\pi^*N\oplus\pi^*N\cong\pi^*N\otimes\CC$ is the normal bundle to $TM$ in $T\RR^m$. We consider this to be a Real bundle over $TM$ (with complex conjugation as its involution). Then similar to the construction in our Riemann-Roch theorem, we can define a map
\begin{equation}\label{KQUmkehr}
	f_!: \KQ_{cpt}(TM)\to \KQ_{cpt}(T\RR^m)
\end{equation}
by composing the Thom isomorphism with the map induced by the inclusion of the normal bundle as a tubular neighborhood of $TM$ in $T\RR^m$. This inclusion can be easily chosen to be compatible with involutions. We now identify $T\RR^m=\RR^{m,m}=\CC^m$, then $\KQ_{cpt}(T\RR^m)\cong \KQ^{m,m}(pt)\cong \KQ^{0,0}(\pt)\cong\ZZ$. Therefore we can define the \textbf{topological index} of $P$ to be the integer $f_!(\sigma(P))$.

As usual, the fact that the topological index is independent of our choice of the embedding follows from the multiplicative property of the $\KR$-Thom class for Real bundles.

The discussion of symbol class and topological index naturally extends to families of quaternionic operators.
\begin{definition}
	Let $P$ be a family of quaternionic elliptic operators on a closed manifold $M$ parameterized by a compact Hausdorff space $A$. Let $\mathcal{M}\to A$ denote the underlying family of manifolds, and let 
	$T\mathcal{M}=\cup_{a\in A}T\mathcal{M}_a$ be the tangent bundle of the family $\mathcal{M}$. Let $\boldsymbol{\sigma}(P)\in \KQ_{cpt}(T\mathcal{M})$ denote the symbol class of the family. Then the topological index of the family $P$ is defined to be the element
	\[
	\tind(P)=q_!f_!\boldsymbol{\sigma}(P)\in\KQ(A)\cong\KSp(A)
	\]
	where $f_!:\KQ_{cpt}(T\mathcal{M})\to \KQ_{cpt}(A\times T\RR^m)$ is constructed similar to \labelcref{KQUmkehr} and $q_!:\KQ_{cpt}(A\times\CC^n)\to \KQ(A)$ is the natural isomorphism given by the Thom isomorphism.
\end{definition}
%\begin{remark}
%	This definition can be extended to a family of operators parametrized by a locally compact Hausdorff space, provided that the principal symbol of $P$ is invertible outside a compact subset of the parameter space.
%\end{remark}
The forgetful morphism $\Res_\CC^\HH:\KSp(A)\to \KU(A)$ is not always injective, so the index we just defined is more refined than the usual index of $P$ as a family of complex operators.

One can define the \textbf{analytic index} for such a family $P$ of quaternionic elliptic operators by putting
\[
\aind(P)=[\ker P]-[\coker P]\in \KSp(A).
\]
To be more precise, if the dimensions of $\ker P_a$ and $\coker P_a$ are constant for $a\in A$, then $\ker P$ and $\coker P$ define two quaternionic bundles over $A$. In this case, $\aind(P)$ is defined to be the difference class $[\ker P]-[\coker P]$. In general, $\ker P_a$ and $\coker P_a$ are not constant dimensional, then we must first "stabilize" the situation as Atiyah and Singer did in the complex case in \cite{AS4}; we simply note the treatment in \cite[sec.2]{AS4} can be easily made to respect the quaternionic structures.

The analytic index, of course, coincides with the topological index.

\begin{theorem}[Index theorem for quaternionic family]\label{indexthm}
	Let $P$ be a family of quaternionic elliptic operators on a closed manifold parametrized by a compact Hausdorff space $A$. Then
	\[
	\aind(P)=\tind(P).
	\]
\end{theorem}

We sketch below the proof of this quaternionic version of index theorem, which proceeds just as in the case of real and complex families \cite{AS4, AS5}.
 Recall that such index theorem for families essentially relies on checking the following three axioms.
 \begin{enumerate}[label=(\roman*)]
 	\item The analytic index $$\aind: \KQ_{cpt}(T\mathcal{M})\to \KQ(A)$$ is a homomorphism of $\KR(A)$-modules, which in the special case $\mathcal{M}=A=\pt$ is the identity map.
 	\item(Excision) Let $\mathcal{M}\to A$ and $\mathcal{M}'\to A$ be two families over $A$ with compact fibers $M, M'$ respectively and let $f: \mathcal{O}\hookrightarrow \mathcal{M}$, $f':\mathcal{O}'\hookrightarrow \mathcal{M}'$ be inclusions of open sets, with a smooth equivalence $\mathcal{O}\cong\mathcal{O}'$ compatible with the maps to $A$. Then, identifying $\mathcal{O}'$ with $\mathcal{O}$, the following diagram commutes:
\begin{center}
	\begin{tikzcd}[row sep=small, column sep=small]
		& \KQ_{cpt}(T\mathcal{M})\arrow[dr,"\aind"] & \\
		\KQ_{cpt}(T\mathcal{O})\arrow[ur,"f_!"]\arrow[dr,"f'_!"']& & \KQ(A)\\
		& \KQ_{cpt}(T\mathcal{M}')\arrow[ur,"\aind"'] &
	\end{tikzcd}
\end{center}
	\item(Multiplicativity) Let $\mathcal{E}\to\mathcal{M}$ be a family of oriented smooth vector bundles of rank $n$, and let $\mathcal{S}=S(\mathcal{E}\oplus\RR)$ be the family of $n$-sphere bundle compactified from $\mathcal{E}$. Then the following diagram commutes:
\begin{center}
\begin{tikzcd}[row sep=small, column sep=small]
	\KQ_{cpt}(T\mathcal{M})\arrow[rr,"i_!"]\arrow[rd,"\aind"']& &\KQ_{cpt}(T\mathcal{S})\arrow[ld,"\aind"]\\
	& \KQ(A) &
\end{tikzcd}
\end{center}
where $i_!$ is multiplication by the fundamental equivariant symbol $b\in \KR_{\SO_n}(TS^n)_{cpt}$ (cf. \cite{AS5}).
 \end{enumerate}

%\begin{lemma}\label{aind-hom}
%	The analytic index $$\aind: \KQ_{cpt}(T\mathfrak{X})\to \KQ(A)$$ is a homomorphism of $\KR(A)$-modules, which in the special case $\mathfrak{X}=A=\pt$ is the identity map. 
%\end{lemma}
%
%\begin{lemma}[excision]
%	Let $\mathfrak{X}\to A$ and $\mathfrak{X}'\to A$ be two families over $A$ with compact fibers $X, X'$ respectively and let $f: \mathcal{O}\hookrightarrow \mathfrak{X}$, $f':\mathcal{O}'\hookrightarrow \mathfrak{X}'$ be inclusions of open sets, with a smooth equivalence $\mathcal{O}\cong\mathcal{O}'$ compatible with the maps to $A$. Then, identifying $\mathcal{O}'$ with $\mathcal{O}$, the following diagram commutes:
%\begin{center}
%	\begin{tikzcd}
%		& \KQ_{cpt}(T\mathfrak{X})\arrow[dr,"\aind"] & \\
%		\KQ_{cpt}(T\mathcal{O})\arrow[ur,"f_!"]\arrow[dr,"f'_!"']& & \KQ(A)\\
%		& \KQ_{cpt}(T\mathfrak{X}')\arrow[ur,"\aind"'] &
%	\end{tikzcd}
%\end{center}
%\end{lemma}
%
%\begin{lemma}[multiplicativity]
%Let $\mathcal{E}\to\mathfrak{X}$ be a family of oriented smooth vector bundles of rank $n$, and let $\mathcal{S}=S(\mathcal{E}\oplus\RR)$ be the family of $n$-sphere bundle compactified from $\mathcal{E}$. Then the following diagram commutes:
%\begin{center}
%\begin{tikzcd}
%	\KQ_{cpt}(T\mathfrak{X})\arrow[rr,"i_!"]\arrow[rd,"\aind"']& &\KQ_{cpt}(T\mathcal{S})\arrow[ld,"\aind"]\\
%	& \KQ(A) &
%\end{tikzcd}
%\end{center}
%where $i_!$ is multiplication by the fundamental equivariant symbol $b\in \KR_{\SO_n}(TS^n)_{cpt}$ (cf. \cite{AS5}).
%\end{lemma}

The argument of \cite{AS4, AS5} for excision and multiplicativity goes through easily in the quaternionic case, only the first axiom requires some special attention. To make sense of and prove the first axiom, we need the following facts. First the analytic index depends only on the homotopy class of the symbol class, which is a consequence of \cite[Main Theorem III]{Matumoto}. Second, every element in $\KQ_{cpt}(T\mathcal{M})$ can be represented by some symbol class. And finally the homomorphism $\aind$ is well-defined, i.e it does not depend on the choice of symbol-class-representatives. The second and the last points can be proved no differently from the real and complex cases.

\subsection{Topological formula of $\Cl_{k}^h$-index}
Assume $E$ is a $\ZZ_2$-graded $\Cl_{k}^h$-bundle over a closed Riemannian manifold $X$. Further assume $E$ carries 
a bundle metric for which the Clifford multiplication by unit vectors in $\RR^k$ is orthogonal and the multiplication by unit quaternions is orthogonal. Let $P:\Gamma(E)\to\Gamma(E)$ be an elliptic \textit{self-adjoint} operator and assume $P$ is $\Cl_{k}^h$-linear and $\ZZ_2$-graded. Recall we defined the index $\ind^h_k(P)\in \KSp^{-k}(\pt)$ in terms of the $\Cl_{k}^h$-module $\ker P$. We shall now give a topological formula for this index using \Cref{indexthm}.

Since $P$ and $(1+P^* P)^{-1/2} P$ have the same kernel, we may assume $P$ has degree zero. With respect to the splitting $E=E^0\oplus E^1$, $P$ can be written as
\[
P=\begin{pmatrix}
	0& P^1\\
	P^0 &0
\end{pmatrix}
\]
where $P^1=(P^0)^*$. Now we construct a family $\mathscr{P}$ of quaternionic elliptic operators parametrized by $\RR^k$ by assigning to each $v\in\RR^k$ the operator
\[
\mathscr{P}_v^0:\Gamma(E^0)\to \Gamma(E^1)
\]
defined by the restriction to $E^0$ of the operator
\[
\mathscr{P}_v:=v+P
\]
where "$v$" denotes Clifford multiplication by $v$. Since both Clifford multiplication and $P$ are $\HH$-linear, so is $\mathscr{P}$. Also since $P$ commutes with Clifford multiplication, there is a "conjugate" family $\overline{\mathscr{P}}_v=v-P$ satisfying
\[
\overline{\mathscr{P}}_v\mathscr{P}_v=\mathscr{P}_v\overline{\mathscr{P}}_v=-(\|v\|^2+P^2).
\]
Therefore $\mathscr{P}_v^0$ is invertible for all $v\neq 0$. Since the space of invertible $\HH$-linear operators on a quaternionic Hilbert space is contractible (see \cite{Segal}, \cite{Matumoto}), we could pass to a family parametrized by $S^k$ by embedding $\RR^k$ into $S^k$ and extending the family to $S^k$ by invertible operators outside $\RR^k$. The index of this extended family lies in $\KSp(S^k)$ and becomes trivial when restricted to any point outside $\RR^k$, so the index in fact lies in the kernel of $\KSp(S^k)\to \KSp(\pt)$ which is isomorphic to $\KSp_{cpt}(\RR^k)$. Alternatively, we can treat $\mathscr{P}^0$ as a family with compact support and directly define its topological index and analytic index in $\KSp_{cpt}(\RR^k)$. The topological index and analytic index so defined coincide by the "passing to $S^k$" argument and \Cref{indexthm}. We will take the latter point of view, for it is more illustrating.

\begin{theorem}\label{Cliffordindex}
	Let $P$ a zero-order elliptic self-adjoint $\ZZ_2$-graded $\Cl_{k}^h$-operator on a closed manifold $M$. Then
	\[
	\ind^h_k(P)=\aind(\mathscr{P}^0).
	\]
\end{theorem}
\begin{proof}
Set $K^0=\ker P^0\subset\Gamma(E^0)$ and $K^1=\ker P^1\cong\coker(P^0)\subset\Gamma(E^1)$. Then $K^0,K^1$ are finite dimensional $\HH$-subspaces of $\Gamma(E^0)$ and $\Gamma(E^1)$ respectively. By assumption the quaternionic structure on $E$ is compatible with the its bundle metric, so there are $L^2$-orthogonal compliments $V^0, V^1$ to $K^0, K^1$ respectively. Then the family $\mathscr{P}_v^0$ decomposes as a direct sum of two operators: the first summand $V^0\xrightarrow{\mathscr{P}_v^0} V^1$ is an $\HH$-isomorphism for all $v\in\RR^k$, thus can be ignored for the purpose of computing the index; meanwhile the second summand is just $K^0\xrightarrow{\mathscr{P}_v^0=v} K^1$ which is independent of variables on $M$. Therefore the analytic index of $\mathscr{P}^0$ is
	\[
	\aind(\mathscr{P})=[K^0,K^1;v]\in \KSp_{cpt}(\RR^k)\cong\KSp^{-k}(\pt).
	\]
	Under the isomorphism $\KSp^{-k}(\pt)\cong\hat{\frN}_k(\HH)$, this corresponds exactly to the element represented by $\ker P=K^0\oplus K^1$, i.e. it corresponds exactly to $\ind^h_k(P)$.
	\end{proof}
Since $\aind(\mathscr{P}^0)=\tind(\mathscr{P}^0)$, \Cref{Cliffordindex} gives a topological formula for $\ind_k^h(P)$.

\subsection{Proof of \Cref{spinhindexthm}}
This subsection is devoted to proving \Cref{spinhindexthm}. We follow Hitchin's  treatment in the spin case, see \cite{Hitchin}. We will only prove the untwisted case, the twisted case is similar and thus omitted.

Let $M$ be a closed spin$^h$ manifold of dimension $n$ carrying a canonical $\Cl_n^h$-Dirac bundle $\slashed{\mathfrak{S}}(M)=P_{\Spin^h}\times_l\Cl_{n}^h$ with Dirac operator $\slashed{\mathfrak{D}}$. We can turn $\slashed{\mathfrak{D}}$ into a zero-order operator $\slashed{\mathfrak{D}}_Q:=Q\slashed{\mathfrak{D}} Q^*$ with isomorphic kernel and the same symbol (see \cite[p. 39]{Hitchin}), where $Q=(1+\nabla^*\nabla)^{-1/4}$ and $\nabla$ is the covariant derivative. Then by \Cref{Cliffordindex} and \Cref{indexthm} we have
\[
\aind(M)=\ind_n^h(\slashed{\mathfrak{D}})=\ind_n^h(\slashed{\mathfrak{D}}_Q)=\tind(\slashed{\mathscr{D}}_Q^0)=\tind(\slashed{\mathscr{D}}^0),
\]
where $\slashed{\mathscr{D}}_Q^0=v+\slashed{\mathfrak{D}}_Q^0$ and $\slashed{\mathscr{D}}^0=v+\slashed{\mathfrak{D}}^0$. Note the last equality follows from that $\slashed{\mathfrak{D}}_Q$ and $\slashed{\mathfrak{D}}$ have the same symbol. So to prove \Cref{spinhindexthm}, it suffices to prove $\tind(\slashed{\mathscr{D}}^0)=\tind(M)$.

For this, choose a smooth embedding $M\hookrightarrow\RR^{n+8k+4}$. Denote the normal bundle of $M$ by $N$ and identify $N$ with a tubular neighborhood of $M$. Consider the following natural diagram of embeddings
\[
\begin{tikzcd}
	M \ar[r,"\kappa"]\ar[d,"\iota_M"]& N\ar[r,"\lambda"]\ar[d,"\iota_N"] & \RR^{n+8k+4}\ar[d,"\iota_\RR"]\\
	TM \ar[r,"\overline{\kappa}"]& TN \ar[r,"\overline{\lambda}"]& T\RR^{n+8k+4}=\CC^{n+8k+4}
\end{tikzcd}
\]
We claim there is a corresponding commutative diagram
\[
\begin{tikzcd}[row sep=small, column sep=small]
	\KR(M)\ar[r,"\kappa_!"]\ar[d,"(\iota_M)_!"] & \KR_{cpt}(N)\ar[r,"\lambda_!"]\ar[d,"(\iota_N)_!"] & \KR_{cpt}(\RR^{n+8k+4})\ar[d,"(\iota_\RR)_!"]\\
	\KQ_{cpt}(\RR^n\times TM)\ar[r,"\overline{\kappa}_!"] & \KQ_{cpt}(\RR^n\times TN) \ar[r,"\overline{\lambda}_!"]& \KQ_{cpt}(\RR^n\times\CC^{n+8k+4})
\end{tikzcd}
\]
with the following properties:
\begin{enumerate}[label=(\roman*)]
	\item $\lambda_!$, $\overline{\lambda}_!$ are the maps induced by the open embeddings $\lambda$ and $\overline{\lambda}$.
	\item $\kappa_!$ is the Thom homomorphism using the induced spin$^h$ structure on the normal bundle $N$. Note $M,N,\RR^{n+8k+4}$ are equipped with trivial involutions.
	\item $\overline{\kappa}_!$ is the Thom isomorphism using the Real structure on $TN\cong\pi_M^*N\otimes\CC$ where $\pi_M:TM\to M$ is the bundle projection.
	\item $(\iota_M)_!$ is multiplication by the symbol class of $\slashed{\mathscr{D}}^0$. 
	\item $(\iota_\RR)_!$ is an isomorphism and commutes with the periodicity isomorphisms $q_!:\KR_{cpt}(\RR^{n+8k+4})\cong \KSp^{-n}(\pt)$ and $\overline{q}_!:\KQ_{cpt}(\RR^n\times\CC^{n+8k+4})\cong\KSp^{-n}(\pt)$, that is $q_!=(\iota_\RR)_!\overline{q}_!$.
\end{enumerate}

By definition, $\tind(\slashed{\mathscr{D}}^0)=\overline{q}_!\overline{\lambda}_!\overline{\kappa}_!(\iota_M)_!(1)$ and $\tind(M)=q_!\lambda_!\kappa_!(1)$. Hence \Cref{spinhindexthm} follows immediately from the claim.

Now to prove the claim, we will first establish a local version of the claim, and then explain how to globalize. The local version is the following commutative diagram:
\[
\begin{tikzcd}
	\hat{\frN}R_{n,0}\ar[r,"\kappa_!"]\ar[d,"(\iota_M)_!"'] & \hat{\frN}R_{n+8k+4,0}\ar[d,"(\iota_N)_!"'] & \hat{\frN}Q_{n,0}\ar[l,"q_!"']\ar[dl,"\overline{q}_!"]\\
	\hat{\frN}Q_{2n,n}\ar[r,"\overline{\kappa}_!"'] & \hat{\frN}Q_{2n+8k+4, n+8k+4} &
\end{tikzcd}
\]
where the names of the maps are chosen to be the same as the their corresponding global maps. Here the maps are defined as follows:
\begin{enumerate}[label=(\roman*)]
	\item $\kappa_!$ is multiplication by the R-module $\Delta_{8k+4}^h\otimes_\RR\CC$ over $\Cl(\RR^{8k+4,0})$. Note $\Delta_{8k+4}^h$ is a $\HH$-module over $\Cl_{8k+4}$ and thus in particular a $\RR$-module. Also recall the way to turn a $\RR$-module into a R-module is to tensor with $\CC$.
	\item $q_!$ is the Bott periodicity isomorphism given by multiplication by the Q-module $\Delta_{8k+4}^h$ over $\Cl(\RR^{8k+4,0})$.
	\item $(\iota_M)_!$ is multiplication by the Q-module $\Cl_n^h$ over $v\in \Cl(\RR^{n,n})$, where $(x,y)\in \RR^{n,n}$ acts on $\Cl_n^h$ by
	\[
	(x,y)\cdot v=x\cdot v+\mathbf{j} y\cdot v.
	\]
	\item $\overline{\kappa}_!$ is multiplication by the R-module $\CCl_{8k+4}$ over $\Cl(\RR^{8k+4,8k+4})$, where $(x,y)\in \RR^{8k+4, 8k+4}$ acts on $\CCl_{8k+4}$ by the above formula except with $\mathbf{j}$ replaced by $\mathbf{i}$. Note $\overline{\kappa}_!$ is the $(1,1)$-periodicity isomorphism.
	\item $\overline{q}_!$ is the (1,1)-periodicity isomorphism given by multiplication by the R-module $\CCl_{n+8k+4}$ over $\Cl(\RR^{n+8k+4,n+8k+4})$.
	\item $(\iota_N)_!$ is multiplication by the $\ZZ_2$-graded tensor product (over $\CC$) of the R-module $\CCl_n$ over $\Cl(\RR^{n,n})$ and the Q-module $\widetilde{\Delta}_{8k+4}^h$ over $\Cl(\RR^{0,8k+4})$, where $x\in \RR^{0,8k+4}$ acts on $v\in \widetilde{\Delta}_{8k+4}^h$ by
	\[
	x\cdot v=v\mathbf{j}\cdot x.
	\]
\end{enumerate}

The commutativity $\overline{\kappa}_!(\iota_M)_!=(\iota_N)_!\kappa_!$ follows from the isomorphism of $\ZZ_2$-graded Q-modules over $\Cl(\RR^{n+8k+4, n+8k+4})$:
\[
\Cl_n^h\hat{\otimes}_\CC\CCl_{8k+4}\cong (\Delta_{8k+4}^h\otimes_\RR\CC)\hat{\otimes}_\CC\CCl_n\hat{\otimes}_\CC \widetilde{\Delta}_{8k+4}^h,
\]
which itself is a consequence of the isomorphism
\[
\Cl_{8k+4}^h\cong \Delta_{8k+4}^h\hat{\otimes}_\RR\widetilde{\Delta}_{8k+4}^h
\]
from \Cref{bimodule}. The commutativity $(\iota_N)_!q_!=\overline{q}_!$ follows from the isomorphism of $\ZZ_2$-graded R-modules over $\Cl(\RR^{8k+4,8k+4})$:
\[
\CCl_{8k+4}\cong \Delta_{8k+4}^h\hat{\otimes}_\CC\widetilde{\Delta}_{8k+4}^h.
\]
Indeed $\Delta_{8k+4}^h$, when viewed as a \textit{complex} module over $\Cl_{8k+4}$, is irreducible and therefore the argument of \Cref{bimodule} proves the above isomorphism holds as $\ZZ_2$-graded complex modules over $\Cl(\RR^{8k+4,8k+4})$. Then one can verify the isomorphism is further a R-module isomorphism. Note for $k=0$ this isomorphism is $\HH^2\otimes_\CC \HH^2=\CC(4)$.

Finally let us argue the local version of the claim globalizes to the desired commutative diagram.
\begin{enumerate}[label=(\roman*)]
	\item The symbol class of $\slashed{\mathscr{D}}^0$ is $[\pi_M^*\slashed{\mathfrak{S}}(M)^0, \pi_M^*\slashed{\mathfrak{S}}(M)^1; v+\mathbf{i}\tau]$, which locally is $\Cl_n^h$ by Atiyah-Bott-Shapiro isomorphism. So the local $(\iota_M)_!$ globalizes to multiplication by the symbol class of $\slashed{\mathscr{D}}^0$.
	\item The local $\kappa_!$ and $\overline{\kappa}_!$ clearly globalize to multiplication by the corresponding Thom classes.
	\item We \textit{define} the global $(\iota_N)_!$ as follows. Notice the structure group of the bundle $TN|_M=TM\oplus N$ can be reduced to $\Spin(n)\times\Spin(8k+4)/\ZZ_2$ since $w_2(TM)=w_2(N)$. Therefore we can associate to it a $\ZZ_2$-graded $Q$-bundle using the representation $\CCl_n\hat{\otimes}_\CC \widetilde{\Delta}_{8k+4}^h$. The degree zero and one components of this bundle, when pulled back all the way to $\RR^n\times TN$, become isomorphic away from $\{0\}\times N$ through Clifford multiplication (from $\RR^n$ and $TN$) and therefore defines an element in $\KQ_{cpt}(\RR^n\times TN)$. We set the global $(\iota_N)_!$ to be multiplication by this element. This clearly globalizes the local $(\iota_N)_!$.
	\item The map $(\iota_\RR)_!$ is similarly defined using that the structure group of the bundle $T\RR^{n+8k+4}$ over $\RR^{n+8k+4}$ can be reduced to $\Spin(n)\times \Spin(8k+4)/\ZZ_2$. The commutativity $(\iota_\RR)_!\lambda_!=\overline{\lambda}_!(\iota_N)_!$ follows from that $TN$ over $N$ is isomorphic to the pull-back of $TN|_M$ along the bundle projection $N\to M$ since the projection is a homotopy equivalence.
\end{enumerate}
Therefore the local diagram does globalize and has the desired properties. The proof of \Cref{spinhindexthm} is now complete.
\appendix
\section{Spin$^h$ cobordism}
In this appendix, we draw some easy conclusions concerning spin$^h$ cobordism. As of the time of writing, the spin$^h$ cobordism groups are completely determined by Keith Mills \cite{Keith}.
\begin{proposition}
	$\hat{\mathcal{A}}^h:\Omega_*^{\spin^h}(\pt)\to\KSp^{-*}(\pt)$ is surjective. In particular $\Omega_{n}^{\spin^h}\neq 0$ for $n\equiv 5,6 \bmod 8$.
\end{proposition}
\begin{proof}
	Since $\hat{\mathcal{A}}^h$ is equivariant with respect to the surjective ring homomorphism $\hat{\mathcal{A}}:\Omega_*^{\spin}(\pt)\to\KO^{-*}(\pt)$, and since $\KSp^{-*-4}(\pt)$ is a free $\KO^{-*}(\pt)$-module generated by $\KSp^{-4}(\pt)$, it suffices to show $\hat{\mathcal{A}}^h$ is surjective in degrees $0$ and $4$. But clearly $\hat{\mathcal{A}}^h(\pt)=\hat{A}^h(\pt)/2=1$ and $\hat{\mathcal{A}}^h(\HH\PP^1_+)=\hat{A}^h(\HH\PP^1_+)=1$.
\end{proof}
\begin{remark}
	With $2$ inverted, $\Spin^h\simeq \Spin\times \Sp(1)$ and consequently $$\Omega_*^{\spin^h}(\pt)[\frac{1}{2}]\cong \Omega_*^{\spin}(\HH\PP^\infty)[\frac{1}{2}]\cong\Omega_*^{\spin}(\pt)\otimes_\ZZ H_*(\HH\PP^\infty;\ZZ[\frac{1}{2}]).$$ This implies $\Omega_n^{\spin^h}$ is a 2-primary torsion group for $n\equiv 5,6$ mod $8$.
\end{remark}

\begin{proposition} Let $F:\Omega_*^{\spin^h}(\pt)\to\Omega_*^{\SO}(\pt)$ be the forgetful homomorphism. Then
	$$(F, \hat{\mathcal{A}}^h):\Omega_n^{\spin^h}(\pt)\to \Omega_n^{\SO}(\pt)\oplus \KSp^{-n}(\pt)$$
	is an isomorphism for $n\le 5$.
\end{proposition}
\begin{proof}[Sketch of proof]
	The surjectivity is clear since $\hat{\mathcal{A}}^h$ is surjective by the previous proposition and $F$ is also surjective: one can enrich oriented manifolds of dimensions $\le 5$ with spin$^h$ structures (see \Cref{existence}). Meanwhile a formidable computation of the spin$^h$ cobordism groups in low dimensions shows in dimensions $\le 5$ the spin$^h$ cobordism groups are abstractly isomorphic to
	$$\ZZ, 0, 0, 0, \ZZ+\ZZ, \ZZ_2+\ZZ_2.$$
	Details will not be given. These groups are also abstractly isomorphic to $\Omega_*^{\SO}(\pt)\oplus \KSp^{-*}(\pt)$ in dimensions $\le 5$. Consequently surjectivity forces isomorphism.
\end{proof}

It is now easy to see $\Omega_4^{\spin^h}(\pt)$ is generated by $\HH\PP^1_+$ and $\CC\PP^2_+$, since $\CC\PP^2$ generates $\Omega_4^{\SO}$ and $\HH\PP^1$ is zero in $\Omega_4^{\SO}$ but $\hat{\mathcal{A}}^h(\HH\PP^1_+)=1$. Similarly $\Omega_5^{\spin^h}(\pt)$ is generated by $\RR\PP^1\times\HH\PP^1_+$ and $\mathrm{SU}(3)/\SO(3)$. Here $\RR\PP^1$ is viewed as a spin manifold with its non-trivial (i.e. non-bounding) spin structure and $\mathrm{SU}(3)/\SO(3)$ carries a natural spin$^h$ structure whose canonical bundle is the natural principal $\SO(3)$-bundle $\SO(3)\to\mathrm{SU}(3)\to \mathrm{SU}(3)/\SO(3)$.

\begin{remark}
	In fact, using standard notations for homotopy theorists, with the knowledge of the cohomology of $\BSpin^h$ calculated in \Cref{mod2cohomology}, one can show the spectrum map $$\underline{\MSpin}^h\to \underline{\mathrm{ksp}} \vee \Sigma^4 \underline{H\ZZ} \vee \Sigma^5 \underline{H\ZZ_2}$$
	labeled by $\hat{\mathcal{A}}^h$, $p_1U$ and $w_2 w_3 U$ induces an isomorphism on 2-local cohomology up to degree $6$. In degree $7$, the induced map on mod 2 cohomology is epic with a one-dimensional kernel reflecting the relation $Sq^3(w_2^2 U)=Sq^2(w_2 w_3 U)=w_2^2 w_3 U$. It follows that the above spectrum map lifts to $\underline{\MSpin}^h\to \underline{\mathrm{ksp}}\vee \underline{F}$ where $\Sigma^{-4}\underline{F}$ is the fiber of $\underline{H\ZZ}\vee \Sigma \underline{H\ZZ_2}\to \Sigma^3 \underline{H\ZZ_2}$ labeled by $Sq^3, Sq^2$. This lifted map is an isomorphism on 2-local cohomology up to degree $7$, hence $\Omega_6^{\spin^h}\cong\ZZ_2+\ZZ_2$.
\end{remark}

\section{KM-theory}\label{appendixB}
\begin{definition}
	Let $(X,f)$ be a real space. An M-bundle over $X$ is a pair $(E,j)$ consisting of a complex vector bundle $E$ over $X$ together with a real bundle map $j:E\to E$ covering $f$ so that $j: E_x\to E_{fx}$ is $\CC$-antilinear and $j^4\equiv 1$. We say $j$ is the M-structure on $E$. In the special case $X$ is a point with trivial involution, we say $(E,j)$ is an M-vector space.
\end{definition}

It is clear both Real bundles and Quaternionic bundles are M-bundles. It may be helpful to think of the Real theory is associated to the group $\ZZ_2$ while the M-theory is associated to the group $\ZZ_4$. The group $\ZZ_4$ admits a natural even-odd filtration where the even subgroup is isomorphic to $\ZZ_2$. Even though the sequence $0\to \ZZ_2\to \ZZ_4\to \ZZ_2\to 0$ does not split, our KM-theory does. Indeed with the assumption that $X$ is connected, every M-bundle is a direct sum of a Real one and a Quaternionic one.
\begin{proposition}
	Let $(E,j)$ be an M-bundle over the \textit{connected} real space $(X,f)$. Then there is a natural M-bundle isomorphism $$E\cong (1+j^2)E\oplus (1-j^2)E$$
	where $(1+j^2) E$, endowed with $j$, is a Real bundle and $(1-j^2)E$ Quaternionic.
\end{proposition}
\begin{proof}
	Notice that $j^2:E\to E$ is a complex linear automorphism of $E$. Since $j^4\equiv 1$, at $x\in X$, $j^2$ decomposes $E_x$ into a direct sum of eigenspaces
	\[
	\ker (1-j^2_x)\oplus \ker (1+j^2_x).
	\]
	The continuity of $j^2_x$ with respect to $x$ implies the dimensions of $\ker (1\mp j^2_x)$ are upper semi-continuous with respect to $x$, whence the sum of the dimensions of $\ker (1\mp j^2_x)$ is a constant. As such both the dimensions of $\ker(1\mp j^2_x)$ are locally constant in $x$. Since $X$ is now assumed to be connected, we conclude $\ker (1\mp j^2)=(1\pm j^2)E$ define complex vector bundles over $X$. It is easy to see when equipped with $j$ these two bundles are Real and Quaternionic respectively. The asserted M-bundle isomorphism follows at once.
\end{proof}

So $\KM=\KR\oplus\KQ$ can be viewed as the Grothendieck group of M-bundles. When dealing with Quaternionic bundles, it is better to think of them as M-bundles, since the theory $\KM$ is multiplicative while the theory $\KQ$ is not. The multiplication on $\KM$ is of course induced by tensor product of complex vector bundles. A special feature for this product is that the product of two Quaternionic bundles is Real. That said, we see the multiplication in $\KM$-theory respects its $\ZZ_2$-grading; in particular $\KR$ is a subring of $\KM$ and $\KQ$ is a $\KR$-module

Most of the results in \cite{KR} for Real bundles and $\KR$-theory now hold for M-bundles and $\KM$-theory, it suffices to replace the Real structures therein by the M-structures. In particular, adopting the notation of \cite{KR}, we have the following projective bundle formula:
\begin{proposition}
	Let $L$ be a Real line-bundle (i.e. of complex rank one) over the real compact space $X$, $H$ is the standard Real line-bundle over the projective bundle $\mathbb{P}(L\oplus 1)$ where $1$ is understood to be the trivialized Real bundle over $X$. Then as a $\KM(X)$-algebra, $\KM(\mathbb{P}(L\oplus 1))$ is generated by $H$ subject to the single relation
	\[
	([H]-[1])([L][H]-1)=0.
	\]
\end{proposition}
%\begin{proof}
%	The relation actually holds in $\KR(\PP(L\oplus 1))$ as proved by Atiyah, thus the mapping $t\mapsto [H]$ induces a $\KM(X)$-algebra homomorphism
%	\[
%	\mu: \KM(X)[t]/(t-1)([L]t-1)\to \KM(\PP(L\oplus 1))
%	\]
%	which respects the $\ZZ_2$-grading.
%\end{proof}
The Thom isomorphism for Real bundles and (1,1)-periodicity follow in a quite formal way.
\begin{theorem}\label{B3}
	Let $E$ be a Real vector bundle over the real compact space $X$. Then
	\[
	\phi: \KM(X)\to \KM_{cpt}(E)
	\]
	is an isomorphism where $\phi(x)=\lambda_E\cdot x$ and $\lambda_E$ is the element of $\KR_{cpt}(E)$ defined by the exterior algebra of $E$.
\end{theorem}
\begin{theorem}
	Let $b=[H]-1\in \KR^{1,1}(\pt)=\KR(\CC\PP^1)$. Then the homomorphism $$\beta: \KM^{r,s}(X,Y)\to \KM^{r+1,s+1}(X,Y)$$ given by $x\mapsto bx$ is an isomorphism.
\end{theorem}

Since the homomorphisms $\phi$ and $\beta$ are both induced by multiplication with Real bundles, they preserve the $\ZZ_2$-grading $\KM=\KR\oplus\KQ$, i.e. they send $\KR$ to $\KR$ and $\KQ$ to $\KQ$. So the corresponding theorems hold for $\KQ$-theory as well. This justifies \Cref{Atiyah-Dupont}.

Recall we have defined $\KM^{r,s}$ for $r,s\ge 0$ using
\[
\KM^{r,s}(X,Y)=\KM(X\times D^{r,s}, X\times S^{r,s}\cup Y\times D^{r,s}),
\]
which in the special case $s=0$ coincides with the usual suspension groups $\KM^{-r}$. Now thanks to the (1,1)-periodicity, we can define $\KM$-groups with positive indices by putting $\KM^r=\KM^{0,r}$. Then we have a natural isomorphism $\KM^{r,s}\cong \KM^{s-r}$. This justifies the use of the group $\KM^4$ in \cite{KQ}.

Now we can quote \cite{KQ} to prove
\begin{proposition}
	For $r,s\ge 0$, multiplication with the generator of $\KQ^{4,0}(\pt)$ yields an isomorphism
	\[
	\KR^{r,s}(\pt)\xrightarrow{\cong}\KQ^{r+4,s}(\pt).
	\]
\end{proposition}
\begin{proof}
	From \cite[(6)]{KQ}, we know multiplication with the generator of $\KQ^4(\pt)\cong\KQ^{0,4}(\pt)$ gives an isomorphism
	\[
	\KQ^{r+4,s}(\pt)\cong \KR^{r+4,s+4}(\pt).
	\]
	On the other hand, the (1,1)-periodicity gives
	\[
	\KR^{r+4,s+4}(\pt)\cong \KR^{r,s}(\pt).
	\]
	Combining the two isomorphisms and summing over all $r,s\ge 0$, we obtain isomorphisms of bigraded-groups
	\[
	\KR^{*,*}(\pt)\cong \KR^{*+4,*+4}(\pt)\cong\KQ^{*+4,*}(\pt).
	\]
	Now observe the above isomorphisms are homomorphisms of $\KR^{*,*}(\pt)$-modules, the proposition thus follows.
\end{proof}

\bibliographystyle{alpha}
\bibliography{ref}
\end{spacing}
\end{document}